\documentclass[11pt]{amsart}
\usepackage{fullpage}
\usepackage[utf8]{inputenc}
\usepackage{microtype}
\usepackage{graphics}
\usepackage{amsmath,amssymb,stmaryrd,array,multirow,dsfont,mathrsfs}
\usepackage{amsthm}
\usepackage{enumerate}
\usepackage{xcolor}
\usepackage{color}
\usepackage{hyperref}
\usepackage{indentfirst}
\usepackage{pifont}
\usepackage[T1]{fontenc}
\usepackage{microtype}
\usepackage{graphics}
\usepackage{amsthm}
\usepackage{xcolor}
\usepackage{color}
\usepackage{hyperref}
\usepackage{stmaryrd}
\graphicspath{{Images/}}
\xdefinecolor{darkgreen}{rgb}{0,0.4,0}
\newcommand{\normmm}[1]{{\left\vert\kern-0.25ex\left\vert\kern-0.25ex\left\vert #1 
   \right\vert\kern-0.25ex\right\vert\kern-0.25ex\right\vert}}
\newcommand{\beq}{\begin{equation}}
\newcommand{\eeq}{\end{equation}}
\newcommand{\beqs}{\begin{equation*}}
\newcommand{\eeqs}{\end{equation*}}
\newcommand{\ben}{\begin{eqnarray}}
\newcommand{\een}{\end{eqnarray}}
\newcommand{\beno}{\begin{eqnarray*}}
\newcommand{\eeno}{\end{eqnarray*}}

\renewcommand{\div}{{\rm div}\,}

\newcommand{\Supp}{{\rm Supp}\,}

\newcommand{\Rmnum}[1]{\uppercase\expandafter{\romannumeral #1} }
 \numberwithin{equation}{section}

\newtheorem{thm}{Theorem}[section]
\newtheorem{lem}[thm]{Lemma}
\newtheorem{prop}[thm]{Proposition}
\newtheorem{rmk}[thm]{Remark}
\newtheorem{cor}[thm]{Corollary}

\newtheorem{definition}{Definition}
\def \La {\Lambda\big(\f{1}{c_0},\cA_{m,t}\big)}
\def\curl{\mathop{\rm curl}\nolimits}

\def\il{|\!|\!|}
\def \d {\mathrm {d}}
\def\cA{{\mathcal A}}

\def\cE{{\mathcal E}}

\def\cH{{\mathcal H}}

\def\cL{{\mathcal L}}

\def\cN{{\mathcal N}}
\def\cO{{\mathcal O}}

\def\cT{{\mathcal T}}

\def\cW{{\mathcal W}}

\def\cZ{{\mathcal Z}}

\let\f=\frac
\def \p {\partial}
\def\mR {\mathbb{R}}

\def\ep{\varepsilon}

\def \pt {\partial_{t}}

\def\si {\sigma}
\def\na{\nabla}
\def\izt{\int_0^t}
\def \lesim {\lesssim}

\def \bn {\textbf{n}}
\def  \bN {\textbf{N}}

\title{Uniform regularity 
for the  compressible Navier-Stokes system  with low Mach number in bounded domains.}
\author{Nader Masmoudi, Fr\'ed\'eric Rousset, Changzhen Sun*}
\address{Department of Mathematics, New York University Abu Dhabi, Saadiyat Island, P.O. Box
129188, Abu Dhabi, United Arab Emirates; Courant Institute of Mathematical Sciences, New York University, 251 Mercer Street, New York, NY 10012, USA.}
\email{masmoudi@cims.nyu.edu}
\address{Universit\'e Paris-Saclay,  CNRS, Laboratoire de Math\'ematiques d'Orsay (UMR 8628),  91405 Orsay Cedex, France}
\email{frederic.rousset@universite-paris-saclay.fr, changzhen.sun@universite-paris-saclay.fr* }
\begin{document}
\maketitle
\begin{abstract}
We establish  uniform with respect to the Mach number   regularity estimates for the isentropic compressible Navier-Stokes system  in smooth domains with  Navier-slip  condition on the boundary in the general case of ill-prepared initial data. To match the boundary layer effects
 due to the fast oscillations and the ill-prepared initial data assumption, we prove  uniform estimates in an anisotropic functional framework
  with only one normal derivative close to the boundary. This allows  to prove the local existence of a strong solution  on  a time interval   independent of the Mach number and to justify the incompressible limit  through a simple compactness argument.
\end{abstract}
%\begin{keywords}

\section{Introduction}
 In this paper, we consider the following scaled isentropic compressible Navier-Stokes system $(CNS)_{\varepsilon}$ 
\beq \label{NSG}
 \left\{
\begin{array}{l}
 \displaystyle\pt \rho^{\varepsilon} +\div( \rho^{\varepsilon} u^\varepsilon)=0,\\
 \displaystyle\pt ( \rho^{\varepsilon} u^{\varepsilon})+\div(\rho^{\varepsilon}u^{\varepsilon}\otimes u^{\varepsilon} )-\div\mathcal{L}u^{\varepsilon}+
\f{\nabla P(\rho^{\varepsilon})}{\ep^2}=0,  \qquad \text{$(t,x)\in \mathbb{R}_{+}\times \Omega $} \\
 \displaystyle u^{\ep}|_{t=0} =u_0^{\varepsilon} ,\rho|_{t=0}=\rho_0^{\varepsilon},\\
\end{array}
\right.\\
\eeq
where $\Omega\subset \mathbb{R}^3$
is a smooth  bounded domain, $\rho^{\ep}(t,x)$ and $u^{\ep}(t,x)$ are  the density and the velocity of the fluid respectively, $P(\rho)$ is the  pressure 
which is  a given smooth function of the density that satisfies
$\f{\d P}{\d \rho}>0,$ for $\rho>0$. 
The viscous stress tensor takes the form:
$$\mathcal{L}u^{\varepsilon}=2\mu\mathbb{S}u^{\varepsilon}+\lambda \div u^{\varepsilon} \text{Id},\quad \mathbb{S} u^{\varepsilon}=\frac{1}{2}(\nabla u^{\varepsilon}+\nabla^{t}u^{\varepsilon}).$$ 
Here,  
$\mu,\lambda$ are viscosity parameters that are assumed to be constant and to satisfy the condition: $\mu>0, 2\mu+3\lambda>0.$ The parameter $\ep$ is the scaled Mach number which  is  assumed  small, that is $\ep\in(0,1]$. 
%Involving one term that depends on the small parameter $\varepsilon,$
% the system \eqref{NSG} can be viewed as a result of a suitable scaling of the original physical variables.

Since we are considering the system  in a domain with boundaries, 
we shall supplement the system \eqref{NSG}
with the Navier-slip boundary condition
\begin{equation}\label{bdyconditions}
u^{\ep}\cdot \bn=0, \quad
 \Pi(\mathbb{S} u^{\ep} \bn)+a\Pi u^{\ep}=0\quad
  \text{on }  {\partial{\Omega}}
\end{equation}
where $\bn$ is the unit outward normal vector and  $a$ is a constant related to a slip length
(our analysis also holds if $a$ is a smooth function).  We use the notation $\Pi f$ for  the tangential part of a  vector $f,$
$\Pi f^{\ep}=f^{\ep}-(f^{\ep}\cdot \bn)\cdot \bn.$

The aim of this paper is to study the uniform regularity (with respect to $\ep$) and the low Mach number limit of system
\eqref{NSG}.
Formally, due to the stiff term  $\frac{\nabla P(\rho^{\varepsilon})}{\varepsilon^2},$ the pressure (and hence the density  $\rho^{\varepsilon}$) is expected to tend to a constant state.  One thus expects to obtain in the limit a solution  to the following incompressible Navier-Stokes system:
  \beq \label{INS}
 \left\{
\begin{array}{l}
 \bar{\rho}(\displaystyle\pt u^0+\div(u^{0}\otimes u^{0} ))-\Delta u^{0}+
\nabla \pi=0,  \\
 \displaystyle \div u^0=0, \qquad\qquad\qquad\qquad\qquad \text{$(t,x)\in \mathbb{R}_{+}\times \Omega $} \\
 \displaystyle u^0|_{t=0} =u_0^{0},\\
\displaystyle
u^0\cdot\bn=0, \quad
 \Pi(\mathbb{S} u^{0} \bn)
 +a\Pi u^{0}=0\quad
  (t,x)\in \mR_{+}\times {\partial{\Omega}}.
\end{array}
\right.
\eeq
%That is the reason that 
This limit process is therefore frequently referred to as the incompressible limit.

The rigorous  justification of this limit process has been studied extensively 
in  different contexts depending on the generality of the system (isentropic or non-isentropic), the type of the system (Navier-Stokes or Euler), the type of solutions (strong solutions or weak solutions), the properties of the domain  (whole space, torus or bounded domain with various boundary conditions), as well as the type of the initial data considered (well-prepared or ill-prepared). 
Roughly speaking, in the case of the compressible Euler system, one proves first that the local strong solution exists
on an interval of time independent of the Mach number, and then  compactness arguments are developed to pass to the limit.
In the case of the compressible Navier-Stokes system, one can either try to use the same approach as
for the inviscid case (prove the existence of a strong solution on an interval of time independent of the Mach number and
then try to pass to the limit) or try to pass to the limit directly  
 from global weak solutions.  Both approaches have been used in domains without boundaries (whole space or torus), 
 nevertheless when a boundary is present the question of uniform regularity for general data is more subtle, as we shall
 see below,  
 and has not been addressed.
 
More precisely, the mathematical justification  of the  low Mach number limit was initiated by Ebin \cite{MR431261}, Klainerman-Majda \cite{MR615627,MR668409} for \textit{local strong solutions} of compressible fluids (Navier-Stokes or Euler), in the whole space with  well-prepared data ($\div u_0^{\varepsilon}=\cO(\varepsilon), \nabla P_0^{\varepsilon}=\cO(\varepsilon^2)$) and later, by Ukai \cite{MR849223} for ill-prepared data ($\div u_0^{\varepsilon}=\cO(1),\nabla P_0^{\varepsilon}=\cO(\varepsilon)$). 
In the latter case, there are acoustic waves of amplitude $1$ and frequency $\ep^{-1}$ in the system.
These works were extended by several authors in different settings. For instance, 
one can refer to \cite{MR2211706,MR1917042,Mr1834114,MR1946548} 
 for  the non-isentropic  system %in various physical domain (whole space, torus)
 and  ill-prepared initial data whenever the domain is the whole space or the  torus, and also
\cite{MR2812710,MR834481}
 for  bounded domains with well-prepared initial data.  Uniform (in Mach number) regularity  estimates for the non-isentropic Euler equations in a bounded domain are established in \cite{MR2106119}. The low Mach number limit of \textit{weak solutions} for the viscous fluid system \eqref{NSG} was studied by  Lions and the first author \cite{MR1628173},  \cite{MR1710123}  where the convergence of the global weak solutions of the isentropic Navier-Stokes system towards a  solution of the incompressible system is established. The result holds for ill-prepared initial data and several different domains (whole space, torus and bounded domain with suitable boundary conditions).  In general, for ill-prepared data, one can only obtain weak convergence in time, nevertheless, 
by using the dispersion of acoustic waves in the whole space,  Desjardins and Grenier \cite{MR1702718} could get  local strong convergence.
There are also many other related works, one can see for example \cite{MR1308856,MR3803773,MR1886005,MR3563240,MR2575476,MR2338352,MR4011035,MR918838,MR4134150}. For more exhaustive information, one can
refer for example to the well-written survey papers by 
 Alazard \cite{MR2425022}, Danchin \cite{MR2157145}, Feireisl \cite{MR3916821}, Gallagher \cite{MR2167201},
 Jiang-Masmoudi \cite{MR3916820},
 Schochet \cite{MR3929616}.

 Let us focus now more specifically on the study of the low Mach limit  of  the isentropic compressible
Navier-Stokes $(CNS)_{\ep}$ system in domains with boundaries with \textit{ill-prepared} initial data, which is more related to the interest of the current paper.
As mentioned above, Lions and Masmoudi \cite{MR1628173} studied the convergence of 
\textit{weak solutions} to $(CNS)_{\ep}$ in  bounded domains with Navier-slip boundary condition.  Later on, for low Mach limit in bounded domains with Dirichlet boundary condition, the authors in \cite{MR1697038,MR3281954}
noticed that, under some geometric assumption on the domain,
the acoustic waves are damped in a boundary layer so that local in time  strong convergence ($L_{t,x}^2$) holds.
 Recently, this result is extended by Feireisl et al \cite{MR2425024} and Xiong \cite{MR3810835} to the case of Navier-slip boundary conditions with 
$a$ of the order   $\varepsilon^{-\frac{1}{2}}$. In this case, 
 the boundary layer effect is comparable to  the one in  the Dirichlet case. One can also refer to \cite{MR2601340,MR2575476,MR2832163}
for the justification of convergence
in unbounded domains by using the local  energy decay for the acoustic system.  Without one of the above properties of the domain, strong convergence 
does not hold for ill-prepared data. 

In the current paper,  our aim is to obtain   
 uniform (with respect to $\varepsilon$) high order regularity estimates
for $(CNS)_{\ep}$ in bounded domains with ill-prepared initial data,  in order to get the existence of a local strong solution on a time interval independent of $\ep$. 
  There are only a few papers addressing this issue. In \cite{MR3240080}, the authors establish 
 uniform global (for small data)  $H^2$ estimates  under a (very) \textit{well-prepared} initial data assumption, namely the 
 second time derivative of the velocity needs to be uniformly bounded initially. %at the initial time.
For   ill-prepared initial data, the situation is more subtle and a uniform $H^2$ estimate, even locally in time, 
 cannot be expected. 
 Indeed, at leading order,  after linearization and symmetrization, the system \eqref{NSG} becomes:
 \begin{equation}\label{lineartoy}
     \partial_t U^{\ep} + \frac{1}{\varepsilon}L U^{\ep}-\left(\begin{array}{c}
           0 \\
        \div\mathcal{L}u^{\ep}
     \end{array}\right)
     =0,\qquad L=\left(\begin{array}{cc}
         0 & \div  \\
        \nabla  & 0
     \end{array}\right),\quad U=(\sigma^{\ep},u^{\ep})\in \mathbb{R}\times \mathbb{R}_{+}^3.
 \end{equation}
 Due to the presence of the diffusion term as well as the singular  linear term, a boundary layer correction to  the  highly oscillating acoustic
 waves appear and create unbounded  high order normal derivatives of the velocity. 
 Note that here, we do not start from a small viscosity problem, nevertheless, at the scale $\tau=t/\ep$ of the acoustic
 waves the system \eqref{lineartoy} behaves like a small viscosity perturbation of the acoustic system.
 For example,  in the easiest case where the boundary is flat (for example $\Omega=\mathbb{R}_{+}^3$), 
 we expect the following expansion of the  solutions to \eqref{lineartoy} involving boundary layers
 \beq\label{formalapp}
 \left\{
 \begin{array}{l}
      \sigma^{\ep}(t,x)=\sigma_0^{I}(\f{t}{\ep},t,x)+\ep^{\f{3}{2}}\sigma^B(\f{t}{\ep},t,x,\f{z}{\sqrt{\ep}})+ \cdots,\\%\f{z}{\ep}) \\
       u^{\ep}(t,x)=u_0^I(\f{t}{\ep},t,x)+\sqrt{\ep}\left(
\begin{array}{l}
      u_{1,\tau}^B(\f{t}{\ep},t,x,\f{z}{{\sqrt\ep}})\\
        \qquad\quad 0
\end{array}
\right)+\ep u_{2}^B(\f{t}{\ep},t,x,\f{z}{{\sqrt\ep}}) + \cdots 
 \end{array}
 \right.
\eeq
where $x= (y,z)$, $z>0$, 
  which suggests that $\|u_{\tau}\|_{L_t^2H^1}, \| u_3^{\ep}\|_{L_t^2H^2},\|\sigma^{\ep}\|_{{L_t^2H^3}}$ can be uniformly bounded whereras $\|\pt(\sigma^{\ep},u^{\ep})\|_{L_{t,x}^2}$
  and $\|\partial_{z}^2 u_{\tau}^{\ep}\|_{L_{t,x}^2}$ will blow up as $\ep$ tends to 0.
 
 %Nevertheless, since we are considering ill-prepared initial data, the oscillating term seems unlikely to be uniformly bounded even in $L_{t,x}^2.$ If we multiply the equation by $\varepsilon,$ the equation, to some extend, is similar to (ICNS) with small viscosity, for which boundary layers is known to appear. Therefore, it seems unlikely to prove some uniform estimates within high order Sobolev space  (for example $\|u\|_{L_t^2H^2}$). %Therefore, inspired by the  
 
 In order to get uniform high order estimates, we shall thus need to use a functional framework based on 
 conormal Sobolev spaces that minimize the use of normal derivatives close to the boundary
in the spirit of    \cite{MR2885569}, \cite{MR3590375}.
 Nevertheless, note that here we have to handle simultaneously the fast oscillations in time and a boundary layer effect
 so that the difficulties and the analysis will be  different from the 
  ones  in \cite{MR3485413,MR3419883} where   compressible slightly viscous fluids are considered.
  Indeed, the energy estimates for  conormal derivatives 
  cannot be easily obtained since  for example tangential vector fields do not commute with the singular part of the system, 
 while   in order to include ill-prepared data, it will be impossible  to get uniform estimates for high order  time derivatives
 as it is done in  \cite{MR3485413,MR3419883} in the study of the inviscid limit.
 We shall explain more these two difficulties below after the introduction of the various norms used in this paper.

% The creation of the boundary layer is a common phenomena for the viscous system in the domain with boundary,
% for instance, the viscous fluids at high Reynolds numbers (or small viscosity parameter). In \cite{MR2885569}, to establish the 
% uniform high order energy estimates with respect to the Reynolds number for incompressible Navier-Stokes equations,
% the first two authors
% work in the so-called conormal Sobolev space to overcome the difficulties stemming from the boundary layers (see also \cite{MR3485413,MR3419883} for compressible case). By introducing some weights in the region of boundary layer, the conormal Sobolev spaces are coherent to the boundary layer effects. 
% Therefore, in our current situation,  it is reasonable to work also in  conormal Sobolev space setting.
%%It goes without saying that the computations shall bear some resemblances with that in \cite{MR2885569,MR3485413,MR3419883}.
% However,  the difficulties to establish the uniform energy estimates and the way we organize are very different.
% Indeed, besides the physical boundary layers effects, there shall be other phenomenon in the current context. For instance, 
% to match the fast oscillating term in the equations, the compressible part of the solutions shall depend on the fast oscillating time variable $\f{t}{\ep},$ which preclude the boundedness of any time derivative of the solutions.

 \subsection{ Conormal Sobolev spaces and notations. }
To define the conormal Sobolev norms,  we take a finite set of generators of vector fields 
that are tangent to the boundary of $\Omega$: $Z_{j}(1\leq j\leq M)$. 
Due to the appearance in \eqref{formalapp} of the 'fast scale' variable $\f{t}{\ep},$ it is also necessary to involve the scaled  time  derivative
$Z_0=\varepsilon \partial_t.$ 
%We shall denote by  $I=(\alpha_0,\alpha_1,\cdots \alpha_M)\in \mathbb{N}^{M+1}$ multi-in and multi 'space-time' conormal derivatives:
We set 
$$Z^I=Z_0^{\alpha_0}\cdots Z_{M}^{{\alpha_M}}, \quad I=(\alpha_0,\alpha_1,\cdots \alpha_M)\in \mathbb{N}^{M+1} $$ Note that $Z^I$ contains not only spatial derivatives but also the scaled time  derivative $\ep\pt.$
 We introduce the following Sobolev conormal spaces: for $p=2$ or $+\infty,$
$$L_t^pH_{co}^{m}=\{f\in L^{p}\big([0,t],L^2(\Omega)\big),Z^{I}f\in L^{p}\big([0,t],L^2(\Omega)\big),|I|\leq m\},$$
equipped with the norm:
\beq\label{not1}
\|f\|_{L_t^pH_{co}^m}=\sum_{|I|\leq m}\|Z^I f\|_{L^{p}([0,t],L^2(\Omega))},
\eeq
where $|I|=\alpha_0+\cdots\alpha_M.$
For the space modeled on  $L^{\infty}$, we shall use the following notation for the norm:
\beq
\interleave f\interleave_{m,\infty,t}%=\colon\|f\|_{L_t^{\infty}W_{co}^{m,\infty}}
=\sum_{|I|\leq m}\|Z^I f\|_{L^{\infty}([0,t]\times\Omega)}.
\eeq
Since the number of time derivatives and spatial conormal derivatives need sometimes  to be distinguished,   we shall also use  the notation:
\beq\label{not3}
\|f\|_{L_t^p\mathcal{H}^{j,l}}=\sum_{I=(k,\tilde{I}), k\leq j,|\tilde{I}|\leq l}\|Z^I f\|_{L^{p}([0,t],L^2(\Omega))}
\eeq
and to simplify, we will use  $\mathcal{H}^{j}=\mathcal{H}^{j,0}.$
To measure pointwise regularity at a given time $t$ (in particular also with $t=0$), we shall use the semi-norms
 \beq
 \label{normfixt}\|f(t)\|_{H_{co}^m}=
 \sum_{|I|\leq m}\|(Z^I f)(t)\|_{L^2(\Omega)}, \quad \|f(t)\|_{\mathcal{H}^{j,l}}=\sum_{I=(k,\tilde{I}),k\leq j,|\tilde{I}|\leq l}\|Z^I f(t)\|_{L^2(\Omega)}.
 \end{equation}
Finally, to measure regularity  along the boundary, we use 
\beq\label{bdynorm}
  |f|_{L_t^p\tilde{H}^s(\partial\Omega)}=\sum_{j=0}^{[s]}|(\varepsilon\partial_t)^j f|_{L^p([0, t], {H}^{s-j}(\partial\Omega))}. 
  \eeq

%\textbf{Definition of the spatial conormal derivatives.}
Let us recall, how the vector fields $Z_{j}$, $1 \leq j \leq M$ can be defined. We consider  
  $\Omega\in \mathbb{R}^3$   a smooth domain (the following construction and our results are  actually valid as long as the boundary of $\Omega$ 
can be covered by a finite number of charts), therefore, there exists 
a  covering such that :
 \begin{equation}\label{covering1}
     \Omega\subset\Omega_0\cup_{i=1}^{N} \Omega_i,\quad  \Omega_0\Subset \Omega, \quad \Omega_i\cap \partial \Omega\neq \varnothing, 
 \end{equation}
and  $\Omega_{i}\cap \Omega$ is the graph of  a smooth  function $z=\varphi_i(x_1,x_2)$.

In $\Omega_{0}$, we just take the vector fields $\partial_{k}$, $k=1, \, 2, \, 3$.
 To define  appropriate  vector  fields near the boundary, we use the local coordinates  in each  $\Omega_{i}:$ 
 \begin{equation}\label{local coordinates}
 \begin{aligned}
 \Phi_{i}:\quad &
 (- \delta_{i}, \delta _{i}) \times (0, \epsilon_{i})
 \rightarrow \Omega_{i}\cap\Omega\\
 &\qquad\quad(y,z)^{t}\rightarrow \Phi_{i}(y,z)=(y, \varphi_i(y)+z)^{t}
 \end{aligned}
 \end{equation}
% We denote the smooth extension of the unit outward normal vector by $\bn$ which is defined in $\Omega_i$:
%  $\bn(\Phi_i(y,z))=\frac{1}{|\textbf{N}|}\textbf{N}$
% with $\textbf{N}(\Phi_i(y,z))=(\partial_1\varphi_i(y),\partial_2\varphi_i(y),-1)^{t}.$ Note that $n$ does not depend on the third variable. Define also the projection $\Pi=\text{Id}-\bn\otimes \bn.$ 
% By this definition, the boundary conditions can be reformulated as:
% \begin{equation}\label{bdryconditionofu}
%    u\cdot \bn|_{\partial\Omega}=0 \quad \Pi (\partial_{\bn} u)=\Pi[ -2a u+(D \bn)u],
% \end{equation}
% where $(D\bn) u=\sum_{j=1}^3  D \bn_j u_j.$
 and we define the vector fields (up to some smooth cut-off functions compactly supported in $\Omega_{i}$) as :
 \begin{equation}\label{spatialvectorinlocal}
Z_{k}^{i}=\partial_{y^k}=\partial_k+\partial_k\varphi_i\partial_3, \quad k=1,2\qquad Z_{3}^{i}=\phi(z)%\p_3.%
(\partial_1\varphi_1\partial_1+\partial_2\varphi_1\partial_2-\partial_3),
\end{equation}
  where $\phi(z)=\frac{z}{1+z},$ and  $\partial_k, k=1,2,3$ are   the derivations with respect to the original  coordinates of $\mathbb{R}^3.$ %We remark that the normal derivative $\p_{\bn}=\partial_1\varphi_1\partial_1+\partial_2\varphi_1\partial_2-\partial_3$ can be spanned by $\p_{y^1},\p_{y^2},\p_3.$
  
  We shall denote by $\bn$ the unit outward normal to the the boundary.
  In each  $\Omega_{i}$, we can extend   it to $\Omega_{i}$   by setting
  $$\bn(\Phi_i(y,z))=\frac{1}{|\textbf{N}|}\textbf{N}, \quad \textbf{N}(\Phi_i(y,z))=(\partial_1\varphi_i(y),\partial_2\varphi_i(y),-1)^{t}.$$
 % Note that $\bn$ does not depend on the third variable {\color{red} unclear ?}. 
 In the same way, the projection on vector fields tangent to the boundary, 
   $$\Pi=\text{Id}-\bn\otimes \bn$$
   can be extended in $\Omega_{i}$ by using the extension of $\bn$.

Let us observe that by identity
\beqs
\Pi(\p_{\bn}u)=\Pi((\nabla u)\bn)=2\Pi(\mathbb{S}u )-\Pi ((D u)\bn)
\eeqs
with $[(\nabla u) \bn]_i=\sum_{j=1}^3 \bn_j \p_j u_i, [(D u) \bn]_i=\sum_{j=1}^3  \p_i u_j \bn_j,$ 
the boundary conditions
\eqref{bdyconditions}
 can be reformulated as:
 \begin{equation}\label{bdryconditionofu}
    u\cdot \bn|_{\partial\Omega}=0, \quad \Pi (\partial_{\bn} u)=\Pi[ -2a u+(D \bn)u]
 \end{equation}
 where $[(D\bn) u]_i=\sum_{j=1}^3  \p_i \bn_j u_j.$ 

\subsection{Main results and strategy of the proof}

Let us introduce the new unknown
$$\sigma^{\ep}=\f{P(\rho^{\ep})-P(\bar{\rho})}{\ep},$$
where $\overline{\rho}$ is a positive  constant state, 
we can rewrite  the system \eqref{NSG} into the following form which is more convenient to perform energy estimates:
  \beq \label{NS1}
 \left\{
\begin{array}{l}
 \displaystyle g_1(\ep\si^{\ep})\big(\pt \si^{\varepsilon} +u^{\ep}\cdot \na\si^{\ep}\big)+\f{\div u^\varepsilon}{\ep}=0,\\
 \displaystyle g_2\big(\ep \si^{\ep})(\pt   u^{\varepsilon}+u^{\varepsilon} \cdot \na u^{\varepsilon} \big)-\div \mathcal{L} u^{\varepsilon}+
\f{\nabla \si^{\varepsilon}}{\ep}=0,  \qquad \text{$(t,x)\in \mathbb{R}_{+}\times \Omega $} \\
 \displaystyle u^\ep|_{t=0} =u_0^{\varepsilon} ,\si^{\ep}|_{t=0}=\si_0^{\varepsilon}.\\
 \end{array}
 \right.
\eeq
where the scalar functions $g_{1}$, $g_{2}$ are defined by 
\beq\label{defofg12}
g_2(s)=\rho^\ep
=P^{-1}(\bar{P}+
s), \quad g_1(s)=(\ln g_2)' (s); \quad s>-\bar{P}=-P(\bar{\rho}).
\eeq
In order to establish uniform energy estimates, 
we shall use the following quantity 
\beqs\label{defcalNmt}
\mathcal{N}_{m,T}(\si^{\ep},u^{\ep})=\mathcal{E}_{m, T}(\si^{\ep},u^{\ep})+\mathcal{A}_{m,T}(\si^{\ep},u^{\ep})
\eeqs
where $\mathcal{E}_{m,T}$ contains  $L^2$ (in space)   type quantities
\begin{multline}
\label{defcalEmt}
\mathcal{E}_{m,T}(\si^{\varepsilon},u^{\varepsilon})=\|(\sigma^{\ep},u^{\ep})\|_{L_T^{\infty}\cH^m}+\|\na(\si^{\ep},u^{\ep})\|_{L_T^{\infty}H_{co}^{m-2}\cap L_T^{2}H_{co}^{m-1}} \\
+ \varepsilon \big(\|(\si^{\ep},u^{\ep})\|_{L_{T}^{\infty}H_{co}^{m}}+\|\nabla(\sigma^{\ep},u^{\ep})\|_{L_T^{\infty}H_{co}^{m-1}}+\|\nabla^2 u^{\ep}\|_{L_T^{\infty}H_{co}^{m-2}}\big)+\ep\|\nabla^2 \sigma\|_{L_T^{\infty}L^2}, 
\end{multline}
and $\mathcal{A}_{m,T}$ involves $L^\infty$ (in space and time) type quantities
\begin{multline}
\label{defcalAmt}
 \mathcal{A}_{m,T}(\sigma^{\ep},u^{\ep}) =\il \nabla u^{\ep}\il_{0,\infty,T}+\il(\nabla \sigma^{\ep},\div u^{\ep},\varepsilon^{\frac{1}{2}}\nabla u)\il_{[\frac{m-1}{2}],\infty,T}+\il(\sigma^{\ep},u^{\ep})\il_{[\frac{m+1}{2}],\infty,T}\\ 
+\varepsilon\il\nabla u^{\ep}\il_{[\frac{m+1}{2}],\infty,T} +\varepsilon\il(\sigma^{\ep},u^{\ep})\il_{[\frac{m+3}{2}],\infty,T}.
\end{multline}
Note that the norms involved in the above definitions are defined in \eqref{not1}-\eqref{not3}. See also Remarks \ref{rmkenergy-1}, \ref{rmkenergy-2} and \ref{rmkLinfty} for the comments on the norms appearing in $\cE_{m,T}$ and $\cA_{m,T}$.

Before stating our main result, we introduce the following definition.
\begin{definition}[Compatibility conditions]
We say that $(\sigma_0^{\ep},u_0^{\ep})$ satisfy  the compatibility conditions up to order $m$ if:
\beqs
(\varepsilon\partial_t)^{j}u^{\ep}\big|_{t=0}\cdot \bn
=0, \qquad \Pi \big[\mathbb{S}\big((\varepsilon\partial_t)^{j}u^{\ep}|_{t=0}\big)\bn\big]=-a\Pi \big[(\varepsilon\partial_t)^{j}u|_{t=0}\big] \quad\text{on} \quad \partial\Omega, j=0,1\cdots m-1.
\eeqs
\end{definition}
Note that the restriction of the time derivatives of the solution at the  initial time can be expressed  inductively by using the equations. For example, we have
$$(\varepsilon\partial_t u^{\varepsilon})(0)=\frac{1}{\rho_0^{\varepsilon}}(-\varepsilon u_0^{\ep}\cdot\nabla u_0^{\ep}+\ep\div \mathcal{L}u_{0}^{\ep}-\nabla\sigma_0^{\ep}).$$
We thus define the  admissible space for  initial data as 
\begin{multline*}
Y_{m}=\bigg\{(\sigma^{\ep}_0,u_0^{\ep})\in H^2(\Omega)^4, \quad 
  Y_{m}^{\ep}(\sigma^{\ep}_0,u^{\ep}_0)<+\infty,
 \\ (\sigma_{0}^{\ep},u_0^{\ep}) \text{ satisfy  the compatibility  conditions up to order $m$}%and } \exists (\sigma_0^{\ep,\kappa},u_{0}^{\ep,\kappa}) \in H^{3m} \text{which satisfy  }\\ 
 %&\text{ compatibility condition up to order $m$ and} Y_{m}(\sigma_0^{\ep}-\sigma_0^{\ep,\kappa},u^{\ep}-u^{\ep,\kappa})\stackrel{\kappa\rightarrow 0}\longrightarrow 0 \bigg\}.
\bigg\}
\end{multline*}
where
\begin{equation}\label{initialnorm}
\begin{aligned}
 Y_{m}(\sigma^{\ep}_0,u^{\ep}_0)&=: \ep \|(\sigma^{\ep}_0,u^{\ep}_0)\|_{H^2(\Omega)}+
 \| (\si^{\ep},u^{\ep})(0)\|_{H^m_{co}} +  \| \nabla (\si^{\ep},u^{\ep})(0)\|_{H^{m-1}_{co}}
 \\& \qquad\qquad\qquad\qquad\qquad +\sum_{|I|\leq [\frac{m-1}{2}]}\|Z^{I}(\na\si^{\ep},\na u^{\ep})(0)\|_{L^{\infty}(\Omega)}
 \end{aligned}
\end{equation}
by using our notation \eqref{normfixt}.

The following is our main  uniform regularity result:
\begin{thm}[Uniform estimates]\label{thm1}

Given an integer $m\geq 6$ and a $C^{m+2}$ smooth bounded domain $\Omega$.
Consider a family of  initial data such that  $(\sigma_0^{\ep},u_0^{\ep})\in Y_{m},$
and  $$\sup_{\ep\in(0,1]}Y_{m}(\sigma_0^{\ep},u_0^{\ep})<+\infty,$$
$$-\bar{c}\bar{P}\leq \ep \sigma_0^{\ep}(x)\leq {\bar{P}}/\bar{c}, \quad \forall x\in\Omega, \ep\in (0,1],$$
where $0<\bar{c}<{1}/{4}$ is a fixed constant, $\bar{P}=P(\bar{\rho}$).
There exist $\ep_0\in (0, 1]$ and $T_{0}>0,$ such that, for any $0< \ep\leq \ep_0$, the
system \eqref{NS1},\eqref{bdyconditions} has a unique solution $(\si^{\ep},u^{\ep})$ 
which satisfies:
 \beq\label{epsigmaLinfty}
 -2{\bar{c}}\bar{P}\leq \ep\sigma^{\ep}(t,x)\leq 2\bar{P}/\bar{c},\quad  \forall (t,x)\in [0,T_0]\times\Omega, 
 \eeq
 and 
 \beq\label{bddcN}
 \sup_{\ep \in (0, \ep_{0}]}\mathcal{N}_{m,T_0}(\si^{\ep},u^{\ep})< +\infty.
 \eeq
\end{thm}
%{\color{red} the statement for $\rho$ is a consequence of the one for $\ep \sigma$, it's better to write it in the remarks}

Let us begin with a few comments about the above  assumptions and our result.
%{\color{red} I think it would be interesting to state in the list of consequences  in a more explicit way some estimates that are indeed uniform in $\ep$, Lipschitz etc... See the remark 1.2}
\begin{rmk}
In view of \eqref{epsigmaLinfty}, there exists $c_0\in(0,1],$ such that:
\beqs
 c_0\leq \rho^{\ep}(t,x)=g_2(\ep\sigma)\leq {1}/{c_0}\quad \forall (t,x)\in [0,T_0]\times\Omega
\eeqs
Moreover, as a consequence of \eqref{bddcN},
the following uniform estimates hold:
$$
 \begin{aligned}
 \sup_{\ep \in (0, \ep_{0}]}\big(
 \|(\sigma^{\ep},u^{\ep})\|_{
 L_{T_0}^{\infty}H_{co}^{m-1}\cap L_{T_0}^2H_{co}^{m}}+\|\nabla(\sigma^{\ep},u^{\ep})\|_{L_{T_0}^{\infty}H_{co}^{m-2}\cap L_{T_0}^2H_{co}^{m-1}}
 +\il \nabla (\sigma^{\ep},u^{\ep})\il_{0,\infty,t}
 \big)<+\infty,
 \end{aligned}
$$
in particular, we have a uniform  estimate for  $ \| \nabla (\sigma^{\ep},u^{\ep}) \|_{L^\infty([0, T_{0}] \times \Omega)}.$
\end{rmk}
 \begin{rmk}
 Because of the compatibility conditions, the assumption
 $\sup_{\ep\in(0,1]}Y_{m}(\sigma_0^{\ep},u_0^{\ep})<+\infty$ imposes that  the data are  prepared (in the sense that it may depend on $\ep$) on the boundary.
 Nevertheless, this is  compatible with the fact that  
  $$ (\div u^{\ep},\nabla\sigma^{\ep})=\mathcal{O}(1)$$ in the domain
  and thus ill-prepared data in the usual sense. 
 Indeed,  note that  $Y_{m}$  clearly contains   smooth functions which vanish identically near the boundary. 
  This kind of compatibility conditions also appears in the study of the incompressible limit of the Euler system
  in bounded domains
  \cite{MR2106119}.
 \end{rmk}

\begin{rmk}\label{rmkenergy-1}
The control of the weighted time derivatives  $(\ep\pt)^k$ up to highest order $k=m:$ %in $L_{T}^{\infty}L_x^2$ 
$\|(\sigma^{\varepsilon},u^{\varepsilon})\|_{L_{T}^{\infty}\cH^{m}}$ is available
since time derivation commute with the space derivation. Moreover,  %$\|(\sigma^{\varepsilon},u^{\varepsilon})\|_{L_T^2H_{co}^m\cap L_{T}^{\infty}H_{co}^{m-1}}$ can be bounded by $\mathcal{E}_{m,T}(\si^{\varepsilon},u^{\varepsilon})$.
\beq\label{hiddenes}
\|(\sigma^{\ep}, u^{\ep})\|_{L_T^{\infty}H_{co}^{m-1}\cap L_T^2H_{co}^m}\lesssim \cE_{m,T}(\sigma^{\ep}, u^{\ep}).
\eeq
 In other words, we can control the highest number of  derivatives in the  $L_{t}^2L_x^2$ norm but lose the uniform  control of the highest space conormal derivatives  in $L_t^{\infty}L_x^2$.  This is due to the bad commutation properties of the space conormal derivatives
 with the singular part of the system.
\end{rmk}
\begin{rmk}\label{rmkenergy-2}
The solution constructed in Theorem \ref{thm1} is a strong solution in the sense that
for $\ep >0$ fixed $(\sigma^\ep, u^\ep) \in L^\infty([0, T_{0}], H^1 \times H^2)$, $u^\ep \in L^2([0, T_{0}], H^3)$. 
 Note that we further have a  uniform control of the  $L_t^{\infty}H^{m-1}\cap L_t^2H^{m}$ norms in every compact set  in the interior of the domain. Nevertheless, due to  boundary layer effects (see \eqref{formalapp}), we cannot expect  uniform estimates for higher order normal derivatives near the boundary.
\end{rmk}
%\begin{rmk}
%The case with Dirichlet boundary condition and ill-prepared initial data is harder to deal with due to the stronger boundary layer effects. Nevertheless, one can expect to get some uniform estimates in the analytic setting. \end{rmk}

By combining the previous result with a compactness argument, we get the following convergence result:
\begin{thm}[Convergence]\label{convergence}
Under the assumptions of Theorem \ref{thm1},  let $(\sigma^\ep, u^\ep)$ the solution defined on $[0, T_{0}]$
 given by Theorem \ref{thm1} and assume that $u_{0}^\ep$ converges strongly in $L^2(\Omega)$ to  some $u_{0}^0$
  when $\ep$ tends to zero.
  Then, as $\ep$ tends to zero,  $\rho^\ep$ (defined by \eqref{defofg12}) converges to $\overline \rho$ in $L^{\infty}( [0, T_{0}] \times \Omega)$ and 
 $u^{\varepsilon} $ converges in $L_{w}^2([0,T_0], L^2(\Omega))$ (weak convergence in time) to  $u^0$ 
 such that 
  \beq\label{additional-regularity}
  u^0\in  L_{T_0}^{\infty}\cH^{0,m-1}\cap L_{T_0}^2\cH^{0,m},\quad \nabla u^0\in L_{T_0}^2\cH^{0,m-1}\cap L^{\infty}([0,T_0]\times\Omega).
  \eeq
%  {\color{red} I don't think you get continuity in time at this level of regularity by the compactness argument, OK, $L_{\infty}$ should be fine.}
Moreover,  $u^0$ is  the (unique in this class)  weak  solution to the incompressible Navier-Stokes system with Navier boundary condition  \eqref{INS}.
 \end{thm}
 Note that $L_{T_0}^2\cH^{0,m}$ is defined in \eqref{not3} and involves only  spatial conormal derivatives.
 \begin{rmk}
Due to the absence of uniform estimate for the second order  normal derivatives and thus also  for the strong
trace of the normal derivative, $u^0$  has to be interpreted as the weak solution to \eqref{INS} in the following usual sense: for any $\psi\in C^{\infty}([0,T_0]\times\ \overline{\Omega})$ with
 $\div \psi =0, \psi\cdot\bn|_{\p\Omega}=0,$ the following identity holds: for every  $0<t\leq T_0,$
 \beq\label{incom-EI}
 \begin{aligned}
 &\bar{\rho}\int_{\Omega}(u^0\cdot\psi)(t,\cdot) \ \d x+\mu\iint_{Q_{t}}\nabla u^0\cdot\nabla\psi \ \d x\d s+\bar{\rho}\iint_{Q_t}(u^0\cdot\nabla u^0)\cdot\psi  \ \d x\d s\\
 &=\bar{\rho}\int_{\Omega}(u^0_{0}\cdot\psi)(0,\cdot)\ \d x+\bar{\rho}\iint_{Q_t}u^0\cdot\pt\psi \ \d x\d s+\mu\int_0^t\int_{\p\Omega}\Pi(-2a u^0+(D\bn) u^0)\cdot\psi \ \d S_y \d s.
 \end{aligned}
 \eeq
 where $Q_t=[0,t]\times\Omega$ and $\d S_y$ denotes the surface measure of $\p\Omega.$
 \end{rmk}
%{\color{red} What is the boundary condition for the limit system ?  In which sense is it satisfied ?I added the boundary condition in the system (1.3) and made precise of its meaning and do the corresponding changes in Section6}
\begin{rmk}
The convergence is weak in the  time variable due to the lack of uniform estimate for $\p_t(\sigma^{\ep},u^{\ep}).$ 
 This cannot be improved since 
in our bounded domain setting, there is no large time dispersion effect for the acoustic waves, 
and since  because of our Navier boundary conditions with fixed slip length, there is no damping in the boundary layers
of the acoustic waves.

Note that  when $\ep$ tends to zero, we have convergence of the whole family $u^\ep$ and not only of subsequences
 due to the uniqueness for the limit system at this level of regularity.

%More precisely, this weak convergence comes only from the compressible part of the velocity which is governed by a wave equation with high speed $\f{1}{\ep}.$
%In our bounded domain setting, there is no dispersion effect, 
% one cannot get better without  any geometric condition imposed on the boundary, partially due to the lack of the large time dispersion in bounded domain. %there is in general no  strichartz type estimate for wave equation with Neumann boundary condition available in bounded domain.
%The only interest of this theorem is that the limit solution is in $C([0,T],H_{co}^{m-1})$ which is strong in the interior of the domain.
%we can not expect to use the strichartz estimates as 
%Indeed, The interest of this theorem is that 
\end{rmk}

We shall  now  explain the main difficulties and the main strategies in order to 
 prove Theorem \ref{thm1}.
As already mentioned the main feature of our problem  is the presence of both fast time oscillations
and a boundary layer in space. These two aspects are well-understood when they occur separately, but in 
order to handle them simultaneously some new ideas will be  needed.
%it is worthy to elaborate the strategies used to establish uniform high regularity
%estimates in other similar problems. Namely, the inviscid limit problem in a bounded domain where a boundary layer
%also appears, and the incompressible limit problem for Euler equations which maybe enlightening due to the similar form of equations.  %for Navier-Stokes equation with large Reynolds number (that is, a small scaled parameter $\epsilon$ is added in front of term $\div\mathcal{L}u^{\ep}$ and $\varepsilon$ is dropped in  \eqref{NS1}), and uniform estimatesfor compressible Euler equations ($\div\mathcal{L}u^{\ep}$ is droppedin \eqref{NS1}).
On the one hand, concerning the inviscid limit problem, one controls \cite{MR2885569, MR3485413, MR3419883} the high order tangential derivatives by direct energy estimates, and then uses the vorticity to control the normal derivatives. 
Nevertheless, for the system with low Mach number, even the tangential derivative estimates are not easy to get, since the spatial tangential derivatives do not commute with  $\nabla,\div$,  defined with the standard derivations in $\mR^3$,  and thus create singular
commutators.
 Without this a priori knowledge on the tangential derivatives, the  estimate of  the vorticity cannot be performed as in  
 \cite{MR2885569} \cite{MR3590375} because of the consequent lack of information on its trace on the boundary.
On the other hand, for the compressible Euler system  with low Mach number, uniform high regularity estimates are established
for example  in 
\cite{MR2106119}. One can 
  %overcome this difficulty
get uniform  $H^s (s>{5}/{2})$ estimates by using  first $\ep \partial_{t}$ derivatives and then
recover space derivatives 
by using the equations to estimate the divergence of the velocity and the gradient of the pressure and  a direct energy estimates for the
vorticity which  solves a transport equation with a characteristic vector field.
Here,  in the case of viscous fluids, we face again the fact that  the estimates  of the vorticity are challenging due to the lack of information on its trace   on the boundary at this stage. 

In order to get the missing  information, we shall first  use the Leray projection (the precise  definition \eqref{def-projection} is in Section 3) to split the velocity into a compressible part and an incompressible part: $u^{\ep}=\nabla\Psi^{\ep}+v^{\ep}.$ On the one hand, the compressible part $\nabla\Psi^{\ep}$ of the velocity can be controlled by $\div u^{\ep}$ thanks  to standard elliptic theory
and hence by using the mass conservation equation and the energy estimates for $\ep \partial_{t}$ derivatives.
On the other hand,  the incompressible part $v^{\ep}$ solves, up to the control of non-local commutators,   a convection-diffusion equation without  oscillations,  and thus one can use direct  energy estimates to get a control of 
$\|v^{\ep}\|_{L_t^{\infty}H_{co}^{m-1}}$ and $\|\nabla v^\ep\|_{L^2_{t}H^{m-1}_{co}}$. Note that we cannot estimate the maximal number of derivatives $m$ 
  due to  the lack of structure of the coupling terms involving  the compressible part in the energy estimates.  The key point here is that the diffusion (which on the other hand creates new  difficulties in the control
  of the vorticity) allows to get the estimate of  $\|\nabla v^\ep\|_{L^2_{t}H^{m-1}_{co}}$.
  This is still not enough to close an estimate  since,  because of the time oscillations, we cannot use Sobolev embedding in time
  to control  $\|\nabla v^\ep\|_{L^\infty_{t}H^{m-2}_{co}}$ as it is done in small viscosity problems for compressible fluids
  (see for example \cite{MR3485413}, \cite{MR3419883}).
  Here,  we only have estimates for powers of  $\ep \partial_{t}$ instead of $\partial_{t}$.
 Nevertheless,   with the additional information obtained from $v^\ep$,  we can then  reduce the matter to the study of  $
 \|\omega^{\ep}\times \bn\|_{L^\infty_{t}H^{m-2}_{co}}$ where 
  $\omega^\ep$ is the vorticity, which solves the heat equation with a non-homogeneous  Dirichlet boundary condition
  which can be controlled from the previous estimates. We shall get the estimate 
  by using the Green's function of  the heat equation.
%Indeed,

\textbf{Outline of the proof of Theorem \ref{thm1}. }
 The uniform energy estimates will  be more precisely achieved in the following steps:
 (we shall skip the $\ep$ dependence in the notations  for the sake of simplicity).

\textbf{Step 1: Uniform high-order  $\ep \partial_{t}$ derivatives and  $\ep-$dependent high-order conormal derivatives. }
In this step, we aim to prove two kinds of estimates. Namely, uniform estimates for high order $\ep \partial_{t}$ derivatives, 
$\|(\sigma,u)\|_{L_t^{\infty}\cH^{m}},$ %\|(\nabla\sigma,\div u)\|_{L_t^{\infty}\cH^{m-1}},$
and $\ep-$dependent estimates: $\ep\|(\sigma,u)\|_{L_t^{\infty}H_{co}^{m}}$, $\ep\|(\nabla\sigma,\div u)\|_{L_t^{\infty}H_{co}^{m-1}}.$
On the one hand, since the time  derivative $\ep\pt$ commutes with the spatial derivatives, we can get
uniform estimates for high order time derivatives. 
Note that we use $\ep \partial_{t}$ instead of $\partial_{t}$ since we are dealing with ill-prepared data.
On the other hand, as 
the spatial conormal vector fields do not commute with $\nabla, \div,$ the singular part of the system,
we need at this stage  to add this additional $\ep$ weight to control the commutator.

\textbf{Step 2: Uniform estimates for the incompressible part of the velocity.}
Let us denote by $v=\mathbb{P} u, $  and $ \nabla\Psi=\mathbb{Q} u$ the incompressible and compressible part of the velocity
respectively,
where $\mathbb{P}, \mathbb{Q}$ are defined in \eqref{def-projection}.
By applying  the projection $\mathbb{P}$ on the equation for  the velocity and expanding  the boundary conditions,
 %as well as the definition of $\mathbb{Q},$ 
we find that $v$  solves:
\beq\label{eqofv0}
\left\{
\begin{array}{l}
  \bar{\rho}\partial_t v -  \mu \Delta v+\nabla q+\frac{g_2-\bar{\rho}}{\varepsilon}\varepsilon\partial_t u+g_2u\cdot\nabla u=0 \quad \text{in}\quad
     {\Omega} \\[7pt]
     v\cdot\bn=0, \quad \Pi(\p_{\bn}v)=\Pi(-2a u+D \bn\cdot  \nabla\Psi+D \bn\cdot u)\quad \text{on}\quad
     {\p\Omega}
\end{array}
\right.
\eeq
%{\color{red} you could explain what you do to get  the boundary condition there is not enough space between the two lines of the system}
where
$$\nabla q=-\mathbb{Q}(\frac{g_2-1}{\varepsilon}\varepsilon\partial_t u+g_2u\cdot\nabla u-\mu \Delta v).$$
Note that the first boundary condition $v\cdot\bn=0$ is due to the definition of the projection $\mathbb{P}$ while
the second boundary condition is deduced from %boundary condition of $u: 
$\eqref{bdryconditionofu}.$
The incompressible part $v$ interacts with the  compressible part $\nabla\Psi$ through the source term and the  boundary condition. Due to the absence of  singular  terms, one can get the uniform estimates for $v$ 
(namely $\|v\|_{L_t^{\infty}H_{co}^{m-1}}$ and $\|\nabla v\|_{L_t^2H_{co}^{m-1}}$)
by direct energy estimates. 
Nevertheless, for  latter use in the proof, we need to track in the energy estimates the
counts of time and spatial conormal derivatives.

\textbf{Step 3: Uniform estimates for the compressible part of the system.}
In this step, we aim to get the control of 
$\|(\nabla\sigma,\div u)\|_{L_t^{\infty}H_{co}^{m-2}\cap L_t^{2}H_{co}^{m-1}}.$ This can be done by using the equations and induction arguments. Indeed, by rewriting the system \eqref{NS1}, %can be rewritten as:
\beno
-\div u&=&
g_1\ep\pt \sigma+ \ep g_1 u\cdot\nabla \sigma, \\
-\nabla\sigma&=&g_2\ep\pt u +\ep (g_2u\cdot\nabla u-\div\mathcal{L} u). 
\eeno
In view of the above two equations, one can 'trade' one spatial derivative by one (small scale) time derivative $\ep\pt$.
We can thus recover the high order spatial (conormal) derivatives by using iteratively  this observation. 

 \textbf{Step 4}: Control of 
 $L_t^{\infty}H_{co}^{m-2}$ norm of $\nabla u$.
 In this step, we aim to get an uniform control of $\|\nabla u\|_{L_t^{\infty}H_{co}^{m-2}}$
 which is quite useful to control $L_{t,x}^{\infty}$ type norms.  %$\|\nabla v\|_{L_t^{\infty}H_{co}^{m-2}}.$  
 The difficulty is the estimate close to the boundary. We can work in a local chart $\Omega_{i}$.
In light of the identities 
\beqs
\partial_{\bn} u\cdot \bn =\div u-(\Pi \partial_{y_1}u)^1-(\Pi \partial_{y_2}u)^2,\quad \Pi(\partial_{\bn} u)=\Pi(\omega\times \bn)-\Pi[(D \bn) u],
\eeqs
 where $\bn$ is an extension of the unit normal and $\Pi$ projects on $(\bn)^\perp$, 
it suffices  to control $\|\omega\times\bn\|_{L_t^{\infty}H_{co}^{m-2}}.$ We remark that the advantage of working on $\omega\times\bn$ rather than $\omega$ is that  the boundary 
condition for $\omega\times\bn$ (see \eqref{bdryomegan}) only involves lower order
terms on the boundary.  
 To estimate $\omega\times\bn$, a natural attempt, used in \cite{MR2885569},  is to perform  energy estimates on the equation for the 'modified vorticity' $w=\omega\times \bn+2\Pi (a u-(D\bn) u)$ and to take advantage of the fact that $w$ vanishes on the boundary. However, the equations for $w$ still involve a stiff term $\frac{1}{\varepsilon}\nabla^{\perp}\sigma,$ which is obviously an obstacle to obtain uniform energy estimates. We shall thus instead use a lifting of the boundary conditions by using  Green's function  for the solution of   the heat equation with non-homogenous 
 boundary conditions and estimate the remainder by energy estimates.

\textbf{Step 4: $L_{t,x}^{\infty}$ estimates.} The control of  the $L_{t,x}^{\infty}$ norms contained in 
$\cA_{m,T}$
mainly stems from the Sobolev embedding and the maximum principle for the  system solved
by the vorticity. Note that at this stage, it is crucial to use the direct $L^\infty_{t}H^{m-1}_{co}$ for $(\sigma, u)$
 and $L^\infty_{t}H^{m-2}_{co}$ for $\nabla(\sigma, u)$ estimates 
obtained in the previous steps since because of the fast oscillations in time, uniform $L^\infty$ estimates in time cannot
be deduced from a Sobolev embedding in time.

 The  case  $\Omega=\mR_{+}^3$ where the boundary is flat is easier to analyze. Indeed, the spatial tangential derivatives can be controlled directly through  energy estimates without weight in $\ep$, since in this case 
  the derivatives $\p_{y^{i}}$ commute with $\div$ or $\nabla.$
 The use of the step with the Helmholtz-Leray projection is thus not necessary. The details can be found in the PhD thesis \cite{thesis}.

 In a forthcoming paper \cite{lowmachfree}, we shall strengthen the strategies used in this paper to deal with the low Mach number limit problem for the  free surface compressible Navier-Stokes system, where we are
 forced to deal with strong enough solutions in the absence of a suitable theory of weak solutions.

\textbf{Organization of the paper.}
We will %prove the Theorem \ref{thm1} in section 2, by using the
state the main uniform estimates in Section 2 which will be proven in Section 3 and Section 4. Section 5 is then devoted to the proof of Theorem \ref{thm1}.
In Section 6, we will justify  the incompressible limit. In the appendix, we gather some useful product and commutator estimates %there are also some other lemmas which are needed for the energy estimates 
 as well as the proofs of  some  technical lemmas.

\section{Uniform estimates.}
In this section, we state the main uniform a priori estimate which is the heart of this paper and the crucial step towards the proof of 
Theorem \ref{thm1}: 
\begin{prop}\label{prop-uniform es}
%Supposing that $g_1,g_2$ are smooth functions satisfying
Let $c_0\in (0,1]$ be such that:
\beq\label{preasption1}
\forall s  \in  \big[-3{\bar{c}}\bar{P},  3\bar{P}/\bar{c}\big],  \, c_0\leq g_{i}(s) \leq 1/{c_0}, \, i=1, \, 2,  \quad |(g_1,g_2)|_{C^m\big(\big[-3{\bar{c}}\bar{P},  3\bar{P}/\bar{c}\big]\big)}\leq {1}/{c_0}
\eeq
where $\overline{c}$ is such that  for some $T \in (0, 1]$ the following assumption holds:
\begin{equation}\label{preasption}
-3{\bar{c}}\bar{P}\leq \ep\sigma^{\ep}(t,x)\leq 3\bar{P}/\bar{c} \qquad \forall (t,x)\in [0,T]\times\Omega,  \forall \ep\in[0,1].
\end{equation}
%$c_0,c_1,$ such that the following condition holds:
%\begin{equation}\label{preasption}0\leq c_0\leq g_1(t,x),g_2(t,x)\leq c_1\qquad \forall (t,x)\in [0,T]\times\Omega\end{equation}
%{\color{red} these functions are not defined as functions of $t,x$ The assumption should be on $\sigma$ or $\ep \sigma$ they  are also not defined everywhere}
Then, there exists  $C(1/c_{0})>0$  and 
an  polynomial
$\Lambda_0$ (whose coefficients are independent of $\ep$), 
 such that, for any $\ep \in (0, 1],$ we have for a smooth enough solution of \eqref{NS1}  on  $[0, T]$
 the following estimate :
 \beq\label{enerineq}
 \mathcal{N}_{m,T}^2(\si^{\ep},u^{\ep})\leq C\big(\f{1}{c_0}\big)
 Y^2_m(\sigma_0^{\ep},u_0^{\ep})+
 (T+\ep)^{\f{1}{2}} \Lambda_0\big(\f{1}{c_0},\mathcal{N}_{m,T}(\si^{\ep},u^{\ep})\big),
 \eeq
 where  $ Y_m(\sigma_0^{\ep},u_0^{\ep})$ is defined in \eqref{initialnorm}.
\end{prop}
%{\color{red} \eqref{preasption1} is not an assumption, it's the definition of $c_{0}$ as a consequence of the previous assumption, 
%$c_{0}$ is thus a function of $\overline{c}$ and thus also $C$  you can simplify. Yes, \eqref{preasption1} is just the property of $g_1,g_2$
%but $c_0$ depends on $\bar{c}$ very implicitly. I think it is still better to introduce $c_0$ and state the dependence of $C$ on $c_0.$}
\begin{proof}
This proposition is the consequence of Proposition \ref{energybdd} and \ref{Linftybd}, which will be established in Section 3 and Section 4 respectively.
\end{proof}
\section{Uniform estimates-energy norm}
In this section, we establish the a-priori estimates for the energy norm $\mathcal{E}_{m,T}.$ Again, for notational convenience, we skip the $\ep-$dependence of the solutions.
\begin{prop}\label{energybdd}
If the estimates  \eqref{preasption} \eqref{preasption1} are satisfied, then we can find  %two  $\Lambda_2,\Lambda_3$ 
a constant $C_1(1/c_0)$ that depends only on $1/c_0$ and a polynomial $\tilde{\Lambda}$
whose coefficients are independent of $\ep,$ such that for a smooth enough solution of \eqref{NS1},
 the following estimate holds on $[0, T]$ for $\ep \in (0, 1]$:
 \begin{equation}\label{energybound}
   \mathcal{E}^2_{m,T}\leq C_1\big(\f{1}{c_0}\big) Y^2_{m}(\sigma_0,u_0)+(T+\varepsilon)^{\frac{1}{2}}   \tilde{\Lambda}\big(\f{1}{c_0}, \mathcal{N}_{m,T}\big).
 \end{equation}
\end{prop}

As explained in the introduction, to overcome the difficulty due to the 
 nontrivial commutators between the tangential spatial derivatives and the standard derivation ($\nabla,\div$), we need to split the velocity
 $u$ into $u=\nabla\Psi+v$, where $\nabla\Psi, v$ are the compressible part and the incompressible part respectively (see \eqref{def-projection} the precisely definition). On the one hand, the compressible part $\nabla\Psi$ satisfies the elliptic equation $\Delta\Psi=\div u$ 
 with  Neumann boundary condition, from which 
 one can deduce  the estimate of $\nabla^2\Psi$ from  that of $\div u.$ On the other hand, since the incompressible part $v$ is governed by  a convection diffusion  equation
 without oscillations, we can control  its 
 conormal derivatives by direct energy estimates.  The estimates for  $\partial_{\bn} v$ will then be  deduced from the ones for $\omega\times \bn.$

\subsection{Preliminaries: Leray projection}\label{prebd}
\qquad\qquad\qquad\qquad\qquad\qquad\qquad\qquad\qquad\qquad\qquad\qquad\qquad\qquad\qquad\qquad\qquad\qquad\qquad\qquad

 %Unlike the half space case, for bounded domain, we can not control the high order spatial tangential derivatives $Z_{tan,k}, (k=1,2)$ by direct energy estimates since they   do not commute with $\div$ and $\nabla$. This shall arise new difficulty when estimating %$\nabla u\|_{L_t^{\infty}H_{co}^{m-2}},$or more precisely $\|\partial_n u\|_{L_t^{\infty}H_{co}^{m-2}}.$ Indeed, by the fact that  $$\partial_n u=(\partial_n u\cdot n)n+\Pi(\partial_n u)=\div u+\omega\times n+ \sum \text{spatial tangential derivatives of $u$,}$$
%one sees that in order to control $\partial_n u,$ it is necessary to know the information of $Z_{tan} u,$ which is unavailable in the beginning.
 
 To define the compressible or acoustic  part and the incompressible part of the velocity field, we shall use  the Leray projection.  One has the decomposition, 
 $$L_{x}^2(\Omega)^3=H\oplus G$$
where 
$$H=\{v\in L_{x}^2(\Omega)^3,\div v=0, v\cdot \bn|_{\partial{\Omega}}=0\}, \quad 
G=\{\nabla \Psi, \nabla \Psi\in L^2(\Omega)^3\}.$$
 We denote $\mathbb{P},\mathbb{Q}$ the  projectors that map $L_{x}^2(\Omega)^3$ to its subspaces $H$ and $G$ respectively, 
 namely,  
\begin{equation}\label{def-projection}
   \begin{aligned}
   \mathbb{Q}: &
L^2(\Omega)^3\rightarrow G\qquad\qquad\qquad \mathbb{P}:
L^2(\Omega)^3\rightarrow H\\
&\qquad f \mapsto \mathbb{Q} f=\nabla\Psi \qquad\qquad\qquad \quad
f\mapsto f-\mathbb{Q} f
   \end{aligned} 
\end{equation}
%{\color{red} Why is it $\vr$ and not $\Psi ? $ because I use $\vr$ to denote a general gradient but use $\Psi$ for the compressible part of u.}
where $\Psi$ is defined as the unique
solution of
\begin{equation}\label{Neumann}
\left\{
  \begin{array}{l}
  \displaystyle \Delta\Psi=\div f  \quad \text{in} \quad \Omega, \\
      \displaystyle\partial_{\bn} \Psi
      =f\cdot {\bn} \quad \text{on} \quad \partial\Omega,\\
 \int_{\Omega} \Psi d x=0.
  \end{array}
  \right.
\end{equation}
 Note that the solvability of the Neumann problem \eqref{Neumann} in $H^{1}(\Omega)$ is well-known as an application of the Lax-Milgram theorem. Moreover, by Proposition \eqref{propneumann}, one has that for a $C^{k+1}$ bounded domain,
 \begin{equation}
 \|\nabla\Psi(t)\|_{H_{co}^k}\lesssim \|f(t)\|_{H_{co}^k},
 \qquad \|\nabla^2\Psi(t)\|_{H_{co}^{k-1}}\lesssim \|\div f(t)\|_{H_{co}^{k-1}}+\|f(t)\|_{H_{co}^{k-1}}.
 \end{equation}
Note that in these estimates, the time variable is just an external parameter. 

Since $[\mathbb{P},\partial_{t}]=0,$ \eqref{NS1} is equivalent to the following system:
\begin{equation}\label{reformulation}
\left\{
\begin{array}{l}
%\displaystyle 
g_1(\partial_t\sigma+u\cdot\nabla\sigma)+\frac{\Delta\Psi}{\varepsilon}=0,    \\[5pt]
\bar{\rho}\partial_t\nabla\Psi+\mathbb{Q}\big(\frac{g_2-\bar{\rho}}{\varepsilon}\varepsilon\partial_t u+g_2u\cdot\nabla u-\mu \Delta v-(2\mu+\lambda)\nabla\div u+\frac{\nabla \sigma}{\varepsilon}\big)=0,\\[5pt]
\bar{\rho}\partial_t v+\frac{g_2-\bar{\rho}}{\varepsilon}\varepsilon\partial_t u+g_2u\cdot\nabla u-\mu \Delta v+\nabla q=0,
\end{array}
\right.
\end{equation}
%{\color{red} If you use these system notations, you have to write bigger an put more space between the lines}
where 
$$v=\mathbb{P}u,\quad  \nabla\Psi=\mathbb{Q}u,\quad \nabla q=-\mathbb{Q}\big(\frac{g_2-\bar{\rho}}{\varepsilon}\varepsilon\partial_t u+g_2u\cdot\nabla u-\mu \Delta v\big), \quad \bar{\rho}=g_{2}(0).$$
By taking the divergence of the third equations of \eqref{reformulation} and noting that $\div v=0, \varepsilon\partial_t u\cdot \bn|_{\partial\Omega}=0,$ we see that $\nabla q$ is governed by
 the following elliptic equation:
\begin{equation}\label{eqofq}
\left\{
 \begin{array}{l}
    \displaystyle \Delta q=- \div\big(\frac{g_2-\bar{\rho}}{\varepsilon}\varepsilon\partial_t u+g_2u\cdot\nabla u\big)\quad \text{in} \quad \Omega,\\[6pt]
 \partial_{\bn} q=-(g_2u\cdot\nabla u)\cdot \bn+\mu \Delta v\cdot \bn \quad \text{on} \quad \partial\Omega.
 \end{array}   
 \right.
\end{equation}
%By standard elliptic theory with Neumann boundary condition, we have that:
%\beqs\|\nabla q(t)\|_{L_{x}^2(\Omega)}\lesssim \|\frac{1-g_2}{\varepsilon}\varepsilon \partial_t u+g_2u\cdot \nabla u(t)\|_{L_x^2}+\mu \|\Delta v(t)\cdot n\|_{H^{-\frac{1}{2}}(\partial{\Omega})},\eeqs

%It follows from  Proposition \ref{propneumann} that, for $m\geq 1,$
%\beqs\|\nabla q\|_{L_t^2H_{co}^{m-1}}\lesssim \|\frac{1-g_2}{\varepsilon}\varepsilon \partial_t u+g_2u\cdot \nabla u\|_{L_t^2H_{co}^{m-1}}+\mu \|\Delta v(t)\cdot n\|_{\tilde{H}^{m-\frac{3}{2}}(\partial{\Omega})}.\eeqs

Proposition \ref{energybdd} can be shown by the first three steps
outlined in the introduction, they will be handled in  the following three subsections.

%$\bullet$Step 1. Uniform estimates for highest order temporal derivatives \&$\varepsilon$-dependent estimates for highest-order conormal derivatives,$\bullet$ %\underline{Step 2: Energy estimate for the incompressible part of velocity, $\bullet$ %\underline{Step 3: Recovering spatial conormal derivatives of $(\nabla\sigma,\div u)$ by the equations.
%We comment that$\|\nabla u\|_{1,\infty,T}$ can not be bounded by $\Lambda_{2,\infty},$ since $\|\omega\|_{1,\infty,T}$ is not included in the definition of $\Lambda_{2,\infty,T}.$Indeed, it seems not easy to prove that  $\|\omega\|_{1,\infty,T}$ is bounded uniformly in $\varepsilon,$ since the equation governed by $\omega$ is a transport-diffusion equation, the commutator between $Z$ and the Laplacianis a linear term that involves two order spatial derivative of $\omega$, which  would be an obstacle to close the estimate.

\subsection {Step 1: highest conormal estimates.} 
For  notational convenience, we 
%set \begin{equation} \Lambda_{1,\infty,T}=\Lambda(1/c_0,\il(\sigma,u)\il_{1,\infty,T}+\il\nabla(\sigma,u)\il_{0,\infty,T});\end{equation}\begin{equation}\Lambda_{2,\infty,T}=\Lambda(1/c_0,\il(\nabla\si,\div u)\il_{1,\infty,T}+\il\nabla u\il_{0,\infty,T}+\il(\sigma,u)\il_{2,\infty,T}).\end{equation}
 denote %for some polynomial 
 $\Lambda$ for a polynomial which may differ from line to line, and  use the notation $\lesssim \cdot $ as $\leq C\cdot$ for some generic constant $C=C(1/c_0)$ that depends on $1/c_0$ but 
not on $\ep.$

Let us state the main result of this subsection.
\begin{lem}\label{lemstep1}
Suppose that \eqref{preasption} is satisfied, then for any $m\geq 0,$ any $0 < T\leq 1$ and $\ep \in (0, 1]$ we have: 
\beq\label{sumstep1}
\begin{aligned}
   &\|(\sigma,u)\|_{L_T^{\infty}\cH^m}^2%
   +\ep^2 (\|(\sigma,u)\|_{L_T^{\infty}H_{co}^m}^2+\|(\nabla\sigma,\div u)\|_{L_T^{\infty}H_{co}^{m-1}}^2)\\
&\qquad +\|\nabla u\|_{L_t^2\cH^{m}}^2
+\ep^2(\|\nabla u\|_{L_T^2H_{co}^m}^2+\|\nabla\div u\|_{L_T^{2}H_{co}^{m-1}}^2)\\
&\lesssim Y^2_m(\sigma_0,u_0)+(T+\ep)^{\f{1}{2}}\Lambda\big(\f{1}{c_0},\cA_{m,T}\big)\cE^2_{m,T}.
\end{aligned}
\eeq 
\end{lem}
%{\color{red}why do we need $t$ and $T$ ?}
%{\color{red} what does it mean $\lesssim $ }
\begin{proof}
The estimate  \eqref{sumstep1} can be derived from the following two lemmas.
 \end{proof}
Let us start with:
 \begin{lem}\label{highest}
 %{Temporal derivative for $(\sigma,u).$}
 %Suppose that \eqref{preasption} is satisfied, then for any $m\geq 0,$ any $t\leq T,$ 
 Under the same assumption as in Lemma \ref{lemstep1}, for any $0<t\leq T,$
 the following  estimates hold:
 \begin{equation}\label{hightimeder}
%\sum_{j=0}^m\|(\ep\pt)^j(\sigma,u)(t)\|_{L^2(\Omega)}^2+\|(\ep\pt)^j\nabla u\|_{L_{t,x}^2}^2
\|(\sigma,u)\|_{L_t^{\infty}\cH^{m}}^2+\|\nabla u\|_{L_t^2\cH^{m}}^2
\lesssim \|(\si,u)(0)\|_{\cH^m}^2+
\Lambda\big(\f{1}{c_0},\cA_{m,T}\big)T^{\f{1}{2}}\cE_{m,T}^2,
 \end{equation}
 \begin{equation}\label{highestconormal}
   \varepsilon^2\big(\|(u,\sigma)(t)\|_{H_{co}^m}^2+\|\nabla u\|_{L_t^2H_{co}^{m}}^2\big) \lesssim 
   \varepsilon^2 \|(\sigma, u)(0)\|_{H^m_{co}}^2+\varepsilon^{\f{1}{2}} \Lambda\big(\f{1}{c_0},\cA_{m,T}\big)\cE_{m,T}^2
   +\ep^2\|\nabla\div u\|_{L_t^2H_{co}^{m-1}}^2.
 \end{equation}

 \end{lem}
 We recall that in our notations the norms at $t=0$ involve the computation of powers of $\ep \partial_{t}$
  at $t=0$.
\begin{proof}
%Define $\sigma^j=Z_0^{j}\sigma=(\ep\pt)^{j}\sigma,u^{j}=(\ep\pt)^{j}u.$ Then $(\sigma^j, u^{j})$
%\begin{equation*}
%\left\{
 %   \begin{array}{l}
 % \displaystyle g_1(\pt\sigma^{j}+u\cdot\nabla \sigma^j)+\frac{\div u^j}{\varepsilon}=-[Z_0^j,\frac{g_1}{\varepsilon}]\varepsilon \partial_t\sigma-[Z_0^j,g_1u\cdot\nabla]\sigma=\colon \mathcal{R}_{\sigma}^j.   \\
%\displaystyle g_2(\partial_t u^j+u\cdot\nabla u^j)-2\mu \div(\mathbb{S}u^j)-\lambda\nabla \div u^j+\frac{\nabla\si^j}{\varepsilon}=-[Z_0^j,\frac{g_2}{\varepsilon}]\varepsilon \partial_t u-[Z_0^j,g_2u\cdot\nabla]u=\colon \mathcal{R}_{u}^j. \\\end{array}
%\right.
%\end{equation*}
Define $\sigma^I=Z^I\sigma,u^{I}=Z^I u.$ Then $(\sigma^I, u^{I})$
satisfies:
\begin{equation}\label{eqafterdervtion}
\left\{
    \begin{array}{l}
  \displaystyle g_1(\pt\sigma^{I}+u\cdot\nabla \sigma^I)+\frac{\div u^I}{\varepsilon}= \mathcal{R}_{\sigma}^I,
    \\
\displaystyle g_2(\partial_t u^I+u\cdot\nabla u^I)-Z^I(\div \mathcal{L}u)%\mu \div(\mathbb{S}u^I)-\lambda\nabla \div u^I
+\frac{\nabla\sigma^I}{\varepsilon}= \mathcal{R}_{u}^I,\\
\end{array}
\right.
\end{equation}
where 
\begin{equation*}
\begin{aligned}
    \mathcal{R}_{\sigma}^I&=-[Z^I,\frac{g_1}{\varepsilon}]\varepsilon \partial_t\sigma-[Z^I,g_1u\cdot\nabla]\sigma-\frac{1}{\varepsilon}[Z^I,\div ]u,\\
    \mathcal{R}_{u}^I&=-[Z^I,\frac{g_2}{\varepsilon}]\varepsilon \partial_t u-[Z^I,g_2u\cdot\nabla]u-\frac{1}{\varepsilon}[Z^I,\nabla]\sigma. 
\end{aligned}
\end{equation*}
We first show \eqref{hightimeder} which is easier. Assuming that $I=(j,0,\cdots, 0), $ $|j|\leq m$ which means that $Z^{I}=(\varepsilon\partial_t)^j$  involves only time  derivatives. The advantage of this case is that the commutators do not include singular terms, that is the third terms in $\mathcal{R}_{\sigma}^I$ and $\mathcal{R}_{u}^I$
vanish.

For the sake of notational simplicity, we denote $(\sigma^j,u^j)=(\varepsilon\partial_t)^j(\sigma, u).$
Taking the scalar product of \eqref{eqafterdervtion} by
 $(\sigma^j,u^j)$ and taking benefits of the boundary conditions 
 \beq\label{bdrycon-time}
 u^j\cdot \bn=0,\quad
 %\Pi(\mathbb{S}u^j n)=-2a\Pi u^j
 \Pi(\partial_{\bn} u^j)=\Pi (-2a u^j+(D\bn) u^j) \quad \text{on}\quad \partial\Omega,
 \eeq
 as well as the relation
$\partial_t g_2+\div(g_2 u)=0,$ 
we get from standard integration by parts  that:
\begin{equation}\label{sec5:eq2}
\begin{aligned}
 &\frac{1}{2}\int_{\Omega} (g_1|\sigma^j|^2+g_2|u^j|^2)(t) \ \d x+\iint_{Q_t}\mu|\nabla u^j|^2+(\mu+\lambda)|\div u^j|^2 \ \d x \d s\\
 &\leq \frac{1}{2} \int_{\Omega} \big(g_1|\sigma^j|^2+g_2|u^j|^2\big)(0) \ \d x+ \left|\iint_{Q_t}\big(\partial_t g_1+\div(g_1 u)\big)|\sigma^j|^2 \ \d x\d s\right|\\
 &\quad+\mu \left|\int_0^t \int_{\partial\Omega}\Pi(\partial_{\bn} u^j)\Pi u^j \d S_y \d s\right|+\|\mathcal{R}_{\sigma}^I\|_{L^2(Q_t)}\|\sigma^j\|_{L^2(Q_t)}+\|\mathcal{R}_{u}^I\|_{L^2(Q_t)}\|u^j\|_{L^2(Q_t)},
\end{aligned} 
\end{equation}
 where we denote by  $\d S_y$ the surface measure of $\partial\Omega$ and $Q_t=[0,t]\times \Omega.$  
The second term in the above  right hand side can be controlled easily by
$\Lambda_{1,\infty,t}\|\sigma^j\|_{L^2(Q_t)}^2.$ Note that 
\beqs
\il\pt g_1\il_{0,\infty,t}\leq \sup_{[-3{\bar{c}}\bar{P}, 3\bar{P}/\bar{c}\big]}( |g_1'(s)| )\il\ep\pt\sigma\il_{0,\infty,t}\leq \f{1}{c_0}\il\ep\pt\sigma\il_{0, \infty,t}.
\eeqs
The boundary term of the last line of  \eqref{sec5:eq2} can be treated thanks to the
boundary condition \eqref{bdrycon-time} and the
trace inequality \eqref{traceL2}
\begin{equation}
 \mu\big | \izt \int_{\partial\Omega}\Pi(\partial_{\bn} u^j)\cdot \Pi u^j \ \d S_y \d s\big | \leq \frac{\mu}{4}\|\nabla u^j\|_{L^2(Q_t)}^2+C_{\mu}\|u^j\|_{L^2(Q_t)}^2.
\end{equation}
We now detail the estimate of $(\mathcal{R}_{\sigma}^I,\mathcal{R}_{u}^I)$
which vanish unless $j\neq 0.$ For $1\leq j\leq m,$
by the commutator estimate \eqref{roughcom-T}
and the estimate \eqref{esofg12} for $g_1,$
\begin{equation}\label{sec5:eq2.25}
\begin{aligned}
\|\mathcal{R}_{\sigma}^I\|_{L^2(Q_t)}&\lesssim \|\partial_t g_1\|_{L_t^2\mathcal{H}^{m-1}}\il(\varepsilon\partial_t)\sigma\il_{[\f{m}{2}]-1,\infty,t}+\il\partial_t g_1\il_{[\f{m-1}{2}], \infty,t}\|(\varepsilon\partial_t)\sigma\|_{L_t^2\mathcal{H}^{m-1}}\\
&\quad + \| g_1 u\|_{L_t^2\mathcal{H}^{m}}\il\nabla\sigma\il_{[\f{m}{2}]-1,\infty,t}+\il g_1 u\il_{[\f{m+1}{2}],\infty,t}\|\nabla\sigma\|_{L_t^2\mathcal{H}^{m-1}}\\
&\lesssim \La\big(\|\nabla\sigma\|_{L_t^2\mathcal{H}^{m-1}}+\|(\sigma,u)\|_{L_t^2\mathcal{H}^{m}}\big).
\end{aligned}
\end{equation}
In a similar way, we have:
\begin{equation}\label{sec5:eq3}
 \|\mathcal{R}_{u}^I\|_{L^2(Q_t)}\lesssim \La\big(\|\nabla(\sigma, u)\|_{L_t^2\mathcal{H}^{m-1}}+\|(\sigma,u)\|_{L_t^2\mathcal{H}^{m}}\big).
\end{equation}
Therefore, \eqref{hightimeder} is the consequence of 
 \eqref{sec5:eq2}-\eqref{sec5:eq3}.
 Note that we have used the fact that
 $$\|(\si,u)\|_{L_t^{2}\mathcal{H}^m}
 \lesssim T^{\f{1}{2}}\|(\si,u)\|_{L_t^{\infty}\mathcal{H}^m}
 \lesim T^{\f{1}{2}}\cE_{m,T}, \qquad \|\nabla(\sigma,u)\|_{L_t^2\cH^{m-1}}\lesssim \cE_{m,T}.$$
We are now ready to prove \eqref{highestconormal}.  Suppose now that $Z^{I}$ involves at least one spatial derivative and $1\leq|I|\leq m$.
 In this case, it seems unlikely to get an uniform estimate with respect to $\varepsilon$
with this approach since $\mathcal{R}_{\sigma}^I,\mathcal{R}_{u}^I$ now contains singular terms. Taking the scalar product 
of system  \eqref{eqafterdervtion} by $\varepsilon^2 (\sigma^I, u^I),$
%{\color{red} what does it means ? }
 and integrating by parts in space and time, we get in the same way as for \eqref{sec5:eq2} that:
\begin{multline}
\label{energyineqofsigmaI}
 \varepsilon^2\int_{\Omega} (g_1|\sigma^I|^2+g_2|u^I|^2)(t)\ \d x \\
 \leq   \varepsilon^2\int_{\Omega} (g_1|\sigma^I|^2+g_2|u^I|^2)(0) \ \d x+ \iint_{Q_t}(\partial_t g_1+\div(g_1 u))|\sigma^I|^2\ \d x\d s
 \\+2\varepsilon^2 \iint_{Q_t}Z^{I}\div\mathcal{L}u\cdot u^{I} \ \d x\d s +
 \varepsilon^2\big(\|\mathcal{R}_{\sigma}^I\|_{L^2(Q_t)}\|\sigma^{I}\|_{L^2(Q_t)}
 +\|\mathcal{R}_{u}^I\|_{L^2(Q_t)}\|u^I\|_{L^2(Q_t)}\big).
\end{multline} 
Before going further, it will be convenient to introduce the notation:
\beq\label{defofE}
\|f\|_{E_t^m}=\|f\|_{L_t^2H_{co}^m}+\|\nabla f\|_{L_t^2H_{co}^{m-1}}.
\eeq
Note that from the definition of $\cE_{m,t}$ in \eqref{defcalEmt}, one has indeed that:
$\|u\|_{E_{t}^m}\lesssim \cE_{m,t}.$
%{\color{red} I just want to be more precise since in many places I used this quantity rather than $\cE_{m,T}$. It can indeed shorten some computations. I prefer to keep it but state clearly that it can be bounded by $\mathcal{E}_{m,T}.$
% }

Let us now estimate the terms in the last line of \eqref{energyineqofsigmaI}.
It follows from the commutator estimate 
\eqref{roughcom} that:
\begin{equation}\label{sec5:eq-9}
 \ep \|(\mathcal{R}_{\sigma}^I,\mathcal{R}_{u}^I)\|_{L^2(Q_t)}\lesssim \|\nabla(\sigma,u)\|_{L_t^2H_{co}^{m-1}}+\ep^{\f{1}{2}} \|(\sigma, u)\|_{E_t^{m}}\La.
\end{equation}
We remark that when controlling the extra term:
$\frac{1}{\varepsilon}[Z^I,\nabla]\sigma,$ we have used the following identity 
which can be shown by induction:
\beq\label{comu}
[Z^{I},\partial_i]=\sum_{j=1}^3\sum_{|{J}|\leq|I|-1} c_{I,{J}} Z^{{J}}\partial_j=\sum_{j=1}^3\sum_{|{{J}}|\leq|I|-1} d_{I,{J}}\partial_j Z^{{J}}
\eeq
%{\color{red} why are the indices depending on $I$ and not on $\tilde I$ why do you need different notations for irrelevant indices ? }
where $J$ is an  $(M+1)$ multi-index and $c_{I,J}, d_{I,J}$ are smooth functions that depend on $I$, $J$,  $i$ and the derivatives (up to order $|I|$) of $\nabla \phi$,   $\p_i$ is the derivation in the standard Euclidean coordinates.

It remains  to estimate the third term in the right hand side of \eqref{energyineqofsigmaI}. Since, we have 
 $$\div\mathcal{L}u=\div (2\mu\mathbb{S}u+\lambda \div u \text{Id})=
 \mu\Delta u+(\mu+\lambda)\nabla\div u,$$
  one has by integrating by parts that:
 \begin{equation}\label{idoflaplacian}
     \begin{aligned}
   &\quad\iint_{Q_t}Z^{I}\mathcal{L}u\cdot u^{I}\ \d x\d s=-\iint_{Q_t}\big(\mu [Z^I,\nabla]u\cdot\nabla u^I
   +(\mu+\lambda)[Z^I,\div]u\div u^I\big) \ \d x\d s\\
   &+\iint_{Q_t}\big(\mu[Z^I,\div]\nabla u+(\mu+\lambda)[Z^{I},\nabla]\div u\big) u^I\ \d x\d s-\iint_{Q_t}\mu|\nabla u^I |^2+(\mu+\lambda)|\div u^I|^2 \ \d x\d s\\
    &+\int_0^t\int_{\partial\Omega}\mu u^I(Z^I\nabla u\cdot \bn) +(\mu+\lambda)Z^I\div u (u^I\cdot \bn)\ \d S_y\d s=\colon \mathcal{K}_1+\mathcal{K}_2+\mathcal{K}_3+\mathcal{K}_4.\\
     \end{aligned}
 \end{equation}
Let us begin with the $\mathcal{K}_1$ term.
 By \eqref{comu} and  the Young inequality, we get 
 \begin{equation}\label{sec5:eq-11}
  \mathcal{K}_1\leq \delta\mu \|\nabla u\|_{L_t^2H_{co}^m}^2 +C_{\delta,\mu,\lambda}\|\nabla u\|_{L_t^2H_{co}^{m-1}}^2
 \end{equation}
 for $\delta>0$ to be chosen sufficiently small independent of $\ep$.
% Note that we have used the notations \eqref{simplenotation}.
 Next, by \eqref{comu} and integration by parts, $\mathcal{K}_2$ can be written as a
combination of the following two types of terms (up to some smooth coefficients that depending on $\phi, \bn$ and their derivatives up to order $m+1$):
%{\color{red} some smooth coefficient are missing in the integrals}
$$\mathcal{K}_{2}^1=\iint_{Q_t}Z^{\tilde{I}}\partial_i u\cdot\partial_j u^I\,\d x\d s,\qquad
\mathcal{K}_{2}^2=\int_0^t\int_{\partial\Omega}Z^{\tilde{I}}\partial_i u\cdot u^I \bn_j\,\d x\d s,\quad|\tilde{I}|\leq |I|-1.$$
 The term  $\mathcal{K}_{2}^1$ can be estimated in the same way as  $\mathcal{K}_1,$ we find again 
 $$\mathcal{K}_{2}^1\leq \delta\mu \|\nabla u^I\|_{L^2(Q_t)}^2 +C_{\delta,\mu,\lambda}\|\nabla u\|_{L_t^2H_{co}^{m-1}}^2.$$
 For  $\mathcal{K}_{2}^2,$  we use the trace inequality \eqref{traceL2}
% {\color{red} ref  to the statement of the trace inequality ?} 
to get that:
\begin{equation*}
\begin{aligned}
\mathcal{K}_{2}^2\lesssim\int_0^t 
|Z^{\tilde{I}}\partial_i u|_{L^2(\partial\Omega)}|u^I\cdot \bn_j|_{L^2(\partial\Omega)}\d s&\lesssim \int_0^t (|u|_{\tilde{H}^{m}(\partial\Omega)}+|\div u|_{\tilde{H}^{m-1}(\partial\Omega)})|u^I\cdot \bn_j|_{L^2(\partial\Omega)}
\ \d s\\
&\leq \delta\mu \|\nabla u\|_{L_t^2H_{co}^m}^2+C_{\delta,\mu,\lambda}\big(\|u\|_{E_t^m}^2
+\|\nabla\div u\|_{L_t^2H_{co}^{m-1}}^2\big).
\end{aligned}
\end{equation*}
To get the second  inequality, we have used
 that $\tilde{I}$ does not contain conormal derivatives of the type  $Z_3^i$ since $Z_{3}^i$ vanishes on the boundary and  the identity:
\beq\label{normalofnormalder}
\partial_{\bn} u\cdot \bn=\div u-(\Pi\partial_{y^1}u)^1-(\Pi\partial_{y^2}u)^2,
\eeq
as well as the boundary condition \eqref{bdryconditionofu}.

To summarize, we have thus proven that there exists an absolute  constant $C>0$ (independent of $\delta$ and of course
 $\ep$) such that
\begin{equation}\label{sec5:eq-12}
    \mathcal{K}_2\leq C\delta \mu \|\nabla u\|_{L_t^2H_{co}^m}^2+C_{\delta,\mu,\lambda}(%\|u\|_{L_t^2H_{co}^k}^2+\|\nabla u\|_{L_t^2H_{co}^{k-1}}^2+
    \|\nabla\div u\|_{L_t^2H_{co}^{m-1}}^2+\|u\|_{E_t^m}^2).
\end{equation}

Finally, we handle the term $\mathcal{K}_4$ in the right hand side of \eqref{idoflaplacian} which is nontrivial only if $Z^I$ contains merely   $\varepsilon\partial_t$
and tangential derivatives which read in local charts $ \partial_{y^1},\partial_{y2}.$ For the second term of $\mathcal{K}_4$, since $Z^I$ is assumed to contain at least one spatial derivative, it can be written as $Z^I=\partial_y Z^{\tilde{I}}$ (we denote $\p_y=\p_{y^1}$ or $\p_y=\p_{y^2}$).
%{\color{red} what is $\partial_{y}$  ?}
Moreover, since $u\cdot \bn|_{\partial\Omega}=0,$
$u^I\cdot\bn=[Z^I,\bn]u.$
Integrating by parts along the boundary, and then use the trace inequality \eqref{normaltraceineq}, we find that 
\begin{equation}\label{sec5:eq-10}
\begin{aligned}
 \int_0^t\int_{\partial\Omega}Z^I\div u (u^I\cdot \bn)\ \d S_y\d s
 &\leq \int_0^t|Z^{\tilde{I}}\div u|_{H^{\frac{1}{2}}(\partial\Omega)} |\partial_y [Z^I, \bn]u|_{H^{-\frac{1}{2}}(\partial\Omega)} \ \d s\\
 &\lesssim  \|\nabla\div  u\|_{L_t^2H_{co}^{m-1}}^2+\|u\|_{E_t^m}^2.
\end{aligned}
\end{equation}

For the first term of $\mathcal{K}_4,$
we can split it into two terms:
\begin{equation*}
 \begin{aligned}
& \mu\int_0^t\int_{\partial\Omega}-u^I([Z^I,\bn]
  \nabla u) +[Z^{I},\bn]\partial_{\bn} u (u^I\cdot \bn)+[Z^{I},\Pi]\partial_{\bn} u\cdot \Pi u^{I} \ \d S_y\d s\\
  &\quad-\mu\int_0^t\int_{\partial\Omega}Z^{I}(\partial_{\bn} u\cdot \bn) (u^I\cdot \bn)+ Z^{I}(\Pi\partial_{\bn} u)\cdot \Pi u^I) \ \d S_y\d s=\colon \mathcal{K}_{411}+\mathcal{K}_{412}.
 \end{aligned} 
\end{equation*}
Thanks to %\eqref{sec5:eq-10},
the trace inequality and the Young's inequality, $\mathcal{K}_{411}$ can be bounded as:
\begin{equation*}
   \begin{aligned}
   \mathcal{K}_{411}\leq \delta\mu
\|\nabla u\|_{L_t^2H_{co}^m}^2+C_{\delta,\mu}(\|u\|_{E_t^m}^2 +\|\nabla\div u\|_{L_t^2H_{co}^{m-1}}^2).%(\|\nabla u\|_{L_t^2H_{co}^{k-1}}^2+\|u\|_{L_t^2H_{co}^k}^2).
   \end{aligned}
\end{equation*}
Next, for $\mathcal{K}_{412},$  we use again 
 the identity \eqref{normalofnormalder}, as well as the boundary conditions \eqref{bdryconditionofu}. Integrating by parts
along the boundary for the first term of $\mathcal{K}_{412},$ we get that by writing  $Z^I=\partial_y Z^{\tilde{I}}$
\begin{equation*}
   \begin{aligned}
   \mathcal{K}_{412}&=\mu\int_0^t |Z^{\tilde{I}}(\partial_{\bn} u\cdot \bn)|_{H^{\frac{1}{2}}(\partial\Omega)} |\partial_y[Z^I,\bn]u|_{H^{-\frac{1}{2}}(\partial\Omega)}+ |Z^I\Pi \partial_{\bn} u|_{L^2(\partial\Omega)}|u^I|_{L^2(\partial\Omega)}\ \d s\\
   & \leq \delta\mu\|\nabla u\|_{L_t^2H_{co}^m}^2+C_{\delta,\mu}%(\|\nabla u\|_{L_t^2H_{co}^{k-1}}^2+\|u\|_{L_t^2H_{co}^k}^2)
   (\|u\|_{E_t^m}^2+\|\nabla\div u\|_{L_t^2H_{co}^{m-1}}^2). 
   \end{aligned}
\end{equation*}
%{\color{red} explain I don't understand how you get that $1/2$ derivate seems missing}
To summarize, we get the following estimate for $\mathcal{K}_{4}:$
\begin{equation}\label{sec5:eq-13}
   \mathcal{K}_{4}\leq 2 \delta\mu
\|\nabla u\|_{L_t^2H_{co}^m}^2+C_{\delta,\mu}%(\|\nabla u\|_{L_t^2H_{co}^{k-1}}^2+\|u\|_{L_t^2H_{co}^k}^2)
(\|u\|_{E_t^m}^2+\|\nabla\div u\|_{L_t^2H_{co}^{m-1}}^2).
\end{equation}
Inserting \eqref{sec5:eq-11},\eqref{sec5:eq-12},\eqref{sec5:eq-13} into \eqref{idoflaplacian}, we get that:
\begin{equation}\label{sec5:eq-14}
    \begin{aligned}
     \int_{Q_t}Z^{I}\mathcal{L}u\cdot u^{I}\ \d x\d s&\leq -\iint_{Q_t}\mu|\nabla u^I |^2+(\mu+\lambda)|\div u^I|^2 \ \d x\d s\\
     &+(C+3)\delta\mu\|\nabla u\|_{L_t^2H_{co}^m}^2+C_{\delta,\mu}%(\|\nabla u\|_{L_t^2H_{co}^{k-1}}^2+\|u\|_{L_t^2H_{co}^k}^2)
     (\|u\|_{E_t^m}^2+\|\nabla\div u\|_{L_t^2H_{co}^{m-1}}^2).
    \end{aligned}
\end{equation}
%{\color{red} the integral on $Q_{t}$ was previously a double integral}
Plugging \eqref{sec5:eq-9} and \eqref{sec5:eq-14} into \eqref{energyineqofsigmaI} and  summing up for $|I|\leq m,$
 we finally get \eqref{highestconormal}  by choosing $\delta$ small enough (independent of $\ep$).
\end{proof}
\begin{lem}\label{highestgrad}
 Under the same assumption as in Lemma \ref{lemstep1}, for any $0<t\leq T,$ one has that:
%\begin{multline}\label{timedernablasigma}\|(\nabla\sigma,\div u)(t)\|_{\cH^{m-1}}^2+\|\nabla\div u\|_{L_t^2\cH^{m-1}}^2 \lesssim  \|(\nabla\sigma,\div u)(0)\|_{\cH^{m-1}}^2  \\+T^{\frac{1}{2}}\Lambda_{2,\infty,T}\cE_{m,T}^2 +\|\nabla u\|_{L_t^2\mathcal{H}^{m-1}}^2,\end{multline}
\begin{multline}\label{highestnablasigma}
   \quad\varepsilon^2\|(\nabla\sigma,\div u)(t)\|_{H_{co}^{m-1}(\Omega)}^2+\varepsilon^2\|\nabla\div u\|_{L_t^2{H}_{co}^{m-1}}^2\\
\lesssim  %\sum_{|I|\leq m-1}
\|(\nabla\sigma,\div u)(0)\|_{H_{co}^{m-1}}^2
 +(T^{\frac{1}{2}}+\ep^{\f{2}{3}})\Lambda_{2,\infty,T}\cE_{m,T}^2.
\end{multline}
\end{lem}
\begin{proof}

Applying  the vector field $Z^I$ with $0\leq |I|\leq m-1, $ we 
then find that  $((\nabla \sigma)^I, u^I)= (Z^I\nabla\sigma, Z^I u)$ solves the system:
\begin{equation}\label{sec5:eq4}
\left\{
    \begin{array}{l}
  \displaystyle g_1(\pt+u\cdot\nabla)( \nabla \sigma)^I+\frac{\nabla\div u^I}{\varepsilon}=\colon \mathcal{C}_{\sigma}^I, \\
\displaystyle g_2\partial_t u^I-\mu \curl(Z^I\omega)-(2\mu+\lambda)\nabla \div u^I+\frac{(\nabla\si)^I}{\varepsilon}=\colon      \mathcal{C}_{u}^I, \\
\end{array}
\right.
\end{equation}
where $\omega=\curl u$ and
\begin{equation}\label{commutators}
    \begin{aligned}
    & \qquad\qquad\mathcal{C}_{\sigma}^I=-[Z^I\nabla,g_1/\varepsilon]\varepsilon \partial_t\sigma-
     [Z^I\nabla,g_1 u\cdot \nabla]\sigma-[Z^I,\nabla\div]u/\varepsilon, \\
     &\mathcal{C}_{u}^I=-Z^I(g_2u\cdot\nabla u)-[Z^I,g_2/\varepsilon]\varepsilon \partial_t u
     +\mu[Z^I,\curl]\omega+(2\mu+\lambda)[Z^{I},\nabla\div]u.%+[Z^I,\nabla]\sigma/\varepsilon.
    \end{aligned}
\end{equation}
We take the scalar product of  the equation \eqref{sec5:eq4}$_1$ by $(\nabla\sigma)^I$, and \eqref{sec5:eq4}$_2$ by $-\nabla\div u^I,$
we then integrate in space and time and sum up the two equations to 
 get %by integrating by parts 
that(note  that the singular  terms cancel):
\beq\label{sec5:eq5}
\begin{aligned}
 &\frac{1}{2}\int_{\Omega} (g_1|(\nabla\sigma)^I|^2+g_2|\div u^I|^2)(t)\  \d x+(2\mu+\lambda)\iint_{Q_t}|\nabla\div u^I|^2\ \d x\d s\\
 &\leq \frac{1}{2}\int_{\Omega} (g_1|\nabla \sigma^I|^2+g_2|\div u^I|^2)(0) \ \d x+\frac{1}{2}\bigg| \iint_{Q_t}(\partial_t g_1+\div(g_1 u))|\nabla\sigma^I|^2\ \d x\d s\bigg|\\
 &\quad+\bigg|\iint_{Q_t}(g_2'\varepsilon\partial_t u^I \cdot \nabla\sigma)\div u^I\ \d x\d s\bigg|+\bigg|\int_0^t\int_{\partial\Omega}g_2\partial_t u^I\cdot \bn\div u^I \ \d S_y \d s\bigg|\\
 &\qquad+\mu\bigg| \iint_{Q_t}\curl Z^I\omega\nabla\div u^I\ \d x\d s\bigg|\\
 &\quad+\|\mathcal{C}_{\sigma}^I\|_{L^2(Q_t)}\|\nabla\sigma^I\|_{L^2(Q_t)}
 +\frac{1}{(2\mu+\lambda)} \|\mathcal{C}_{u}^I\|_{L^2(Q_t)}^2
 +\frac{2\mu+\lambda}{4}\|\nabla\div u^I\|_{L^2(Q_t)}^2.
 \end{aligned}
\eeq
Among the terms in the right hand side, the second and the third terms can be bounded by:
\begin{equation}\label{sec5:eq-19}
  \Lambda\big(\f{1}{c_0}, \il(\sigma,u)\il_{1,\infty,t}+\il(\nabla\sigma,\div u)\il_{0,\infty,t}\big) \left\|\left((\nabla \sigma)^I,\div u^I,\varepsilon\partial_t u^I\right)\right\|_{L^2(Q_t)}^2. 
\end{equation}
Next, we note that
%by the boundary condition \eqref{bdrycon-time}, 
 the fourth term vanishes if 
 $Z^I$ involves at least one conormal derivative $Z_{3}^i$ which vanishes on the boundary.
 %{\color{red} I don't understand if there is a $Z_{3}$ involved, this  vanishes, why do you need to use the boundary conditions ? }
We thus suppose that $I=(l,I'),|I'|\geq 1$ and $Z^I$ does not contain $Z_{3}^i$.
Consequently, the trace inequality  \eqref{traceL2} leads to
\begin{equation}\label{sec5:eq-18}
\begin{aligned}
  \qquad \big| \int_0^t\int_{\partial\Omega}&g_2\partial_t u^I\cdot \bn\,\div u^I\  \d S_y \d s\big|\lesssim\frac{1}{\varepsilon} \int_0^t | [Z^I,\bn]\varepsilon\partial_t u(s)|_{L^2(\partial\Omega)}|\div u^I(s)|_{L^2(\partial\Omega)}\ \d s\\
    &\lesssim \frac{1}{\varepsilon}  
    (\|\nabla u\|_{L_t^2H_{co}^{m-1}}+\|u\|_{L_t^2H_{co}^{m-1}})(\|\nabla\div u^I\|_{L^2(Q_t)}^{\frac{1}{2}}\|\div u^I\|_{L^2(Q_t)}^{\frac{1}{2}}+\|\div u^I\|_{L^2(Q_t)})\\
    &\leq \frac{2\mu+\lambda}{4}\|\nabla\div u\|_{L^2(Q_t)}^2+C_{\mu,\lambda}(1+\varepsilon^{-\frac{4}{3}})\|(u,\nabla u)\|_{L_t^2H_{co}^{m-1}}^2.
\end{aligned}
\end{equation}
Note that since $\partial_t u\cdot \bn|_{\partial\Omega}=0,$ one has $(Z^I \partial_t u\cdot \bn)|_{\partial\Omega}=([Z^I,\bn]\partial_t u)|_{\partial\Omega}.$

For the fifth term in the right hand side of \eqref{sec5:eq5} we first integrate by parts and then use the  duality $\langle \cdot, \cdot \rangle_{H^{\f{1}{2}}(\p\Omega)\times H^{-\f{1}{2}}(\p\Omega)}$
to get that
 \begin{equation*}
   \begin{aligned}
   \mu \big|\iint_{Q_t} \curl Z^I\omega\cdot \nabla\div u^I\ \d x\d s\big|&=-\mu\int_0^t\int_{\partial\Omega} (Z^I\omega\times \bn)\cdot \Pi\nabla\div u^I \ \d S_y\d s\\
   &\leq \mu \int_0^t 
   |Z^I\omega\times  \bn(s)|_{H^{\frac{1}{2}}(\p\Omega)}|\div u^I(s)|_{H^{\frac{1}{2}}(\p\Omega)} \ \d s\\
   \end{aligned}
\end{equation*}
%{\color{red} you have to explain how you get the last line}
We point out  that for the derivation of the last line, the fact that 
$\Pi\nabla$ involves only tangential derivatives has been used. It remains  to control $Z^I \omega\times\bn$ on the boundary.
%by the boundary conditions \eqref{bdryconditionofu}, we have that on the boundary,
One first deduces by \eqref{bdryconditionofu} that on the boundary,
\beq\label{bdryomegan}
\omega\times \bn
=\Pi (\omega\times\bn)=2\Pi(\mathbb{S} u)-2\Pi((\nabla u)^{t}\cdot\bn)=
2\Pi(-a u+D \bn\cdot u )|_{\partial\Omega}.
\eeq
%Therefore, one controls the first term as %for $0\leq |I|=k\leq m-1,$
which leads to:
\begin{equation*}
\begin{aligned}
  |Z^I\omega\times \bn(s)|_{H^{\frac{1}{2}}(\partial\Omega)}&\lesssim 
   |Z^{I}(\omega(s)\times \bn)|_{H^{\frac{1}{2}}(\partial\Omega)}+|[Z^I,\bn]\times \omega|_{H^{\frac{1}{2}}(\partial\Omega)}\\
   &\lesssim |u(s)|_{\tilde{H}^{m-\frac{1}{2}}}+|\omega(s)|_{\tilde{H}^{m-\frac{3}{2}}}\lesssim |u(s)|_{\tilde{H}^{m-\frac{1}{2}}}+|\div u(s)|_{\tilde{H}^{m-\frac{3}{2}}}
\end{aligned}
\end{equation*}
where we recall that we  denote:
$$|f(t)|_{\tilde{H}^r}:=\sum_{k\leq [r]}|(\ep\pt)^k f(t)|_{H^{r-k}(\p\Omega)}.$$
Note that by using the boundary condition \eqref{bdryconditionofu} and the  identity \eqref{normalofnormalder}, we have that:
$$|\nabla u|_{\tilde{H}^s}\lesssim |u|_{\tilde{H}^{s+1}}+|\div u|_{\tilde{H}^s}.$$
%where the $\tilde{H}^{\tau}(\partial\Omega)(\tau\geq-\frac{1}{2})$ is defined by: %$s\geq -\frac{1}{2}:$
%\begin{equation}\label{tildesobolev}
%  \|h(t)\|_{\tilde{H}^{\tau}(\partial\Omega)}=\sum_{j=0}^{[\tau+1/2]}|(\varepsilon\partial_t)^{j} h(t)|_{H^{\tau-j}(\partial\Omega)}.  
%\end{equation}
Finally, owing to the trace inequality \eqref{normaltraceineq}
%{\color{red} you have to refer precisely  to which  inequalityyou use, it does not seem to be the one you quote in the appendix} 
and Young's inequality, one obtains that:
\begin{equation}\label{sec5:eq-20}
       \begin{aligned}
  & \mu\big| \iint_{Q_t}\curl Z^I\omega\cdot\nabla\div u^I\ \d x\d s\big|\\ &\leq C\mu (\|\nabla\div u\|_{L_t^2H_{co}^{m-2}}+\|\nabla u\|_{L_t^2H_{co}^{m-1}}+\|u\|_{L_t^2H_{co}^{m}}) (\|\div u^I\|_{L^2(Q_t)}+\|\nabla\div u^I\|_{L^2(Q_t)})\\
   &\leq \frac{2\mu+\lambda}{4}\|\nabla\div u^I\|_{L^2(Q_t)}^2+C_{\mu,\lambda}(\|\nabla\div u\|_{L_t^2H_{co}^{m-2}}^2+\|u\|_{E_t^{m}}^2)%\|\nabla u\|_{L_t^2H_{co}^{k}}^2+\|u\|_{L_t^2H_{co}^{k+1}}^2).
   \end{aligned}
\end{equation}
where we use again the notation \eqref{defofE}.

It remains to control the $L^2(Q_t)$ norm of ${\mathcal{C}}_{\sigma}^I,\mathcal{C}_{u}^I$ in \eqref{sec5:eq5}.
Let us begin with the estimate ${\mathcal{C}}_{\sigma}^I.$
For the term: 
$$[Z^I\nabla,\frac{g_1}{\varepsilon}]\varepsilon \partial_t\sigma=Z^I((\nabla g_1/\ep) \ep\pt\sigma)+[Z^I,g_1/\ep] (\varepsilon \partial_t)\nabla\sigma,$$
 the product estimates \eqref{roughproduct1} the commutator estimate \eqref{roughcom} and the estimate 
\eqref{esofg12-2} yield:
\begin{equation*}
\begin{aligned}
 \|[Z^I\nabla,g_1/\varepsilon]\varepsilon \partial_t\sigma\|_{L^2(Q_t)}&\lesssim
 \|(\ep\pt\sigma,\nabla\sigma)\|_{L_t^2H_{co}^{m-1}}\Lambda\big(\f{1}{c_0},\il\nabla\sigma\il_{[\f{m}{2}]-1,\infty,t}+\il\sigma\il_{[\f{m+1}{2}],\infty,t}\big)\\
 &\lesssim  \|\sigma\|_{E^m_t}\La.
 \end{aligned}
\end{equation*}
For the term 
$$[Z^I\nabla,g_1 u\cdot \nabla] \sigma= Z^I\big(\nabla(g_1 u)\nabla \sigma\big)+[Z^I, g_1 u]\nabla\nabla \sigma,$$
 since in the interior domain $\Omega_0$, the spatial conormal derivatives are equivalent to the  
derivations with respect to the standard coordinates in $\mathbb{R}^3.$ We thus have that:
%The commutator estimate \eqref{roughcom} thus yield:
\begin{align*}
  \ep\|\chi_0[Z^I\nabla,g_1 u\cdot \nabla] \sigma\|_{L^2(Q_t)}&\lesssim (\|\tilde{\chi}_0(\sigma,u)\|_{L_t^2H^{m}}+\|\tilde{\chi}_0\nabla(\sigma,u)\|_{L_t^2H^{m-1}})
  \Lambda\big(\f{1}{c_0},%\il\tilde{\chi}_0,
  \il \ep (\sigma,u)\il_{[\f{m}{2}]+1,\infty,t}\big).\\
  &\lesssim \|(\sigma,u)\|_{E_t^m}\La.
\end{align*}
where $\Supp(\tilde{\chi}_0)\Subset \Omega$ and $\tilde{\chi}_0\chi_0=\chi_0.$
%{\color{red} I don't think this statement is correct or I don't understand it, it's strange that you can get an estimate with the same
% $\chi_{0}$ on both sides. Yes, you are right, I changed the $\chi_0$ in the right-hand side }.
It suffices to focus on the case near the boundary. 
 Direct computations show that, in the local coordinates \eqref{local coordinates},
 \begin{equation}\label{convection identity}
 u\cdot \nabla f=u_1\partial_{y^1} f+u_2\partial_{y^2}f+ u\cdot \textbf{N}\partial_z f,
 \end{equation}
%{\color{red} the tangential vector fields are defined as $\partial_{y^{i}}$, you should also use $u^i$ to stress the fact that these are not the standard components in $\mathbb{R}^3$: In this identity, $u_1, u_2$ are usual components in standard coordinate so we do not need to stress.}
which leads to:
\begin{equation}\label{split-com}
\begin{aligned}
&[Z^I\nabla,g_1 u\cdot \nabla] \sigma = Z^I\big(\nabla(g_1 u)\nabla \sigma\big)+\sum_{j=1}^2
[Z^I, g_1 u_j]\partial_{y^j}\nabla \sigma\\
&\qquad+[Z^I,  (g_1 u\cdot \textbf{N})/\phi]\phi\partial_z\nabla \sigma+\big((g_1 u\cdot \textbf{N})/\phi\big)[Z^I,\phi]\partial_z\nabla\sigma+(g_1 u\cdot \textbf{N})[Z^I,\partial_z]\nabla\sigma.
\end{aligned}
\end{equation}
With the help of the product and commutator estimates 
\eqref{roughproduct1}, \eqref{roughcom} and the estimate  \eqref{esofg12-2} for $g_1$, the first two terms in the right hand side of \eqref{split-com} can be bounded as:
\beq\label{com-1}
\begin{aligned}
&\ep\|\chi_i Z^I\big(\nabla(g_1 u)\nabla \sigma\big)\|_{L^2(Q_t)}+\sum_{j=1}^2
\|\chi_i[Z^I, g_1 u_j]\partial_{y^j}\nabla \sigma\|_{L^2(Q_t)}\\
&\lesssim 
\|(\sigma,u)\|_{E_t^{m}}\Lambda\big(\f{1}{c_0}, \ep\il(\sigma,u)\il_{[\f{m}{2}],\infty,t}+\il\nabla\sigma\il_{[\f{m-1}{2}],\infty,t}+\ep\il\nabla u\il_{[\f{m}{2}],\infty,t}\big)\\
&\lesssim \|(\sigma,u)\|_{E_t^{m}}\La.
\end{aligned}
\eeq
To continue, we need to establish some estimates on $(g_1u\cdot \textbf{N})/\phi.$
At first, since $(u\cdot \bn)|_{\partial\Omega}=0,$ one has by the fundamental theorem of calculus and the identity \eqref{normalofnormalder} that:
\begin{equation}\label{unLinfty}
\begin{aligned}
    \il\chi_j (g_ju\cdot \textbf{N})/\phi\il_{k,\infty,t}&\lesssim 
    (\il \nabla(u\cdot \textbf{N}))\il_{k,\infty,t}+ \il u\il_{k,\infty,t})\il g\il_{k,\infty,t}\\
    &\lesssim \Lambda\big(\f{1}{c_0},\il u\il_{k+1,\infty,t}+\il(\sigma, \div u)\il_{k,\infty,t}\big),\, \quad j=1,2.
\end{aligned}
\end{equation}
Next, thanks to Hardy inequality and product estimate \eqref{roughproduct1}, estimate \eqref{esofg12-1} for $g_j$ 
\begin{equation}\label{hardyconormal}
\begin{aligned}
\|\chi_i (g_j u\cdot \textbf{N})/\phi\|_{L_t^2H_{co}^{m-1}}&\lesssim \|\tilde{\chi}_i(u\cdot\bN)/\phi\|_{L_t^2H_{co}^{m-1}}+\|(g_j-g_j(0))(u\cdot\bN)/\phi\|_{L_t^2H_{co}^{m-1}})\\
&\lesssim
\big(\|\tilde{\chi}_i(u,\nabla  u)\|_{L_t^2H_{co}^{m-1}}+\|g_j-g_j(0)\|_{L_t^2H_{co}^{m-1}}\big)\La\\
&\lesssim \La\|(\sigma,u)\|_{E_t^m},\qquad j=1,2,
\end{aligned}
\end{equation}
where $\tilde{\chi}_i$ is a cut-off function supported on the vicinity of $\Omega_i$ and $\tilde{\chi}_i\chi_i=\chi_i.$
Therefore, since $\phi\p_z$ can be spanned by $Z_1^{i}, Z_2^{i}, Z_3^{i},$
it follows from \eqref{unLinfty},
\eqref{hardyconormal}, 
\eqref{roughcom},
\eqref{esofg12-2} that:
\begin{equation}\label{com-2}
\begin{aligned}
 &\ep \|\chi_{i}[Z^I,  (g_1 u\cdot \textbf{N})/\phi]\phi\partial_z\nabla \sigma\|_{L^2(Q_t)}\\
 &\lesssim
  \|(\nabla\sigma,(g_1u\cdot\bN)/\phi)\|_{L_t^2H_{co}^{m-1}}\Lambda\big(\f{1}{c_0}, \|\nabla \sigma \|_{[\f{m-1}{2}],\infty,t}+\ep\il\tilde{\chi}_i(g_1u\cdot\bN)/\phi\il_{[\f{m}{2}],\infty,t}
 \big)\\
  &\lesssim \|(\sigma,u)\|_{E_t^{m}}\La.
\end{aligned}
\end{equation}
Moreover, one gets by induction that (up to some coefficients that depend only on $\phi$ and its derivatives)
 \beq\label{sec5:eq-16}
 [Z^{I},\phi](\partial_z f)=\sum_{|\tilde{I}|\leq|I|-1} *_{\tilde I}Z^{\tilde{I}}(\phi\partial_z f),\qquad [Z^I,\partial_z]=
 \sum_{|\tilde{I}|\leq|I|-1} *_{\tilde I}\partial_z
 Z^{\tilde{I}}
 \eeq
Hence, by \eqref{unLinfty}, the last two terms in \eqref{split-com} can be controlled by
$\|\nabla\sigma\|_{L_t^2H_{co}^{m-1}}\La,$ which, together with
\eqref{com-1}, \eqref{com-2} leads to:
\beq
\ep\|\chi_i[Z^I\nabla,g_1 u\cdot \nabla] \sigma\|_{L^2(Q_t)}\lesssim \|(\sigma,u)\|_{E_{t}^{m}}\La.
\eeq
 We switch to the estimate of the third term of ${\mathcal{C}}_{\sigma}^I$ defined in \eqref{commutators},
 which is nontrivial only if $Z^{I}$ contains at least one spatial derivative, that is $|I'|\geq 1.$
By induction, one has that (up to some coefficients which are regular enough) $$[Z^{I},\nabla\div]=\sum_{|\tilde{I}|\leq |I|-1,|\tilde{\tilde{I}}|\leq|I|-1}\sum_{j,k=1}^3 *_{jk \tilde I} \partial_{jk}^2 Z^{\tilde{I}}+ *_{j \tilde I}\partial_j Z^{\tilde{\tilde{I}}},$$
which yields that:
\begin{equation*}
   \f{1}{\ep} \|[Z^I,\nabla\div]u\|_{L^2(Q_t)}\lesssim  %h(|I'|)
     \f{1}{\ep} (\|\nabla^2 u\|_{L_t^2H_{co}^{m-2}}+\|\nabla u\|_{L_t^2H_{co}^{m-2}}).
\end{equation*}
To summarize, we have thus obtained from the above estimates that:
\begin{equation}\label{sec5:eq-21}
  \ep\|{\mathcal{C}}_{\sigma}^I\|_{L^2(Q_t)}\lesssim   \La\|(\sigma,u)\|_{E_t^{m}}
+\|(\nabla^2 u,\nabla u)\|_{L_t^2H_{co}^{m-2}}.
\end{equation}
 By using the same argument, $\mathcal{C}_u^I$ (defined in \eqref{commutators})
 can be controlled as follows:
\begin{equation}\label{sec5:eq-22}
 \ep \|\mathcal{C}_u^I\|_{L^2(Q_t)}\lesssim
  \La\|(\sigma,u)\|_{E_t^{m}}+\ep
  \|\nabla^2 u\|_{L_t^2H_{co}^{m-2}}.
\end{equation}
Plugging \eqref{sec5:eq-19} \eqref{sec5:eq-18} \eqref{sec5:eq-20} \eqref{sec5:eq-21}
\eqref{sec5:eq-22} 
in \eqref{sec5:eq5},
we arrive at 
\begin{equation}\label{sec5:eq-23}
    \begin{aligned}
     &\quad\varepsilon^{2}\big(\|((\nabla\sigma)^{I},\div u^I)(t)\|_{L^2(\Omega)}^2+%(2\mu+\lambda)
     \|\nabla\div u^I\|_{L^2(Q_t)}^2\big)\\
     &\lesssim \varepsilon^{2}%h(|I'|)}
     \|((\nabla\sigma)^{I},\div u^I)(0)\|_{L^2(\Omega)}^2
     + \varepsilon^{\frac{2}{3}} \La
    \|(\sigma,u)\|_{E_t^{m}}^2\\
    &\qquad +T^{\f{1}{2}}\|\varepsilon\nabla^2 u\|_{L_t^{\infty}H_{co}^{m-2}}(\|\varepsilon\nabla^2 u\|_{L_t^2H_{co}^{m-2}}+\|\nabla\sigma\|_{L_t^2H_{co}^{m-1}}).
    \end{aligned}
\end{equation}
We thus get  
\eqref{highestnablasigma} by 
summing up \eqref{sec5:eq-23} for $0\leq |I|\leq m-1.$
\end{proof}
\subsection{Step 2: Energy estimate for the incompressible part of velocity}
In this subsection, we focus on the estimates of the  incompressible part of the velocity $v=\mathbb{P}u$ which solves $\eqref{reformulation}_3$. 

In the following, we recall for convenience the definition of the $L_{t,x}^{\infty}$ norm:
\begin{equation}\label{defofLambdam}
\begin{aligned}
\cA_{m,t}&=%\Lambda\big(\f{1}{c_0},
\il \nabla u\il_{0,\infty,t}+
\il (u,\sigma)\il_{[\frac{m+1}{2}],\infty,t}+\il(\nabla\sigma,\div u,\varepsilon^{\frac{1}{2}}\nabla u)\il_{[\frac{m-1}{2}],\infty,t}\\
&\qquad\qquad\qquad+\il
\ep \nabla u\il_{[\frac{m+1}{2}],\infty,t}+\varepsilon\il(\sigma,u)\il_{[\frac{m+3}{2}],\infty,t}. %\big)
\end{aligned}
\end{equation}
%for some polynomial $\Lambda$ which may still differ from line to line.
\begin{rmk}\label{rmkLinfty}
In view of the first term in $\cA_{m,t},$ we have only the uniform control of $\nabla u$  
in $L_{t,x}^{\infty}$ space.  
Indeed, by some delicate analysis on the Green function for the vorticity in the local coordinates, it is possible to get the uniform control of the high order conormal derivatives of $\nabla u$ (say $\il\nabla u\il_{[\f{m}{2}]-2,\infty,t}$). One can refer for instance to \cite{lowmachfree}. Nevertheless, 
involving only $\il\nabla u\il_{0,\infty,t}$ in $\cA_{m,t}$ is enough for us to close our estimate.
See Lemma \ref{lem-f-bd} and Proposition \ref{propGomega}.
\end{rmk}

We begin with some additional estimates on $\nabla\div u:$
\begin{lem}
Suppose that \eqref{preasption} holds then for any $0<t\leq T\leq 1.$
\beq\label{nabladivu-4}
 \|\nabla\div u\|_{L_t^2H_{co}^{m-2}}\lesssim
 \|\nabla\sigma\|_{L_t^2H_{co}^{m-1}}+ \ep^{\f{1}{2}}\La\|(\sigma,u)\|_{E_t^{m}},
\eeq
\beq\label{nabladivu-5}
 \begin{aligned}
 \ep\|\nabla\div u(t)\|_{H_{co}^{m-2}}\lesssim 
 \ep\|\nabla\sigma\|_{L_t^{\infty}H_{co}^{m-1}}+
 \ep \La \cE_{m,t},
  \end{aligned}
 \eeq
 \beq\label{nabladivu-6}
 \|\nabla\div u(t)\|_{ H_{co}^{m-3}}\lesssim  \La\cE_{m,t}.
 \eeq
\end{lem}
\begin{proof}
By the equation for $\sigma,$ we have that: 
 \beq\label{rewriteofsigma}
 \nabla\div u=g_1(0)\ep\pt\nabla\sigma+ \ep\nabla\big(\f{g_1(\ep \sigma)-g_1(0)}{\ep}\varepsilon\partial_t\sigma+ g_1(\ep \sigma) u\cdot\nabla\sigma\big).
 \eeq
 We can control $\ep\nabla\div u$ as follows, for $p=2,+\infty,$
 \beq
 \begin{aligned}
 \|\nabla\div u\|_{L_t^pH_{co}^{m-2}}&\lesssim \|\nabla\sigma\|_{L_t^p H_{co}^{m-1}}+\ep \|\nabla \big((g_1-g_1(0))\pt\sigma, g_1u\cdot\nabla\sigma\big)(t)\|_{L_t^pH_{co}^{m-2}}\\
 %&\lesssim  \|\nabla\sigma\|_{L_t^p H_{co}^{m-1}}+\ep\Lambda_{[\frac{m}{2}],\infty,t}^{\varepsilon}\cE_{m,t}.
  \end{aligned}
 \eeq
Inequalities  \eqref{nabladivu-4}-\eqref{nabladivu-5} can thus be derived from the following estimate:
\beqs
\begin{aligned}
&\ep\|\nabla \big((g_1-g_1(0))\pt\sigma, g_1u\cdot\nabla\sigma\big)(t)\|_{L_t^pH_{co}^{m-2}}\\
&\lesssim \La \big(\|\ep\nabla(\sigma, u)\|_{L_t^pH_{co}^{m-1}}+\ep^{\f{1}{2}}\|(\sigma,u,\nabla\sigma,\nabla u)\|_{L_t^pH_{co}^{m-2}}\big).
\end{aligned}
\eeqs
Let us show the estimate of 
the term $g_1u\cdot\nabla\nabla\sigma,$ the other terms can be controlled in a similar way.
% Therefore, by using the product estimates \eqref{roughproduct1},
 %\eqref{esofg12-1}, 
 %\begin{equation*}\begin{aligned} \|\nabla\div u\|_{L_t^2H_{co}^{m-2}}\lesssim \La(\|\nabla\sigma\|_{L_t^2H_{co}^{m-1}}+\ep\|(\sigma,u)\|_{E_t^m}).\end{aligned}\end{equation*}
 %Note that by using the  identity \eqref{convection identity},  the term $\|g_1u\cdot\nabla\nabla\sigma\|_{L_t^2H_{co}^{m-2}}$ can be controlled by  $$\|(\sigma,u)\|_{E_t^m}\Lambda\big(\f{1}{c_0}, \il(\nabla\sigma,\div u)\il_{[\f{m-1}{2}],\infty,t}+\il(\sigma,u)\il_{[\f{m+1}{2}],\infty,t}\big)\lesssim \|(\sigma,u)\|_{E_t^m}\La. $$
%We now prove \eqref{nabladivu-5}.
 %By using the equation \eqref{rewriteofsigma}, 
%We detail only the estimate of the term: $ g_1u\cdot\nabla\nabla\sigma$ near the boundary, the other terms are easier to control.
Again, we focus only on the estimate near the boundary.
Thanks to the  identity \eqref{convection identity}, we have
 \beqs
 \chi_i g_1u\cdot\nabla\nabla\sigma=\chi_ig_1u_y\cdot\p_y \nabla\sigma+\chi_i g_1\f{u\cdot\bN}{\phi}\phi\p_z\nabla\sigma.
 \eeqs
Therefore, by applying the product estimate \eqref{roughproduct1} and inequality \eqref{unLinfty}, we find
\beq
\begin{aligned}
\ep \|\chi_i(g_1 u\cdot\nabla\nabla\sigma)\|_{L_t^p H_{co}^{m-2}}&\lesssim\ep \|
(u_y, {\chi}_i u\cdot\bN/\phi)\|_{L_t^pH_{co}^{m-2}}\il g_1Z\nabla\sigma\il_{[\f{m-1}{2}]-1,\infty,t}\\
&\quad+\ep \| g_1Z\nabla\sigma(t)\|_{L_t^pH_{co}^{m-2}}\il (u_y,{\chi}_i u\cdot\bN/\phi)\il_{[\f{m}{2}]-1,\infty,t}\\
%&\quad+\|\ep g_1Z\nabla\sigma(t)\|_{L_t^pH_{co}^{m-2}}\il (u_y, {\chi}_i u\cdot\bN/\phi) \il_{\infty,t}\\
&\lesssim \La (\|\ep\nabla\sigma\|_{L_t^{p}H_{co}^{m-1}}+\ep^{\f{1}{2}}\|(u,\nabla\sigma,\nabla u)\|_{L_t^{p}H_{co}^{m-2}}).
\end{aligned}
\eeq
%Note that by the fundamental theorem of calculus and  the identity \eqref{normalofnormalder},
%\beq\label{ucdotbn}
%\begin{aligned}\il{\chi}_i(u\cdot\bN)/\phi\il_{[\f{m-1}{2}],\infty,t}&\lesssim \il\chi_i (\p_{\bn}u\cdot\bN, u\cdot\bN)\il_{[\f{m-1}{2}],\infty,t}\\&\lesssim \il \div u\il_{[\f{m-1}{2}],\infty,t}+\il u\il_{[\f{m+1}{2}],\infty,t}.\end{aligned}\eeq
Finally, one gets \eqref{nabladivu-6} by using similar arguments  as in the derivation of \eqref{nabladivu-5}, we skip the details.
\end{proof}
\begin{rmk}
By \eqref{sumstep1} and \eqref{nabladivu-5}, we have that:
\beq\label{nabladivu-51}
\ep\|\nabla\div u\|_{L_t^{\infty}H_{co}^{m-2}}\lesssim Y_m(\sigma_0, u_0)+(T+\ep)^{\f{1}{4}}\La \cE_{m,t}.
\eeq
\end{rmk}

\begin{lem}\label{lem-f-bd}
Let 
\beq\label{def-f-bd}
f=-\frac{g_2-\bar{\rho}}{\varepsilon}\varepsilon\partial_t u-g_2u\cdot\nabla u
\eeq
and assume that \eqref{preasption} holds, then we have:
\beq\label{es-f-bd}
\|f\|_{L_t^{2}H_{co}^{m-1}}+\|f\|_{L_t^{\infty}H_{co}^{m-2}}\lesssim \La\cE_{m,t}.
\eeq
\end{lem}
\begin{proof}
Since the the higher order $L_{t,x}^{\infty}$ norm of $\p_{\bn} u$ is not included in 
the definition of $\cA_{m,t},$ 
we need to use again the fact that $u\cdot\bn$ vanishes on the boundary. 
More precisely, by using the product estimate \eqref{roughproduct1}, identity \eqref{convection identity} and the estimate \eqref{hardyconormal}, we get for 
$(p,k)=(2,1), (\infty,2),$
\beqs
\|g_2u\cdot\nabla u\|_{L_t^pH_{co}^{m-k}}\lesssim 
\|(\sigma,u,\nabla\sigma,\nabla u)\|_{L_t^pH_{co}^{m-k}}\Lambda\big(\f{1}{c_0},\il(\nabla\sigma,\div u)\il_{[\f{m-1}{2}],\infty,t}+\il(\sigma,u)\il_{[\f{m+1}{2}],\infty,t}\big).
\eeqs
The first term is a direct application of the product estimate \eqref{roughproduct1}, we omit the detail.
\end{proof}
We split the estimate for 
$v$ in the following three subsections.

\subsubsection{Estimate of $\nabla q$}
We first give the estimate of $\nabla q$ that appears in
$\eqref{reformulation}_3.$
Since $q$ is governed by the elliptic equation \eqref{eqofq} without singular terms, it can be easily estimated by standard elliptic regularity theory.
\begin{lem}\label{lemofq}
Under the  assumptions \eqref{preasption}, we have 
the following estimates: for 
$j+l\leq m-1, l\geq 1,$
\begin{equation}\label{esofnablaq}
\begin{aligned}
   \|\nabla q\|_{L_t^2\cH^{j,l}}+\ep^{\f{1}{2}}\|\nabla q\|_{L_t^2\cH^{m-1}}&\lesssim 
   \La\cE_{m,t}
 \end{aligned}
\end{equation}
where $E_t^m$ is defined in \eqref{defofE}.
Moreover, %for any $0\leq t\leq T,$
\begin{equation}\label{curlwLinfty}
\begin{aligned}
 & \varepsilon \|\curl\omega(t)\|_{H_{co}^{m-2}} + \varepsilon\|\nabla q(t)\|_{H_{co}^{m-2}} \lesssim \|v(t)\|_{H_{co}^{m-1}}+Y_m(\sigma_0,u_0)
+(T+\ep)^{\f{1}{4}}\La \cE_{m,t}.
\end{aligned}
\end{equation}
\end{lem}
\begin{proof}
Recall that $q$ is governed by \eqref{eqofq}, an elliptic equation with Neumann boundary conditions.  
We can apply \eqref{highconormal} in  the appendix by setting
 $$f=-\frac{g_2-\bar{\rho}}{\varepsilon}\varepsilon\partial_t u-g_2u\cdot\nabla u, \quad g=\mu\Delta v\cdot \bn$$
 to get 
\begin{equation}\label{nablaqpre}
    \begin{aligned}
    \|\nabla q\|_{L_t^2\cH^{j,l}}&\lesssim
   \big \|f\big\|_{L_t^2H_{co}^{m-1}}+\sum_{|I|\leq m-1}|Z^I(\Delta v\cdot \bn)|_{L_t^2{H}^{-\frac{1}{2}}(\partial\Omega)}
  % \bn)|_{L_t^2{H}^{-\frac{1}{2}}(\partial\Omega)}.
    \end{aligned}
\end{equation}
The first term in the right hand side
has been controlled in \eqref{es-f-bd},
it  remains to estimate
the boundary term. By using the identity 
\beq\label{atimesb}
(\nabla\times a)\cdot b=\nabla\cdot(a\times b)+a\cdot(\nabla\times b),
\eeq
we have that:
$$-\Delta v\cdot \bn=(\nabla\times\omega) \cdot \bn=\div(\omega\times \bn)+\omega\cdot \curl \bn.$$
Near the boundary, it follows from \eqref{normalofnormalder} that:
\begin{equation}\label{omegatimesn}
\begin{aligned}
\div (\omega\times \bn)&=\partial_{\bn}(\omega\times \bn)\cdot \bn+(\Pi\partial_{y^1}(\omega\times \bn))^1+(\Pi\partial_{y^2}(\omega\times \bn))^2\\
&=-(\omega\times \bn)\cdot\partial_{\bn} \bn+(\Pi\partial_{y^1}(\omega\times \bn))^1+(\Pi\partial_{y^2}(\omega\times \bn))^2.
\end{aligned}   
\end{equation}
Therefore,  by using the  boundary condition \eqref{bdryomegan}, %$\omega\times \bn|_{\partial\Omega}=2\Pi(-a u+D n\cdot u )|_{\partial\Omega},$
 one has that for $|I|\leq m-1,$
 \beq\label{bdry1}
 | Z^I(\div (\omega\times \bn))|_{L_t^2{H}^{-\frac{1}{2}}(\partial\Omega)}\lesssim |u|_{L_t^2\tilde{H}^{m-\frac{1}{2}}(\partial\Omega)}
 \eeq
where $L_t^2 \tilde{H}^s(\partial\Omega)$ is defined in \eqref{bdynorm}.
%One could also prove by the boundary condition of $u$ that:
In view of the  identity \eqref{normalofnormalder} and the boundary condition 
 \eqref{bdryconditionofu}, we have
 for $l\geq 1$
\beq\label{ZIomega}
\begin{aligned}
   |Z^I \omega|_{L_t^2{H}^{-\frac{1}{2}}(\partial\Omega)}\lesssim |u|_{L_t^2\tilde{H}^{m-\frac{1}{2}}}+&|Z^I(\partial_{\bn} u)|_{L_t^2{H}^{-\frac{1}{2}}} \lesssim | u|_{L_t^2\tilde{H}^{m-\frac{1}{2}}}+|Z^I \div  u|_{L_t^2H^{-\f{1}{2}}%\tilde{H}^{m-\frac{3}{2}
   }\\
   &\quad\lesssim \|u\|_{E_t^m}+\|\nabla\div u\|_{L_t^2H_{co}^{m-2}}.%+\|\nabla\div u\|_{L_t^2\mathcal{H}^{m-1}}.
 \end{aligned}
\eeq
Moreover,
if $Z^I=(\varepsilon\partial_t)^{m-1},$ we have by $L^2(\partial\Omega)\hookrightarrow H^{-\frac{1}{2}}(\partial\Omega)$ and the  trace inequality \eqref{traceL2} \beq\label{ZIomega1}
\ep^{\f{1}{2}}|Z^I\div u|_{L_t^2H^{-\frac{1}{2}}}\lesssim \|(\div u,\ep \nabla\div u)\|_{L_t^2\mathcal{H}^{m-1}}
\eeq
Collecting \eqref{nablaqpre}-\eqref{ZIomega1}, and using \eqref{nabladivu-4}, \eqref{es-f-bd}, 
 one obtains that:
 \beqs
 \begin{aligned}
 & \|\nabla q\|_{L_t^2\cH^{j,l}}+\ep^{\f{1}{2}}|\nabla q\|_{L_t^2\cH^{m-1}}\\
  &\lesssim \|f\|_{L_t^2{H}_{co}^{m-1}}+\|u\|_{E_t^m} +\|\nabla\div u\|_{L_t^2{H}_{co}^{m-2}}+\ep\|\nabla\div u\|_{L_t^2H_{co}^{m-1}}\lesssim \La\cE_{m,t}.
  \end{aligned}
 \eeqs

 We are now ready to prove \eqref{curlwLinfty}. 
 By using the equation $\eqref{reformulation}_3,$  the elliptic estimate \eqref{highconormal} and the product estimate \eqref{roughproduct1}, one finds:
 \begin{equation}\label{curlw-Linftypre}
     \begin{aligned}
       &\varepsilon \|\Delta v(t)\|_{H_{co}^{m-2}} + \varepsilon\|\nabla q(t)\|_{H_{co}^{m-2}} \\
       &\lesssim \|v(t)\|_{H_{co}^{m-1}}+\varepsilon\big\|f(t)\big\|_{H_{co}^{m-2}}+\varepsilon\sum_{|I|\leq m-2}|Z^I(\Delta v\cdot \bn)(t)|_{{H}^{-\frac{1}{2}}(\partial\Omega)}\\
     \end{aligned}
 \end{equation}
 With the aid of the boundary condition \eqref{bdryconditionofu}, the  identities
\eqref{normalofnormalder},  \eqref{omegatimesn} and  the estimates \eqref{sumstep1}, \eqref{nabladivu-51}, the boundary term can be treated  as,
 \begin{equation}\label{bdy-v-Linfty}
 \begin{aligned}
    &\quad\varepsilon\sum_{|I|\leq m-2}|Z^I(\Delta v\cdot \bn)|_{{H}^{-\frac{1}{2}}(\partial\Omega)} \\
    &\lesssim
   \varepsilon ( \|\nabla u(t)\|_{H_{co}^{m-2}}+\|u(t)\|_{H_{co}^{m-1}})+\varepsilon\|\nabla\div u(t)\|_{H_{co}^{m-2}}\\
  &\lesssim Y_m(\sigma_0,u_0)
+(T+\ep)^{\f{1}{4}}\La \cE_{m,t}.
 \end{aligned}
 \end{equation}
 Combined with \eqref{curlw-Linftypre} and 
 the fact that  $\Delta v=-\curl\omega,$ this yields   \eqref{curlwLinfty}. 
\end{proof}
\subsubsection{High order regularity estimates for $v$} %$\uppercase\expandafter{\romannumeral1}$}
This subsection is devoted to the 
high order estimates for $v: \|v\|_{L_t^{\infty}H_{co}^{m-1}}, \|\nabla v\|_{L_t^2H_{co}^{m-1}}$.
\begin{lem}\label{lemofv}
Suppose that \eqref{preasption} is satisfied, then for any $j+l\leq m-1,j,l\geq 0$ and for every $0<t\leq T,$ the following a-priori estimate holds:
\begin{equation}\label{EI-v}
    \begin{aligned}
   &\|v\|_{L_t^{\infty}\mathcal{H}^{j,l}}^2+\varepsilon^2 \|\nabla v\|_{L_t^{\infty}\mathcal{H}^{j,l}} ^2+\|\nabla v\|_{L_t^{2}\mathcal{H}^{j,l}}^2 +\varepsilon^2\|\curl \omega\|_{L_t^2\mathcal{H}^{j,l}}^2\\
   &\lesssim Y_{m}^2(\sigma_0,u_0)+(T+\ep)^{\f{1}{2}}\Lambda_{2,\infty,T}\cE_{m,T}^2+\|\div u\|_{L_t^2\mathcal{H}^{j,l}\cap L_t^2\mathcal{H}^{j+1,l-1}}^2
    \end{aligned}
\end{equation}
where we use the notation \eqref{not3}.
\end{lem}
\begin{rmk}
The estimate \eqref{EI-v} will be used later (see Lemma \ref{normalofnablasigma}) to get the high order spatial regularity for $\div u,$ which in turn, together with \eqref{EI-v}, gives the control of $v.$
\end{rmk}
\begin{proof}
In view of \eqref{hiddenes}, \eqref{sumstep1},
it suffices  to show that the left hand side of \eqref{EI-v} can be controlled by:
$$C(1/c_0)\big(Y_{m}^2(\sigma_0,u_0)+\cW_{m,T}^2+\|\div u\|_{L_t^2\mathcal{H}^{j,l}\cap L_t^2\mathcal{H}^{j+1,l-1}}^2\big)$$ where:
\beq\label{defofX}
\begin{aligned}
\cW_{m,T}^2&=\|u\|_{L_t^{\infty}\cH^{m-1}}^2+\|\nabla u\|_{L_t^2\cH^{m-1}}^2+\varepsilon^2\|\nabla u\|_{L_t^2H_{co}^{m}}^2+(T+\varepsilon)^{\frac{1}{2}}\La\cE_{m,t}.
\end{aligned}
 \eeq
This estimate will be obtained as the direct consequence of the following three inequalities:
\beq\label{highv-T}
 \|v\|_{L_t^{\infty}\mathcal{H}^{m-1}}^2+\|\nabla v\|_{L_t^{2}\mathcal{H}^{m-1}}^2\lesssim  \|u\|_{L_t^{\infty}\mathcal{H}^{m-1}}^2+\|\nabla u\|_{L_t^{2}\mathcal{H}^{m-1}}^2,
\eeq
\begin{equation}\label{highesofv}
\begin{aligned}
   \|v\|_{L_t^{\infty}\mathcal{H}^{j,l}}^2+\|\nabla v\|_{L_t^{2}\mathcal{H}^{j,l}}^2&\lesssim 
  \|v(0)\|_{H_{co}^{m-1}}^2+\|\nabla u\|_{L_t^2\cH^{m-1}}^2
   \\
    &\quad +\|\div u\|_{L_t^2\mathcal{H}^{j,l}}^2 +T^{\frac{1}{2}}\La \cE_{m,t}^2, \, l\geq 1,
\end{aligned}
\end{equation}
\begin{equation}\label{curlwL2}
\begin{aligned}
  \varepsilon^2 \|\nabla v\|_{L_t^{\infty}\mathcal{H}^{j,l}} ^2+\varepsilon^2\|\Delta v\|_{L_t^2\mathcal{H}^{j,l}}^2&\lesssim
  \ep^2 \|(\nabla v,v)(0)\|_{H_{co}^{m-1}}+\|\nabla v\|_{L_t^2\mathcal{H}^{j,l}\cap L_t^2\mathcal{H}^{j+1,l-1}}^2\\
  &\quad+\varepsilon^2\|\nabla u\|_{L_t^{2}H_{co}^{m}}^2+(T^{\f{1}{2}}+\ep)\La\cE_{m,t}^2.
\end{aligned}
\end{equation}
Note that since the Leray projector $\mathbb{P}$ commutes with $\ep\pt,$
one has that: $\mathbb{P}((\ep\pt)^j u)=(\ep\pt)^j v.$ Therefore, from the  continuity of the projection, we have:
$$\|v(0)\|_{H_{co}^{m-1}}\lesssim 
\|u(0)\|_{H_{co}^{m-1}}.$$

The inequality \eqref{highv-T} is a direct consequence of the definition of $v$ and the elliptic estimates in
Proposition \ref{propneumann}. We thus focus on the other two inequalities. Let us first prove \eqref{highesofv} and then sketch the proof of \eqref{curlwL2}.
By \eqref{eqofv0}, $v$ solves
\beq\label{eqofv}
\bar{\rho}\partial_t v-\mu \Delta v+\nabla q=-(\frac{g_2-\bar{\rho}}{\varepsilon}\varepsilon\partial_t u+g_2u\cdot\nabla u)=:f
\eeq
supplemented with the boundary conditions:
\begin{equation}\label{bdrycon-v}
\begin{aligned}
v\cdot \bn|_{\partial\Omega}=0, \quad \Pi(\partial_{\bn} v)%&=\Pi(-2a u+D \bn\cdot  \nabla\Psi+D \bn\cdot u )\\
=\Pi (-2a v+D \bn\cdot v)+2\Pi(-a \nabla\Psi +D \bn\cdot  \nabla\Psi).
\end{aligned}
\end{equation}
%where $D n\cdot u=\sum_{i}^3u_i\nabla n_i,$ $\nabla\Psi=\mathbb{Q}u.$
We apply  $Z^{I}$ to  the equation  \eqref{eqofv} with 
$I=(j,I'), 0\leq j+|I'|=j+l=k\leq m-1, |I'|\geq 1.$ Taking the scalar product  by $Z^{I}v,$ and then integrating in space and time, we   get that:
\begin{equation}\label{sec5:eq12}
\begin{aligned}
\f{1}{2}\bar{\rho}\int_{\Omega}|Z^{I}v(t)|^2\ \d x
&\leq \f{1}{2}\bar{\rho} \int_{\Omega} |(Z^I v)(0)|^2\ \d x+\mu\iint_{Q_t}Z^{I}(\Delta v) Z^{I}v \ \d x\d s\\
&\quad
+\|Z^{I}v\|_{L^2(Q_t)}\big(\|\nabla q\|_{L_t^2\cH^{j,l}}+\|f\|_{L_t^2H_{co}^{m-1}}\big).
\end{aligned}
\end{equation}
By \eqref{es-f-bd} and \eqref{esofnablaq}, the second line in the above inequality can be bounded as:
\beq 
\begin{aligned}
\|Z^{I}v\|_{L^2(Q_t)}\big(\|\nabla q\|_{L_t^2\cH^{j,l}}+\|f\|_{L_t^2H_{co}^{m-1}}\big)&\lesssim T^{\f{1}{2}}\|u\|_{L_t^{\infty}H_{co}^{m-1}}\La \cE_{m,t}\\
&\lesssim T^{\f{1}{2}}\La \cE_{m,t}^2.
\end{aligned}
\eeq
It remains to control the second term in the right hand side of 
\eqref{sec5:eq12}, which is the following task. We split it into three terms:
\begin{equation}\label{lapofv}
\begin{aligned}
\mu\iint_{Q_t}Z^{I}(\Delta v)\cdot Z^{I}v\ \d x\d s&=\mu\iint_{Q_t}[Z^{I},\div]\nabla v \cdot Z^{I}v \, \d x\d s-\mu\iint_{Q_t}Z^{I}\nabla v\cdot \nabla Z^{I}v \ \d x\d s\\
&\qquad+\mu\int_0^t\int_{\partial\Omega} Z^{I}\nabla v\cdot \bn \,Z^{I}v\ \d S_y\d s=\colon \mathcal{T}_1+\mathcal{T}_2+\mathcal{T}_3.%(1)+(2)+(3). %\eqref{lapofv}_1+\eqref{lapofv}_2+\eqref{lapofv}_3.
\end{aligned}
\end{equation}
The estimate of $ \mathcal{T}_1-\mathcal{T}_3$ will be 
similar to that of $\mathcal{K}_1-\mathcal{K}_4$ in \eqref{idoflaplacian}.

We first estimate $ \mathcal{T}_2.$ By integrating by parts, one has that:
\begin{equation}
\begin{aligned}
  \mathcal{T}_2&=-\mu\iint_{Q_t}|Z^{I}\nabla v|^2\, \d x\d s-\mu\iint_{Q_t}Z^{I}\nabla v[\nabla,Z^{I}]v \,\d x\d s \\
 &\leq -\frac{\mu}{2}\|Z^{I}\nabla v\|_{L^2(Q_t)}^2+\frac{\mu}{2}\|[\nabla,Z^I]v\|_{L^2(Q_t)}^2\leq -\frac{\mu}{2}\|Z^{I}\nabla v\|_{L^2(Q_t)}^2+C\|\nabla v\|_{L_t^2\mathcal{H}^{j,l-1}}^2.
\end{aligned}
\end{equation}
Note that in the last estimate, by \eqref{comu}, we know that $[\nabla,Z^I]v$  involves only lower order ($\leq k-1$) conormal derivatives of $\nabla v.$  

We now switch to the estimate of the boundary term $ \mathcal{T}_3$ in
\eqref{lapofv}, which vanishes if $Z^{I}$
involves at least one weighted normal derivative $Z_3^{i}$.
We thus assume that $Z^{I}$ contains only time derivatives and spatial tangential derivatives.
\begin{equation*}
 \begin{aligned}
   \mathcal{T}_{3}&=-\mu\int_0^t\int_{\partial\Omega}\big(-[Z^I,\bn]
  \nabla v \cdot Z^{I}v+[Z^{I},\bn]\cdot \partial_{\bn} v (Z^{I}v\cdot \bn)+[Z^{I},\Pi]\partial_{\bn} v\cdot\Pi Z^{I}v\big)\,\d S_y\d s\\
  &\quad+\mu\int_0^t\int_{\partial\Omega}\big(Z^{I}(\partial_{\bn} v\cdot \bn) (Z^{I}v\cdot \bn)+ Z^{I}(\Pi\partial_{\bn} v)\cdot\Pi Z^I v\big) \ \d S_y\d s=\colon  \mathcal{T}_{31}+ \mathcal{T}_{32}.
 \end{aligned} 
\end{equation*}
The first term $ \mathcal{T}_{31}$
can be dealt with thanks to  H\"older inequality and the trace inequality 
\eqref{traceL2}
$$
 \begin{aligned}
     \mathcal{T}_{31}&\lesssim \int_0^t |(\varepsilon\partial_t)^j\nabla v(s)|_{{H}^{l-1}(\partial\Omega)}|Z^{I}v(s)|_{L^2(\partial\Omega)}\, \d s\\
     &\lesssim \int_0^t (|(\varepsilon\partial_t)^j v|_{{H}^l(\partial\Omega)}+|(\varepsilon\partial_t)^j \nabla\Psi|_{{H}^l(\partial\Omega)})|Z^I v|_{L^2(\partial\Omega)}\, \d s\\
     &\leq \delta\mu \|\nabla v\|_{L_t^2\mathcal{H}^{j,l}}^2+C(\delta,\mu)\|(u,\div u)\|_{L_t^2H_{co}^k}^2. 
 \end{aligned}   
$$
Note that in the second inequality, we have used  the boundary condition \eqref{bdrycon-v} and the
identity (since $\div v=0$):
\beq\label{usefulidentity}
\partial_{\bn} v\cdot \bn =-(\Pi \partial_{y^1}v)^1-(\Pi \partial_{y^2}v)^2,
\eeq
to obtain that:
\beq\label{gradonbdy}
|(\varepsilon\partial_t)^j\nabla v(s)|_{{H}^{l-1}}\lesssim |(\varepsilon\partial_t)^j v(s)|_{{H}^{l}}+|(\varepsilon\partial_t)^j\nabla\Psi(s)|_{{H}^{l}}.
\eeq
For the second term $ \mathcal{T}_{32},$ since %$Z^{I}v\cdot n|_{\partial\Omega}=0$ if $Z^I=(\varepsilon\partial_t)^k,$ 
$l\geq 1,$ we might as well assume that $Z^{I}=\partial_{y} Z^{\tilde{I}},$ where $\p_y=\p_{y^1}$ or $\p_{y^2}.$ 
In view of the boundary condition
 \eqref{bdrycon-v} and  the identity \eqref{usefulidentity}, %as well as  $Z^I v\cdot \bn=[Z^I,\bn]v,$
 we have by integrating by parts along the boundary that:
\begin{equation}
    \begin{aligned}
    \mathcal{T}_{32}&=\int_0^t\int_{\partial\Omega}
    Z^{\tilde{I}}(\partial_{\bn} v\cdot \bn) \partial_{y}\cdot([Z^{I},\bn\cdot] v)+Z^{I}(\Pi\partial_{\bn} v)\Pi Z^I v) \, \d S_y\d s\\
    &\lesssim \int_0^t 
    |(\varepsilon\partial_t)^j v|_{{H}^{l}(\partial\Omega)}^2+|(\varepsilon\partial_t)^j(v,\nabla\Psi)|_{{H}^{l}(\partial\Omega)}|(\varepsilon\partial_t)^j v|_{{H}^{l}(\partial\Omega)}\,\d s\\
  %  &\lesssim %\int_0^t  (\|v\|_{L_{t}^2H_{co}^{k}}+\|\nabla v\|_{L_{t}^2H_{co}^{k}})\|v\|_{L_{t}^2H_{co}^{k}}+\|(u,\nabla\Psi)\|_{L_t^2H_{co}^{k}}(\|v\|_{L_{t}^2H_{co}^{k}}+\|\nabla v\|_{L_{t}^2H_{co}^{k}})\\
  & \lesssim \delta\mu \|\nabla v\|_{L_t^2\mathcal{H}^{j,l}}^2 +C(\delta,\mu)\|(u,\div u)\|_{L_t^2\mathcal{H}^{j,l}}^2.
    \end{aligned}
\end{equation}
It remains  to control $\mathcal{T}_1$. %Integration by parts, the facts 
Owing to \eqref{comu} and
\eqref{gradonbdy}, one obtains again
by integrating by parts that:
\begin{equation}\label{sec5:eq13}
\begin{aligned}
  \mathcal{T}_1&\lesssim \|\nabla v\|_{L_t^2\mathcal{H}^{j,l-1}}(\|v\|_{L_t^2\mathcal{H}^{j,l}}+\|\nabla v\|_{L_t^2\mathcal{H}^{j,l}})+
   |(\varepsilon\partial_t)^j\nabla v(s)|_{{H}^{l-1}(\partial\Omega)}|v|_{{H}^{l}(\partial\Omega)}\\
   &\lesssim \delta\mu \|\nabla v\|_{L_t^2\mathcal{H}^{j,l}}^2+C(\delta,\mu)(\|(u,\div u)\|_{L_t^2\mathcal{H}^{j,l}}^2+\|\nabla v\|_{L_t^2\mathcal{H}^{j,l-1}}^2).
\end{aligned}
\end{equation}
Plugging \eqref{lapofv}-\eqref{sec5:eq13} into \eqref{sec5:eq12} and summing up for all $I=(j,I'),|I'|=l,$ one has by choosing $\delta$ small enough that
\begin{equation}\label{caselgeq1}
    \begin{aligned}
      \|%(\varepsilon\partial_t)^j 
      v(t)\|_{\mathcal{H}^{j,l}}^2+\frac{\mu}{4}\|\nabla v\|_{L_t^2\mathcal{H}^{j,l}}^2&\leq\|v(0)\|_{\mathcal{H}^{j,l}}^2+%C\delta\mu \|\nabla v\|_{L_t^2H_{co}^{k}}^2+
      C(\delta,\mu)\|\nabla v\|_{L_t^2\mathcal{H}^{j,l-1}}^2+\|\div u\|_{L_t^2\mathcal{H}^{j,l}}^2\\
      &\qquad +T^{\f{1}{2}}\La \cE_{m,t}^2.
    \end{aligned}
\end{equation}
In view of inequalities \eqref{highv-T} and \eqref{caselgeq1}, we obtain \eqref{highesofv} by induction on $l.$

We are now in position to prove \eqref{curlwL2}. As before, 
we apply  $Z^I$ to  the equation  \eqref{eqofv} for $v$  and we take the scalar product  by $-\varepsilon^2Z^I\Delta v.$ One gets by integration by parts and by using  Young's inequality that:
\begin{equation}\label{EIcurlw}
\begin{aligned}
  &\frac{1}{2}\bar{\rho} \varepsilon^2\int_{\Omega}|\nabla Z^I v(t)|^2\ \d x+\frac{\mu}{2}\ep^2\iint_{Q_t}|Z^I(\Delta v)|^2\ \d x\d s\\
  &\leq \frac{1}{2}\bar{\rho} \varepsilon^2\int_{\Omega}|\nabla Z^I v(0)|^2\,\d x+\varepsilon \iint_{Q_t}\varepsilon\partial_t Z^I v\cdot [Z^I,\Delta]v\, \d x\d s \\
  &+\varepsilon\int_0^t\int_{\partial\Omega}\varepsilon\partial_t Z^I v\cdot\partial_{\bn} Z^I v\, \d S_{y}\d s+C_{\mu}\varepsilon^2 \|(\nabla q,f)\|_{L_t^2H_{co}^{m-1}}^2.
\end{aligned}
\end{equation}
By induction, the following identity (up to some coefficients that depends on $\phi,\varphi$ and their derivatives up to order $m$) holds: $$[Z^{I},\Delta]=\sum_{\substack{|\tilde{I}|\leq |I|-1,|J|\leq|I|-1\\\tilde{I}_0=j,J_0=j}}\sum_{i,k=1}^3 (* Z^{\tilde{I}}\partial_{ik}^2+* Z^{J}\partial_k).$$
This identity, combined with elliptic regularity theory yields:
\begin{equation*}
\begin{aligned}
    %\varepsilon
    \|[Z^{I},\Delta]v\|_{L^2(Q_t)}\lesssim %\sum_{|\tilde{I}|\leq |I|-1}\|Z^{\tilde{I}}
 \|\nabla ^2 v\|_{L_t^2\mathcal{H}^{j,l-1}}+   \|\nabla  v\|_{L_t^2\mathcal{H}^{j,l-1}}&\lesssim  \|\Delta v\|_{L_t^2\mathcal{H}^{j,l-1}}+ |\partial_{\bn}(\varepsilon\partial_t)^j v|_{H^{l-\frac{1}{2}}}\\
  &\lesssim   \|\Delta v\|_{L_t^2\mathcal{H}^{j,l-1}}+\|(u,\nabla u)\|_{L_t^2\mathcal{H}^{j,l}}.
\end{aligned}
\end{equation*}
Note that in the last inequality, we have used \eqref{gradonbdy} and the trace inequality \eqref{traceL2}.
We thus control the second term in \eqref{EIcurlw} as follows:
\begin{equation}\label{curlwcomu}
\varepsilon \iint_{Q_t}\varepsilon\partial_t Z^I v\cdot [Z^I,\Delta]v\,\d x\d s\lesssim
   \varepsilon^2\|\Delta v\|_{L_t^2\mathcal{H}^{j,l-1}}^2+\|\varepsilon\partial_t v\|_{L_t^2\mathcal{H}^{j,l}}^2+\varepsilon \|u\|_{E_t^m}^2.
\end{equation}
Moreover, the third term of \eqref{EIcurlw} can be dealt with by arguments very similar  to the ones for  $\mathcal{T}_3:$ %(we omit the detail) 
\begin{equation}\label{curlwbdy}
\begin{aligned}
   & \varepsilon\int_0^t\int_{\partial\Omega}\varepsilon\partial_t Z^I v\cdot\partial_{\bn} Z^I v \, \d S_{y}\d s\\
    &\lesssim  \varepsilon\int_0^t |Z^I \varepsilon\partial_t v|_{L^2}\big(|(\varepsilon\partial_t)^j v|_{H^{l+1}}+|(\varepsilon\partial_t)^j\nabla\Psi|_{H^{l}}\big)\, \d s\\
&\lesssim \varepsilon (\|\nabla v\|_{L_t^2\mathcal{H}^{j+1,l}}^{\frac{1}{2}}\|v\|_{L_t^2\mathcal{H}^{j+1,l}}^{\frac{1}{2}}+\|v\|_{L_t^2\mathcal{H}^{j+1,l}})\cdot\\
&\qquad\quad(\|\nabla v\|_{L_t^2\mathcal{H}^{j,l+1}}^{\frac{1}{2}}\|v\|_{L_t^2\mathcal{H}^{j,l+1}}^{\frac{1}{2}}+\|v\|_{L_t^2\mathcal{H}^{j,l+1}}+\|u,\div u\|_{L_t^2\mathcal{H}^{j,l}})\\
&\lesssim \varepsilon^2 \|\nabla v\|_{L_t^2H_{co}^m}^2+ \varepsilon\|(u,\div u)\|_{L_t^2H_{co}^{m-1}}^2%+\|v\|_{L_t^2\mathcal{H}^m}^2
+\|\nabla v\|_{L_t^2\mathcal{H}^{j+1,l-1}\cap L_t^2\mathcal{H}^{j,l}}^2+\| v\|_{L_t^2\cH^{m}}^2
\end{aligned}
\end{equation}
Inserting \eqref{curlwcomu} and \eqref{curlwbdy} into \eqref{EIcurlw}, and use \eqref{es-f-bd}, \eqref{esofnablaq} to find 
$$\ep^2 \|(\nabla q, f)\|_{L_t^2H_{co}^{m-1}}^2\lesssim \ep \La\cE_{m,t}^2,$$
we obtain \eqref{curlwL2} by induction.
\end{proof}
\subsection{Step 3: Uniform estimates for $(\nabla\sigma,\div u)$}\label{step3}

In this subsection, we aim to get uniform 
control of higher spatial conormal derivatives of $(\nabla\sigma,\div u).$ More precisely, we prove uniform boundedness of $\|(\nabla\sigma,\div u)\|_{L_t^{\infty}H_{co}^{m-2}\cap L_t^2H_{co}^{m-1}}.$ This will be achieved by using the equation iteratively. 
\begin{lem}\label{normalofnablasigma}
Assume that \eqref{preasption} holds,
we then have that for every $0<t\leq T,$
\beq\label{recovernablasi}
\|(\nabla\sigma,\div u)\|_{L_t^{\infty}H_{co}^{m-2}\cap L_t^2H_{co}^{m-1}}^2\lesssim Y^2_m(\sigma_0,u_0)+(T+\ep)^{\f{1}{2}}
\cE^2_{m,T}\Lambda(\f{1}{c_0},\cA_{m,T}).
\eeq
\end{lem}
\begin{proof}
We will prove the following two inequalities:

$\bullet$ $L_t^2H_{co}^{m-1}$ estimate: for any $j,k\geq 0, j+k\leq m-1$:
\begin{equation}\label{inductionL2}
\begin{aligned}
\|(\nabla\sigma,\div u)\|_{L_t^2\mathcal{H}^{j,k}}&\lesssim Y_m(\sigma_0,u_0)+
T^{\frac{1}{2}}\|(u,\sigma)\|_{L_t^{\infty}\mathcal{H}^{m}}
\\
&\quad+\varepsilon\|\nabla\div u\|_{L_t^2H_{co}^{m-1}}+(T+\varepsilon)^{\f{1}{4}}\Lambda(\f{1}{c_0},\cA_{m,T})\cE_{m,T}.
\end{aligned}
\end{equation}

$\bullet$ $L_t^{\infty}H_{co}^{m-2}$ estimate: for any $j,l\geq 0$ and $j+l\leq m-2:$
\begin{equation}\label{inductionLinfty}
\begin{aligned}
&\|(\nabla\sigma,\div u)\|_{L_t^{\infty}\mathcal{H}^{j,l}}
\lesssim 
Y_m(\sigma_0,u_0)+\varepsilon \|(\nabla\div u,\curl \omega)\|_{L_t^{\infty}H_{co}^{m-2}}+\| v\|_{L_t^{\infty}H_{co}^{m-1}}\\
&\qquad\qquad\quad+\|(\sigma, u)\|_{L_t^{\infty}\mathcal{H}^{m-1}}+\ep\|\nabla\sigma\|_{L_t^{\infty}H_{co}^{m-1}} +\ep\Lambda(\f{1}{c_0},\cA_{m,T}) \cE_{m,T}.
\end{aligned}
\end{equation}
These two inequalities, combined with
the estimates \eqref{sumstep1}, \eqref{curlwLinfty}, \eqref{EI-v} 
and the definition \eqref{defofLambdam}, 
yield \eqref{recovernablasi}. 

The inequality \eqref{inductionL2} can be obtained  by induction on the number of space conormal derivatives. Let us first prove  \eqref{inductionL2} for $k=0,j\leq m-1.$ 
By \eqref{rewriteofsigma} 
and product estimate \eqref{roughproduct1}, we find that:
\beq\label{indu-div0}
\|\div u\|_{L_t^2\cH^{m-1}}\lesssim T^{\f{1}{2}}\|\sigma\|_{L_t^{\infty}\cH^m}+\ep\La \cE_{m,t}.
\eeq
Moreover, by the equations $\eqref{NS1}_2$ for $u$,
\beq\label{rewriteofu}
\nabla\sigma=-\bar{\rho}\varepsilon\partial_t u+\ep f-\varepsilon\mu\curl \omega+\varepsilon(2\mu+\lambda)\nabla\div u,
\eeq
we thus have by \eqref{es-f-bd}, \eqref{EI-v} that:
\beq\label{indu-si-0}
\begin{aligned}
\|\nabla\sigma\|_{L_t^2\cH^{m-1}}&\lesssim \|u\|_{L_t^2\cH^{m}}+\ep\|\curl\omega\|_{L_t^2\cH^{m-1}}+\ep\|\nabla\div u\|_{L_t^2H_{co}^{m-1}}+\ep\La\cE_{m,t}\\
&\lesssim T^{\f{1}{2}}\|u\|_{L_t^{\infty}\cH^{m}}+\|\div u\|_{L_t^2\cH^{m-1}}+Y_{m}(\sigma_0,u_0)\\
&\qquad+\ep\|\nabla\div u\|_{L_t^2H_{co}^{m-1}}+(T+\ep)^{\f{1}{4}}\La\cE_{m,t},
\end{aligned}
\eeq
which, together with \eqref{indu-div0}, yields \eqref{inductionL2} for $k=0,j\leq m-1.$

 Now suppose that \eqref{inductionL2} holds for $k=k_0-1$ with $
 k_0\geq 1,$ it suffices to prove that it is also true for $k=k_0$ and  for every $j$ such that $j+k_{0} \leq m-1.$
We begin with the estimate of $\div u,$ which again follows from the equation \eqref{rewriteofsigma} 
and product estimate \eqref{roughproduct1}:
\begin{equation}
\begin{aligned}
  \|\div u\|_{L_t^2\mathcal{H}^{j,k_0}}&\lesssim \|\varepsilon\partial_t\sigma\|_{L_t^2\mathcal{H}^{j,k_0}}+\varepsilon\Lambda(\f{1}{c_0},\cA_{m,t})\cE_{m,t}\\
  &\lesssim \|(\sigma,\nabla\sigma)\|_{L_t^2\mathcal{H}^{j+1,k_0-1}}+\Lambda\big(\f{1}{c_0},\cA_{m,t}\big)\cE_{m,t}\lesssim \text{R.H.S of } \eqref{inductionL2}.
\end{aligned}
\end{equation}
Next, one gets by equation \eqref{rewriteofu}, estimate \eqref{EI-v} and the  induction hypothesis that:
\begin{equation*}
\begin{aligned}
    \|\nabla\sigma\|_{L_t^2\mathcal{H}^{j,k_0}}&\lesssim 
    \|u\|_{L_t^2\mathcal{H}^{j+1,k_0}}+\ep\|\curl\omega\|_{L_t^2\cH^{j,k_0}}+\varepsilon\|\nabla\div u\|_{L_t^2H_{co}^{m-1}}+\ep\Lambda(\f{1}{c_0},\cA_{m,T})\cE_{m,t}\\
    &\lesssim \| (\div u,\nabla v)\|_{L_t^2\mathcal{H}^{j+1,k_0-1}}
+\ep\|\curl\omega\|_{L_t^2\cH^{j,k_0}}+\varepsilon\|\nabla\div u\|_{L_t^2H_{co}^{m-1}}+\ep\Lambda(\f{1}{c_0},\cA_{m,T})\cE_{m,t}\\
    &\lesssim \text{R.H.S of } \eqref{inductionL2}.
\end{aligned}
\end{equation*}

Let us switch to the proof of \eqref{inductionLinfty}. By similar argument as in the derivation of \eqref{indu-div0}, \eqref{indu-si-0},
one can find that:
\beq
\|(\nabla\sigma,\div u)\|_{L_t^{\infty}\cH^{m-2}}\lesssim \|(\sigma,u)\|_{L_t^{\infty}\cH^{m-1}}+\ep\|(\nabla\div u,\curl\omega) \|_{L_t^{\infty}H_{co}^{m-2}}+\ep\La\cE_{m,t},
\eeq
which proves \eqref{inductionLinfty}
for $l=0.$
 Suppose that it is true for  $ l=l_0-1\leq m-3,$ 
we show that it also holds for $l=l_0$
and for any $j,$ such that $j+l_0\leq m-2.$
Let us start with the estimate of $\div u.$ It follows from the equation 
\eqref{rewriteofsigma}, the product estimate \eqref{roughproduct1} and
the induction hypothesis that:
\begin{equation*}
\begin{aligned}
    \|\div u\|_{L_t^{\infty}\mathcal{H}^{j,l_0}}&\lesssim  \|\varepsilon\partial_t\sigma\|_{L_t^{\infty}\mathcal{H}^{j,l_0}}+\ep\La\cE_{m,t}\\
    &\lesssim \|(\sigma,\nabla\sigma)\|_{L_t^{\infty}\mathcal{H}^{j+1,l_0-1}} +\ep\La\cE_{m,t}\\
    &\lesssim \|\sigma\|_{L_t^{\infty}\mathcal{H}^{m-2}}+\|\nabla\sigma\|_{L_t^{\infty}\mathcal{H}^{j+1,l_0-1}}+\ep\La\cE_{m,t}\\
   &\lesssim \text{R.H.S of } \eqref{inductionLinfty}.
   \end{aligned}
\end{equation*}
For the estimate of $\nabla\sigma,$ we use the  equation \eqref{rewriteofu} and the product estimate \eqref{roughproduct1} to obtain:
\beq\label{sec3:eq10}
\begin{aligned}
&\quad\|\nabla\sigma\|_{L_t^{\infty}\cH^{j,l_0}}\\
&\lesssim \|\ep\pt u\|_{L_t^{\infty}\mathcal{H}^{j,l_0}}+\ep\|(\nabla\div u,\curl\omega)\|_{L_t^{\infty}H_{co}^{m-2}}+\ep^{\f{1}{2}}\La\cE_{m,t}.
\end{aligned}
\eeq
It remains  to bound $\|\varepsilon\partial_t u\|_{L_t^{\infty}\mathcal{H}^{j,l_0}}$.  
We use that for $j+l_{0} \leq m-2$, 
\begin{equation}\label{sec3:eq11}
  \begin{aligned}
  \|\varepsilon\partial_t u\|_{L_t^{\infty}\mathcal{H}^{j,l_0}}&\lesssim%\|u\|_{L_t^{\infty}\mathcal{H}^{m-1}}+
  \|v\|_{L_t^{\infty}H_{co}^{m-1}}+\|(\nabla\Psi,\nabla^2\Psi) \|_{L_t^{\infty}\mathcal{H}^{j+1,l_0-1}}\\
%&\lesssim \|u\|_{L_t^{\infty}\mathcal{H}^{m-1}}+ \|\nabla^2\Psi\|_{L_t^{\infty}\mathcal{H}^{j+1,l_0-1}}+\|\nabla v\|_{L_t^{\infty}H_{co}^{m-2}}\\
  &\lesssim %\|u\|_{L_t^{\infty}\mathcal{H}^{m-1}}+
  \|v\|_{L_t^{\infty}{H}_{co}^{m-1}}+ \|(u,\div u)\|_{L_t^{\infty}\mathcal{H}^{j+1,l_0-1}}\\
 & \lesssim \|u\|_{{L_t^{\infty}\mathcal{H}^{m-2}}}+
 \|v\|_{L_t^{\infty}{H}_{co}^{m-1}}+\sum_{k=1}^{l_0}\|\div u\|_{{L_t^{\infty}\mathcal{H}^{j+k,l_0-k}}}.
  \end{aligned}  
\end{equation}
Plugging \eqref{nabladivu-5} and \eqref{sec3:eq11} into \eqref{sec3:eq10} and 
using the  induction hypothesis, we get that:
\beqs
\|\nabla\sigma\|_{L_t^{\infty}\cH^{j,l_0}}\lesssim \text{R.H.S of } \eqref{inductionLinfty}.
\eeqs
We thus proved that \eqref{inductionLinfty} holds for $j+1,l_0$ which ends the proof.
\end{proof}
\begin{rmk}
By Lemmas \ref{lemofv}, \ref{normalofnablasigma}, we get that:
\begin{align}\label{L2L2bd}
\|(\sigma,u)\|_{E_t^{m}}^2\lesssim 
Y_m^2(\sigma_0,u_0)+(T+\ep)^{\f{1}{2}}\Lambda\big(\f{1}{c_0},\cN_{m,t}\big).
\end{align}
\end{rmk}
\subsection{Step 4: Uniform estimates for the  gradient of  the velocity.}
In this section, we will bound $\|\nabla v\|_{L_t^{\infty}H_{co}^{m-2}}$, which, combined with \eqref{Neumann} \eqref{inductionLinfty}, gives the control of $\|\nabla u\|_{L_t^{\infty}H_{co}^{m-2}}.$
\begin{lem}\label{lemofnablav}
Suppose that \eqref{preasption}
 holds, then for any $0<t\leq T,$ we have the following estimate, 
\begin{equation}\label{nablavLinfty}
    \begin{aligned}
  \|\nabla v\|_{L_t^{\infty}H_{co}^{m-2}}^2&\lesssim 
  %\|(\omega\times n)(0)\|_{H_{co}^{m-2}}
  Y^2_m(\sigma_0,u_0)+\|v\|_{L_t^{\infty}H_{co}^{m-1}}^2+T^{\f{1}{2}}\Lambda\big(\f{1}{c_0},\cN_{m,t}\big).
    \end{aligned}
\end{equation}
\end{lem}
\begin{proof}
Since in the interior domain, 
the conormal spatial derivatives are equivalent to the standard spatial derivatives, we only have   to estimate $\nabla v$ near the boundary, say
$\|\chi_{i} \nabla v\|_{L_t^{\infty}H_{co}^{m-2}}$   where
 $\chi_i, (i=1\cdots N)$ are smooth functions associated to the covering \eqref{covering1} and are compactly supported in $\Omega_i.$
% {\color{red} for the $L^\infty$ estimate you just took  $\chi$ any function supported close to the boundary. Yes,I changed them in the $L^\infty$ estimate .}
Close to the boundary, it follows from the   identity \eqref{usefulidentity} 
and the following identity
\beqs
\Pi(\partial_{\bn} v)=\Pi ((\nabla  v-D v)\bn)+\Pi((D v)\bn)=\Pi(\omega\times \bn)+\Pi(-(D \bn) v)
\eeqs
that:
\beqs
\begin{aligned}
\|\chi_{i}\nabla v\|_{L_t^{\infty}H_{co}^{m-2}}&\lesssim
\|\chi_{i}\Pi(\partial_{\bn} v)\|_{L_t^{\infty}H_{co}^{m-2}}+\|v\|_{L_t^{\infty}H_{co}^{m-1}}\\
&\lesssim \|\chi_i(\omega\times\bn)\|_{L_t^{\infty}H_{co}^{m-2}}+\|v\|_{L_t^{\infty}H_{co}^{m-1}}.
\end{aligned}
\eeqs
 We thus reduce the problem to the estimate of $\chi_{i}(\omega\times \bn),$ which is the aim of the following lemma. 
 
%is the smooth function supported on the vicinity of the boundary and shall be chosen later).
%Moreover, as $\|\partial_y v\|_{L_t^{\infty}H_{co}^{m-2}}\lesssim \|v\|_{L_t^{\infty}H_{co}^{m-1}}$ which has been controlled by last lemma, we are left to estimate $\|\chi_{i}\partial_{\bn} v\|_{L_t^{\infty}H_{co}^{m-2}}.$
%Further,, it remains for us to control $\|\chi_{i} \Pi(\partial_{\bn} v)\|_{L_t^{\infty}H_{co}^{m-2}}.$ Nevertheless, by using the identity:
% that $\Pi\nabla$ involves only tangential derivatives%
\end{proof}

\begin{lem}\label{omeganLinfty}
Under the assumption \eqref{preasption}, the following estimate holds: for every $0<t\leq T,$
\begin{equation}\label{ineqofomegatimesn}
    \begin{aligned}
  \|\chi_i(\omega\times \bn)\|_{L_t^{\infty}H_{co}^{m-2}(\Omega)}^2&\lesssim 
  \|\chi_i(\omega\times \bn)(0)\|_{H_{co}^{m-2}}^2
  +(T+\ep)^{\frac{1}{2}}\Lambda\big(\f{1}{c_0},\cN_{m,t}\big).
    \end{aligned}
\end{equation}
where $\chi_{i}$ is a smooth function   compactly supported in $\Omega_i.$
\end{lem}
\begin{proof}
Note that the important feature of $\chi_i(\omega\times \bn)$ is that: it solves  a transport-diffusion system \textbf{without} singular terms, with a non-homogeneous Dirichlet boundary condition.
 In order to perform the estimate, we   split the system for  $\chi_i(\omega\times \bn)$ into two parts, one which just 
 solves the  heat equation with the  nontrivial Dirichlet boundary condition and a remainder which is amenable to energy
 estimates since it satisfies a convection-diffusion equation with homogeneous Dirichlet boundary condition.
  To deal with the first  system,
the explicit formula for heat equation will play an important role.
It is thus helpful  to transform the problem to the half-space.

Let us set  $\eta_i=\chi_i\omega\times \bn, i\geq 1.$
Direct computations show that $\omega$ solves the following system:
\beq\label{defGw}
g_2\partial_t \omega+g_2 u\cdot\nabla\omega-\mu\Delta\omega=g_2\omega\cdot\nabla u-g_2\omega\div u-\frac{\nabla g_2}{\varepsilon}\times (\varepsilon\partial_t u+\varepsilon u\cdot\nabla u)=\colon G^{\omega}
\eeq
from which we obtain the equations satisfied by $\eta_i$ (which is compactly supported in $\Omega_{i}$)
\begin{equation}\label{eqofeta}
\left\{
\begin{aligned}
&\bar{\rho}\partial_t \eta_i -\mu \Delta \eta_i=F_{i}^{\omega}  \quad \text{in}\quad \Omega_i\cap\Omega.\\
&\eta_i=\chi_i \Pi(\omega\times \bn)=2\chi_i\Pi
(-a u+(D \bn) u )
\quad \text{on} \quad \Omega_i\cap \partial\bar{\Omega}, 
\end{aligned}
\right.
\end{equation}
where
$$F_{i}^{\omega}=\colon -\Delta(\chi_i \bn)\times\omega-2\nabla\omega\times\nabla(\chi_i \bn)-(g_2u\cdot\nabla \omega)\times(\chi_i \bn)+\frac{\bar{\rho}-g_2}{\varepsilon}\varepsilon\partial_t\omega\times (\chi_i \bn)+G^{\omega}\times(\chi_i \bn).$$
Since we will use the local coordinate \eqref{local coordinates},
it is useful to know the expressions of Laplacian in this new coordinates.
By direct computation, we find that: 
\beq\label{change-variable}
(\nabla f)\circ \Phi_i= P{\nabla} (f\circ{\Phi_i}), \quad (\div F)\circ \Phi_i={\div}\big(P^{*}(F\circ\Phi_i)\big)\quad (\Delta f)\circ \Phi_i={\div}(E{\nabla} \big(f\circ\Phi_i)\big)
\eeq
where ${\nabla}=(\p_{y^1},\p_{y^2},\p_z)^{t}, {\div}=({\nabla})^{*}$
 represent the gradient and the divergence in the new coordinates and
\beq\label{defPE}
\left(\begin{array}{ccc}
    1 &0&-\p_{y^1}\varphi_i  \\
     0& 1& -\p_{y^2}\varphi_i\\
     0&0 &1
\end{array}
\right),\quad
E=P^{*}P=\left(\begin{array}{ccc}
    1 &0&-\p_{y^1}\varphi_i  \\
     0& 1& -\p_{y^2}\varphi_i\\
    -\p_{y^1}\varphi_i &-\p_{y^2}\varphi_i & |\bN|^2
\end{array}
\right).
\eeq
Let us set $\tilde{\eta_{i}}(t,y,z)= \eta_{i}(t, \Phi_{i}(y,z)):= (\eta_i\circ {\Phi}_{i})
(y,z), (y,z)\in {\Phi}_{i}^{-1}(\Omega_{i}\cap\bar{\Omega}).$ Denote also
$\widetilde{F_{i}^{\omega}}=F_{i}^{\omega}\circ {\Phi}_{i}.$
Since $\Supp \chi_i|_{\bar{\Omega}}\Subset \Omega_i\cap\bar\Omega,$
We can extend the definition of  $\tilde{\eta_i}$ and $\widetilde{F_{i}^{\omega}}$ from ${\Phi}_{i}^{-1}(\Omega_{i}\cap\bar{\Omega})$ to $\overline{\mathbb{R}_{+}^3}$ by zero extension, which are still denoted by $\tilde{\eta_i},$
$\widetilde{F_{i}^{\omega}}$.
Consequently, by \eqref{eqofeta} and \eqref{change-variable},
we find that $\tilde{\eta}_i$ satisfies:
\begin{equation}\label{eqoftildeeta}
\left\{
\begin{aligned}
&\bar{\rho}\partial_t \tilde{\eta_i} -\mu{\div}(E{\nabla}\tilde{\eta_i}\big)=
{F_{ni}^{\omega}} \quad \text{in}\quad \mathbb{R}_{+}^3.\\
&\tilde{\eta_i}|_{z=0}=2[\chi_i\Pi
(-a u+(D \bn) u )]\circ {\Phi_i}\big|_{z=0}.
\end{aligned}
\right.
\end{equation}
 Let us set $\mathcal{Z}_0=\varepsilon\partial_t,$
 $\mathcal{Z}_j=\partial_{y^j},j=1,2,$ $\mathcal{Z}_3=\phi(z)\partial_z$ 
and define
\beq\label{newnotation}
\|\tilde{\eta_i}\|_{m,t}=\sum_{|\alpha|\leq m}\|\mathcal{Z}^{\alpha}\tilde{\eta_i}\|_{L^2([0,t]\times\mathbb{R}_{+}^3)},\quad \|\tilde{\eta_i}(t)\|_{m}=\sum_{|\alpha|\leq m}\|(\mathcal{Z}^{\alpha}\tilde{\eta_i})(t)\|_{L^2(\mathbb{R}_{+}^3)}.
\eeq
where $\mathcal{Z}^{\alpha}=\mathcal{Z}_0^{\alpha_0}\mathcal{Z}_1^{\alpha_1}\mathcal{Z}_2^{\alpha_2}\mathcal{Z}_3^{\alpha_3}, \alpha=(\alpha_0,\alpha_1,\alpha_2,\alpha_3),$
by  the definition of  the conormal spaces \eqref{not1} and the  vector fields \eqref{spatialvectorinlocal}
we find that:
 \beq\label{norm-equiv-bd}
 \|\tilde{\eta_i}\|_{m,t}\approx \|\eta_i\|_{L_t^2H_{co}^m(\Omega)}, \quad  \|\tilde{\eta_i}(t)\|_{m}\approx \|\eta_i(t)\|_{H_{co}^m(\Omega)}.
 \eeq
Therefore, our following task is to establish an  estimate for  $\sup_{0 \leq t\leq T}\|\tilde{\eta_i}(t)\|_{m-2}.$ 

We shall write $\tilde{\eta}_i, \widetilde{F_{i}^{\omega}}$ by $\tilde{\eta}, \widetilde{F^{\omega}}$
 for the sake of notational clarity. 
 We write $\tilde{\eta}=\tilde{\eta}_{h}+\tilde{\eta}_{nh},$ where $\tilde{\eta}_{h}$ solves
 \begin{equation}\label{eqetah}
 \left\{
\begin{array}{l}
\bar{\rho}\partial_t \tilde{\eta}_{h}-\mu|\bN|^2\partial_z^2 \tilde{\eta}_{h}=0 \quad \text{in}\quad \mathbb{R}_{+}^3,\\
\tilde{\eta}_{h}|_{t=0}=0, \tilde{\eta}_{h}|_{z=0}=\tilde{\eta}|_{z=0}
\end{array}
  \right.
 \end{equation}
while $\tilde{\eta}_{nh}$ satisfies
 \begin{equation}
 \left\{
\begin{array}{l}
\bar{\rho}\partial_t \tilde{\eta}_{nh}-\mu{\div}(E{\nabla} \tilde{\eta}_{nh}\big)= H(\tilde{\eta}_h)+{F^{\omega}} \quad \text{in}\quad \mathbb{R}_{+}^3,\\[5pt]
\tilde{\eta}_{nh}|_{t=0}=\tilde{\eta}|_{t=0}, \tilde{\eta}_{nh}|_{z=0}=0
    \end{array}
  \right.
 \end{equation}
 where $$H(\tilde{\eta}_h)=\mu\sum_{i,j=1}^2\p_{y^i}(E_{ij}\p_{y^j}\tilde{\eta}_h)%\quad H_2(\tilde{\eta}_h)=
 +\mu\sum_{i=1}^2\p_{y^i}(E_{i3}\p_z\tilde{\eta}_h)+\p_z(E_{3i}\p_{y^i}\tilde{\eta}_h).$$
Estimate  \eqref{ineqofomegatimesn} will be the consequence of the following two lemmas. 
 \end{proof}
\begin{lem}\label{lemofetah}
Adopting the notation introduced in \eqref{newnotation},
we have the following estimate: for any $0<t\leq T,$
\begin{equation}\label{ineqofetah}
   \sup_{0 \leq t\leq T} \|\tilde{\eta}_h(t)\|_{m-2} +\|\tilde{\eta}_h\|_{m-1,T}\lesssim T^{\f{1}{4}}\cE_{m,T}.
\end{equation}
\end{lem}
\begin{proof}
Since $|\bN|^2$ depends only on the 
tangential variable $y^1,y^2,$
%Suppose $\mathcal{Z}^{\gamma}=\mathcal{Z}_3^{\gamma_3}\mathcal{Z}_{tan}^{\gamma'},$ solving explicitly the heat equation on the half line with Dirichlet boundary condition, we get
the equation \eqref{eqetah} can be seen as a heat equation on the half line with Dirichlet boundary condition,
 which can be solved explicitly:
\begin{equation*}
    \tilde{\eta}_h(t,y,z)=-2\tilde{\mu}\int_0^t \frac{|\bN|^2}{\big(4\pi\tilde{\mu}|\bN|^2(t-s)\big)^{\frac{1}{2}}}
   \partial_z\big(e^{-\frac{z^2}{4\tilde{\mu}|\bN|^2(t-s)}}\big)
   \tilde{\eta}|_{z=0}(s,y)\d s
\end{equation*}
where $\tilde{\mu}=\mu/{\bar{\rho}}.$ 
Taking a multi-index $\gamma=(\gamma_0,\gamma_1,\gamma_2,\gamma_3), $ since time derivation commutes
with $\pt,\p_z^2,$ we have that:
\beqs
\big((\ep\pt)^{\gamma_0}\tilde{\eta}_h\big)(t,y,z)=-2\tilde{\mu}\int_0^t \frac{|\bN|^2}{\big(4\pi\tilde{\mu}|\bN|^2(t-s)\big)^{\frac{1}{2}}}
   \partial_z\big(e^{-\frac{z^2}{4\tilde{\mu}|\bN|^2(t-s)}}\big)
   \big((\ep\pt)^{\gamma_0}\tilde{\eta}\big)|_{z=0}(s,y)\d s,
\eeqs
%Taking $L_{y,z}^2(\mathbb{R}_{+}^3)$ norm, and using \eqref{L2 norm of K}, we obtain %by summing up for $|\gamma|\leq m-2, $
which, combined with \eqref{L2 norm of K} established in the appendix, yields that:
\begin{equation}\label{tildeetah}
 \|\mathcal{Z}^{\gamma}\tilde{\eta}_{h}(t)\|_{L_{y,z}^2(\mathbb{R}_{+}^3)}\lesssim % C_{|\gamma|+1}
 \int_0^t (t-s)^{-\f{3}{4}}\big|\tilde{\eta}|_{z=0}(s)\big|_{\tilde{H}^{|\gamma|}(\mathbb{R}_{y}^2)}\d s.
\end{equation}
The above inequality, combined with the boundary condition 
$\eqref{eqoftildeeta}_2$ and the trace inequality \eqref{traceLinfty}, %\eqref{traceL2} as well as the convolution inequality,
 yields % \eqref{ineqofetah}. 
that:
\beqs
\|\tilde{\eta}_{h}(t)\|_{m-2}\lesssim T^{\f{1}{4}} \sup_{0\leq s\leq t}
|\tilde{\eta}(s)|_{\tilde{H}^{m-2}(\mathbb{R}_y^2)}\lesssim T^{\f{1}{4}}\|(u,\nabla u)\|_{L_t^{\infty}H_{co}^{m-2}}\lesssim T^{\f{1}{4}}\cE_{m,T}.
\eeqs
Similarly, we apply a  convolution inequality
in the time variable (after extending $\tilde{\eta}(s)|_{z=0}$ to $s\in \mathbb{R}$ by zero extension) to \eqref{tildeetah}, and use  the boundary condition $\eqref{eqoftildeeta}_2$ and the  trace inequality \eqref{traceL2} to obtain:
\beqs
\|\tilde{\eta}_{h}(t)\|_{m-1,t}\lesssim T^{\f{1}{4}} 
|\tilde{\eta}|_{L_t^2\tilde{H}^{m-1}(\mathbb{R}_y^2)}\lesssim T^{\f{1}{4}}\|(u,\nabla u)\|_{L_t^{2}H_{co}^{m-1}}\lesssim T^{\f{1}{4}}\cE_{m,T}.
\eeqs
\end{proof}

\begin{lem}\label{lemofetanh}
Using the notation \eqref{newnotation}, the following energy inequality holds: for any $0<t\leq T,$
\begin{equation}\label{ineqofetanh}
  \begin{aligned}
 \|\tilde{\eta}_{nh}(t)\|_{
 m-2}^2+\|\nabla \tilde{\eta}_{nh}\|_{m-2,t}^2&\lesssim \|\eta(0)\|_{H_{co}^{m-2}}^2+
(T+\ep)^{\f{1}{2}}\Lambda\big(\f{1}{c_0},\cN_{m,t}\big).
 \end{aligned}  
\end{equation}
\end{lem}
\begin{proof}
Suppose that $0\leq |\gamma|=k\leq m-2.$
Denote $\tilde{\eta}_{nh}^{\gamma}=\mathcal{Z}^{\gamma}\tilde{\eta}_{nh},$ then $\tilde{\eta}_{nh}^{\gamma}$ solves the system
(note that $[\cZ^{\gamma},E]=0$): 
\begin{equation*}
\begin{aligned}
   \bar{\rho} \partial_t \tilde{\eta}_{nh}^{\gamma}-\mu{\div}\big(E{\nabla}\tilde{\eta}_{nh}^{\gamma}\big)&=\mu [\mathcal{Z}^{\gamma},{\div}](E{\nabla}\tilde{\eta}_{nh}^{\gamma})
    +\mu\mathcal{Z}^{\gamma}H(\tilde{\eta}_h)+\mathcal{Z}^{\gamma}{F^{\omega}}\\
    &=\colon \mathcal{R}_{1}^{\gamma}+\mathcal{R}_{2}^{\gamma}+\mathcal{Z}^{\gamma}{F^{\omega}}
\end{aligned}   
\end{equation*}
 with the initial condition
$\tilde{\eta}_{nh}^{\gamma}|_{t=0}=\mathcal{Z}^{\gamma}\tilde{\eta}|_{t=0}$ and the  boundary condition $\tilde{\eta}_{nh}^{\gamma}|_{z=0}=0.$

%We now perform the 
Standard energy estimates %for $\tilde{\eta}_{nh}^{\gamma}.$
show that:
\begin{multline}\label{sec5:eq14}
        \bar{\rho}\|\tilde{\eta}_{nh}^{\gamma}(t)\|_{L^2(\mathbb{R}_{+}^3)}^2+ \int_0^t \int_{\mathbb{R}_{+}^3}E{\nabla} \tilde{\eta}_{nh}^{\gamma}\cdot {\nabla}\tilde{\eta}_{nh}^{\gamma}\,\d x\d s\\
      = \bar{\rho}\|\tilde{\eta}_{nh}^{\gamma}(0)\|_{L^2(\mathbb{R}_{+}^3)}^2+\int_0^t\int_{\mathbb{R}_{+}^3}(\mathcal{R}_{1}^{\gamma}+\mathcal{R}_{2}^{\gamma}+\mathcal{Z}^{\gamma}\widetilde{F^{\omega}})
      \tilde{\eta}_{nh}^{\gamma}\,\d x\d s.
    \end{multline}
At first, since we can find some $\kappa>0,$ such that $2|\bN|^2\leq 1/{\kappa},$ one has that $E X\cdot X=|P X|^2\geq \f{1}{2 |\bN|^2}|X|^2\geq \kappa|X|^2$ and hence, we deduce that:
\beq
\int_0^t \int_{\mathbb{R}_{+}^3}E{\nabla} \tilde{\eta}_{nh}^{\gamma}\cdot {\nabla}\tilde{\eta}_{nh}^{\gamma}\,\d x\d s\geq \kappa\|{\nabla}\tilde{\eta}_{nh}^{\gamma}\|_{0,t}^2.
\eeq
For the second term of the right hand side of \eqref{sec5:eq14},
one needs to integrate by parts to avoid involving additional  normal derivatives.
Let us first study  $\mathcal{R}_{1}^{\gamma}$ which vanishes if $|\gamma|=0.$
By induction, one gets that for $k=|\gamma|\geq 1,$
\begin{equation}
\begin{aligned}
&[\mathcal{Z}^{\gamma}, {\div}] =[\mathcal{Z}^{\gamma},\partial_z]=
\sum_{\beta<\gamma}
 C_{\phi,\beta,\gamma}\partial_z\mathcal{Z}^{\beta}
\end{aligned}
\end{equation}
where $C_{\phi,\beta,\gamma}$ are smooth functions that depend on $\phi$ and its derivatives. Consequently, by integration by parts and Young's inequality, we obtain that:
 \begin{equation}\label{sec5:eq15}
 \int_0^t\int_{\mathbb{R}_{+}^3}\mathcal{R}_{1}^{\gamma}\cdot\tilde{\eta}_{nh}^{\gamma}\,\d x\d s\leq \delta \|{\nabla} \widetilde{\eta}_{nh}^{\gamma}\|_{0,t}^2
 +C_{\delta}(\|{\nabla} \tilde{\eta}_{nh} \|_{k-1,t}^2+\|\tilde{\eta}_{nh}\|_{k,t}^2).
 \end{equation}
   Similarly, by taking benefits of the zero boundary condition of $\tilde{\eta}_{nh}^{\gamma},$ one integrates by parts to get:
 \begin{equation}\label{cR2-bd}
     \int_0^t\int_{\mathbb{R}_{+}^3}\mathcal{R}_{2}^{\gamma}\tilde{\eta}_{nh}^{\gamma}\,\d x\d s\leq 
     \delta\|{\nabla} \tilde{\eta}_{nh}^{\gamma}\|_{0,t}^2
 +C_{\delta}(
 \|\tilde{\eta}_{h}\|_{{k+1},t}^2+\|\tilde{\eta}_{nh}\|_{k,t}^2).
 \end{equation}
 We are now left to deal with the  term:
 \beq\label{defcT1-cT5}
  \int_0^t\int_{\mathbb{R}_{+}^3} \mathcal{Z}^{\gamma}\widetilde{F^{\omega}}\tilde{\eta}_{nh}^{\gamma}\,\d x\d s=\sum_{j=1}^5\int_0^t\int_{\mathbb{R}_{+}^3} \mathcal{Z}^{\gamma}\widetilde{F_j^{\omega}}\tilde{\eta}_{nh}^{\gamma}\, \d x\d t=\colon \sum_{j=1}^5 \mathcal{I}_j
  \eeq
 where we denote that:
 \begin{equation*}
 \begin{aligned}
 \widetilde{F^{\omega}}&=-\widetilde{\Delta(\chi_i \bn)}\times\widetilde{\omega}-2\widetilde{\nabla\omega}\times\widetilde{\nabla(\chi_i \bn)} -\widetilde{(g_2u\cdot\nabla \omega)}\times\widetilde{(\chi_i \bn)}
 +\f{(\widetilde{\bar{\rho}-g_2})}{\ep}\widetilde{\ep\partial_t\omega}\times \widetilde{(\chi_i \bn)}+\widetilde{G^{\omega}}\times\widetilde{(\chi_i\bn)}.\\
 &=\colon \widetilde{F_{1}^{\omega}}+\widetilde{F_{2}^{\omega}}+\widetilde{F_{3}^{\omega}}+\widetilde{F_{4}^{\omega}}+\widetilde{F_{5}^{\omega}}.
 \end{aligned}
 \end{equation*}
 Note that $G^{\omega}$ is defined in \eqref{defGw}. Moreover, without much ambiguity, we denote $\tilde{f}$ as $(\tilde{\chi_i}f)\circ \Phi_i$ where $\tilde{\chi_i}$ is a smooth function such that $\tilde{\chi_i}\chi_i=\chi_i.$
 
 By the Cauchy-Schwarz inequality and the fact \eqref{norm-equiv-bd}, $\mathcal{I}_1$ can be controlled by:
 \begin{equation}\label{cT1-bd}
        \mathcal{I}_1\lesssim \|\tilde{\omega}\|_{k,t}\|\tilde{\eta}_{nh}\|_{k,t}\lesssim T^{\f{1}{2}}\|\nabla u\|_{L_t^{\infty}H_{co}^{m-2}}\|\tilde{\eta}_{nh}\|_{k,t}.
 \end{equation}
 Nevertheless, for $\mathcal{I}_2$ and $\mathcal{I}_3,$ as $\widetilde{F_{2}^{\omega}}$, $\widetilde{F_{3}^{\omega}}$ involve normal derivatives of $\omega,$ it is necessary to use integration by parts. By doing so,
 we can bound the term $\cT_2$ as follows:
 \begin{equation}
     \mathcal{I}_2\leq \delta \|{\nabla} \tilde{\eta}_{nh}^{\gamma}\|_{0,t}^2+C_{\delta}
     (\|\tilde{\eta}_{nh}\|_{k,t}^2+\|\widetilde{\nabla u}\|_{k,t}^2).
 \end{equation}
 
Next, for $\mathcal{I}_3,$ by noticing
 the expression 
 \begin{align*}
 \widetilde{g_2 u\cdot\nabla \omega}&=\p_{y^1}(\widetilde{g_2 u_1}\widetilde{\omega})+\p_{y^2}(\widetilde{g_2 u_2}\widetilde{\omega})+\p_z(\widetilde{(g_2 u\cdot\bN)}\widetilde{\omega})\\
&\quad -\big(\p_{y^1}\widetilde{g_2 u_1}+\p_{y^2}\widetilde{g_2 u_2}+\p_z(\widetilde{g_2 u\cdot\bN})\big)\widetilde{\omega},
  \end{align*}
 one performs an integration by parts again to get that:
 \begin{equation*}
 \begin{aligned}
     \mathcal{I}_3&\lesssim \|\tilde{g_2}\tilde{u}\tilde{\omega}\|_{k,t}\|{\nabla} \tilde{\eta}_{nh}^{\gamma}\|_{0,t}
     +\|\tilde{\omega}(\partial_{y^j}(\widetilde{g_2 u}_j),\p_z\widetilde{(g_2 u\cdot\bN}))\|_{k,t}\|\tilde{\eta}_{nh}^{\gamma}\|_{0,t}\\
    &\leq \delta \|{\nabla} \tilde{\eta}_{nh}^{\gamma}\|_{0,t}^2+C_{\delta}\|\tilde{g_2}\tilde{u}\tilde{\omega}\|_{k,t}+T^{\f{1}{2}} (\sup_{s\in[0,t]}
    \|\tilde{\eta}_{nh}(s)\|_{k})\|\tilde{\omega}(\partial_{y_j}(\widetilde{g_2 u}_j),\p_z(\widetilde{g_2 u\cdot\bN}))\|_{k,t}
 \end{aligned}
 \end{equation*}
 Here we used Einstein summation convention for $j=1,2.$ 
 By \eqref{norm-equiv-bd}, \eqref{ineqofetah} and the assumption $k\leq m-2,$ one can have that:
 \beq\label{etanhLinfty}
 \sup_{s\in[0,t]}\|\tilde{\eta}_{nh}\|_{k}\lesssim \sup_{s\in[0,t]} \|(\tilde{\eta},\tilde{\eta}_h)(s)\|_{k}\lesssim \|\nabla u\|_{L_t^{\infty}H_{co}^{m-2}}+T^{\f{1}{4}}\cE_{m,t}\lesssim \cE_{m,t}.
 \eeq
Moreover, since $k\leq m-2,$ we have thanks to
\eqref{norm-equiv-bd} that:
\beqs
\begin{aligned}
&\|\tilde{\omega}(\partial_{y^j}(\widetilde{g_2 u}^j),\p_z(\widetilde{g_2 u\cdot\bN}))\|_{m-2,t}\\
&\lesssim \il\omega\il_{0,\infty,t}\|Z_i (g_2 u_j), \nabla (g_2 u\cdot\bN)\|_{L_t^2H_{co}^{m-2}}\\
&\qquad+\|\omega\|_{L_t^{\infty}H_{co}^{m-2}}\big(\int_0^t\| [Z_i (g_2 u_j), \nabla (g_2 u\cdot\bN)](s)\|_{m-3,\infty}^2 \d s\big)^{\f{1}{2}}
\end{aligned}
\eeqs
where $Z_i$ stands for the tangential vector fields in $\Omega_i.$
By identity \eqref{normalofnormalder} and the Sobolev embedding \eqref{sobebd} and estimate \eqref{nabladivu-4},
 \beqs
 \begin{aligned}
  \big(\int_0^t\| Z_i (g_2 u_j), \nabla (g_2 u\cdot\bN)(s)\|_{m-3,\infty}^2 \d s\big)^{\f{1}{2}}&\lesssim 
  \|u\|_{E_t^m}+\|\nabla\div u\|_{L_t^2H_{co}^{m-2}}+\ep \Lambda\big(\f{1}{c_0},\cN_{m,t}\big)\\
  &\lesssim \|(\sigma,u)\|_{E_t^m}+\ep^{\f{1}{2}}\Lambda\big(\f{1}{c_0},\cN_{m,t}\big),
 \end{aligned}
 \eeqs
 which together with the previous inequality, yields:
 \beqs
 \|\tilde{\omega}(\partial_{y^j}(\widetilde{g_2 u}_j),\p_z(\widetilde{g_2 u\cdot\bN}))\|_{m-2,t}\lesssim \Lambda\big(\f{1}{c_0},\cN_{m,t}\big).
 \eeqs
Similarly, we have that:
\beqs
\|\tilde{g_2}\tilde{u}\tilde{\omega}\|_{k,t}\lesssim T^{\f{1}{2}}\il\omega\il_{0,\infty,t}\|u\|_{L_t^{\infty}H_{co}^{m-2}} %\|\omega\|_{L_t^{\infty}H_{co}^{m-2}}
+\|u\|_{E_t^m}\|\omega\|_{L_t^{\infty}H_{co}^{[\f{m}{2}]-2}}+(T+\ep)^{\f{1}{2}}\La\cE_{m,t}.
\eeqs
Moreover, if $k\leq [\f{m}{2}]-2,$
 \beqs
 \|\tilde{g_2}\tilde{u}\tilde{\omega}\|_{k,t}\lesssim \La\|\nabla u\|_{L_t^2H_{co}^{[\f{m}{2}]-2}}\lesssim T^{\f{1}{2}}\La\cE_{m,t}.
 \eeqs
 To summarize, we control $\cT_3$ (defined in \eqref{defcT1-cT5}) as follows:
 \beq
 \cT_3 \leq \delta \|{\nabla} \tilde{\eta}_{nh}^{\gamma}\|_{0,t}^2+
 %+ \|\tilde{\eta}_{nh}\|_{k,t}\Lambda\big(\f{1}{c_0},\cN_{m,t}\big)+
 (T+\ep)^{\f{1}{2}}\Lambda\big(\f{1}{c_0},\cN_{m,t}\big), \, \text{ if } k\leq [\f{m}{2}]-2, 
 \eeq
 and for $k\leq m-2,$
 \beq
 \cT_3 \leq \delta \|{\nabla} \tilde{\eta}_{nh}^{\gamma}\|_{0,t}^2+(T+\ep)^{\f{1}{2}}\Lambda\big(\f{1}{c_0},\cN_{m,t}\big)+\|(\sigma,u)\|_{E_t^m}\|\omega\|_{L_t^{\infty}H_{co}^{[\f{m}{2}]-2}}.
 \eeq
 For $\mathcal{I}_4,$ the  direct application of the H\"older inequality requires the control of the  quantity $\|\f{(\widetilde{\bar{\rho}-g_2})}{\ep}\widetilde{\ep\partial_t\omega}\|_{k,t},$ which further requires  the estimate of $L_{t,x}^{\infty}$ type norm of $\pt\omega$
However, $\il \varepsilon\partial_t\omega\il_{\infty,t}$ (or $\il\nabla u \il_{1,\infty,t}$)
seems out of control and does not appear in the $L_{t,x}^{\infty}$ type norms present in  $\cA_{m,T}.$
To avoid this problem, since  $\widetilde{\ep\pt\omega}=(P%\widetilde
 {\nabla})\times\widetilde{\ep\pt u},$ we can integrate by parts in space before using product estimate. By doing so, we achieve that:
 \begin{equation}\label{cT2-bd}
 \begin{aligned}
     \mathcal{I}_4&\leq \delta \|{\nabla}\tilde{\eta}_{nh}^{\gamma}\|_{0,t}^2+C_{\delta}\|\tilde{\eta}_{nh}\|_{k,t}^2+\|(\widetilde{\nabla\sigma},\widetilde{\ep\pt u})\|_{{k},t}^2\La\\
    &\lesssim \delta \|{\nabla}\tilde{\eta}_{nh}^{\gamma}\|_{0,t}^2+C_{\delta}\|\tilde{\eta}_{nh}\|_{k,t}^2+ T\La \cE_{m,t}^2.
\end{aligned}
 \end{equation}
 
Finally, regarding the term $\cT_5$ 
(defined in \eqref{defcT1-cT5})
we control it by Cauchy-Shwarz inequality as:
\beqs
\cT_5\lesssim T^{\f{1}{2}}\big(\sup_{s\in [0,t]} \|\tilde{\eta}_{nh}(s)\|_{k}\big)
\|\widetilde{G^{\omega}}\|_{k,t}.
\eeqs
By the estimate \eqref{etanhLinfty}, the fact \eqref{norm-equiv-bd} and the Proposition \ref{propGomega}, we get that:
\beq\label{cT5}
\cT_5\lesssim 
(T+\ep)^{\f{1}{2}}\Lambda\big(\f{1}{c_0},\cN_{m,t}\big).
\eeq

To summarize, we have found by collecting \eqref{cT1-bd}-\eqref{cT2-bd} that
for $0\leq k\leq m-2,$
 \begin{equation}\label{sec5:eq16}
 \begin{aligned}
      &\quad\int_0^t\int_{\mathbb{R}_{+}^3} \mathcal{Z}^{\gamma}\widetilde{F^{\omega}}\tilde{\eta}_{nh}^{\gamma}\, \d x\d t\leq 3\delta \|
      {\nabla} \tilde{\eta}_{nh}^{\gamma}\|_{0,t}^2\\
      &+C_{\delta}
   \big(\|(\tilde{\eta},\tilde{\eta}_{h})\|_{{k},t}^2+
 +\|u\|_{E_t^m}\|\omega\|_{L_t^{\infty}H_{co}^{[\f{m}{2}]-2}}\big)+(T+\ep)^{\f{1}{2}}\Lambda\big(\f{1}{c_0},\cN_{m,t}\big)\\
&\leq  3\delta \|{\nabla}  \tilde{\eta}_{nh}^{\gamma}\|_{0,t}^2+C_{\delta}
    \|u\|_{E_t^m}\|\omega\|_{L_t^{\infty}H_{co}^{[\f{m}{2}]-2}}+(T+\ep)^{\f{1}{2}}\Lambda\big(\f{1}{c_0},\cN_{m,t}\big),
 \end{aligned}
 \end{equation}
 and also for $0\leq k\leq [\f{m}{2}]-2,$ 
 \begin{equation}\label{sec5:eq17}
 \begin{aligned}
      &\quad\int_0^t\int_{\mathbb{R}_{+}^3} \mathcal{Z}^{\gamma}\widetilde{F^{\omega}}\tilde{\eta}_{nh}^{\gamma}\, \d x\d t\leq 3\delta \|
      {\nabla} \tilde{\eta}_{nh}^{\gamma}\|_{0,t}^2+(T+\ep)^{\f{1}{2}}\Lambda\big(\f{1}{c_0},\cN_{m,t}\big).
    \end{aligned}
\end{equation}
 
Inserting \eqref{sec5:eq15}-\eqref{cR2-bd} 
\eqref{sec5:eq16}-\eqref{sec5:eq17} in \eqref{sec5:eq14}, we obtain
 by choosing $\delta$ small enough that for any $0\leq k\leq m-2,$
 \begin{equation}\label{sec5:eq18}
  \begin{aligned}
 &\quad \|\tilde{\eta}_{nh}(t)\|_{k}^2+\|\nabla \tilde{\eta}_{nh}\|_{{k},t}^2\lesssim \|\eta(0)\|_{H_{co}^{k}}^2+\|\nabla \tilde{\eta}_{nh}\|_{{k-1},t}^2\\
 &+(T+\ep)^{\f{1}{2}}\Lambda\big(\f{1}{c_0},\cN_{m,t}\big)+\|(\sigma,u)\|_{E_t^m}\|\omega\|_{L_t^{\infty}H_{co}^{[\f{m}{2}]-2}}\mathbb{I}_{\{k\geq [\f{m}{2}]-1\}}.
 \end{aligned}  
\end{equation}
where the convention 
$\|\cdot\|_{l,t}=0$ if $l< 0$ is used. 
We thus get by induction on $0\leq k\leq [\f{m}{2}]-2$ that:
 \begin{equation}\label{tildeetanh1}
  \begin{aligned}
 &\quad \|\tilde{\eta}_{nh}(t)\|_{[\f{m}{2}]-2}^2+\|\nabla \tilde{\eta}_{nh}\|_{[\f{m}{2}]-2,t}^2\lesssim \|\eta(0)\|_{H_{co}^{m-2}}^2+(T+\ep)^{\f{1}{2}}\Lambda\big(\f{1}{c_0},\cN_{m,t}\big),
 \end{aligned}  
\end{equation}
which, together with \eqref{ineqofetah} and \eqref{recovernablasi} gives that:
\beqs
\|\nabla u\|_{L_t^{\infty}H_{co}^{[\f{m}{2}]-2}}^2\lesssim Y_{m}^2(\sigma_0,u_0)+(T+\ep)^{\f{1}{2}}\Lambda\big(\f{1}{c_0},\cN_{m,t}\big).
\eeqs
We then combine this estimate and \eqref{L2L2bd} to obtain that:
\beqs
\|u\|_{E_t^m}\|\omega\|_{L_t^{\infty}H_{co}^{[\f{m}{2}]-2}}\lesssim Y_{m}^2(\sigma_0,u_0)+(T+\ep)^{\f{1}{2}}\Lambda\big(\f{1}{c_0},\cN_{m,t}\big).
\eeqs
Therefore, we take benefits of the estimate \eqref{sec5:eq18} and the induction arguments to get  \eqref{ineqofetanh}.
\end{proof}
\begin{prop}\label{propGomega}
Assume that \eqref{preasption1} holds and 
let $$G^{\omega}=g_2\omega\cdot\nabla u-g_2\omega\div u-\frac{\nabla g_2}{\varepsilon}\times (\varepsilon\partial_t u+\varepsilon u\cdot\nabla u),$$ then we have:
\beqs
\|\tilde{\chi_i}G^{\omega}\|_{L_t^2H_{co}^{m-2}}\lesssim \Lambda\big(\f{1}{c_0},\cN_{m,t}\big).
\eeqs
\end{prop}
\begin{proof}
Let us show the estimate of $\tilde{\chi_i}\omega\cdot\nabla u,$
which is not direct since the higher order $L_{t,x}^{\infty}$ norm 
(say $\|\nabla u\|_{[\f{m}{2}]-1,\infty,t}$) is unlikely to be uniformly bounded.
Nevertheless, thanks to identity \eqref{convection identity}, one can write this term as:
\beqs
\tilde{\chi_i} \omega\cdot\nabla u
=\tilde{\chi_i} \big(\omega_1\p_{y^1}u+\omega_2\p_{y^2}u +(\omega\cdot\bN) \p_{\bn}u\big).
\eeqs
Moreover, by identities \eqref{atimesb} and \eqref{normalofnormalder},
\beqs
\begin{aligned}
\omega\cdot\bN&=(\nabla\times u)\cdot\bN\\
&=-(u\times\bN)\p_{\bn}\bn+(\Pi\p_{y^1}(u\times\bN))^{1}+(\Pi\p_{y^2}(u\times\bN))^{2}+u\cdot\curl\bN
\end{aligned}
\eeqs
which gives that for any $t\in [0,T],$ any $k\geq 0,$
%$\omega\cdot\bN\approx F(\p_y u, u, \bN, \nabla \bN)$
\beqs
\|(\omega\cdot\bN)(t)\|_{H_{co}^k}\lesssim \|u(t)\|_{H_{co}^{k+1}},\qquad 
\|(\omega\cdot\bN)(t)\|_{k,\infty}\lesssim \|u(t) \|_{k+1,\infty}
\eeqs
Therefore, by the Sobolev embedding \eqref{sobebd}, we have that: 
\beqs
\begin{aligned}
&\|\tilde{\chi_i} \omega\cdot\nabla u\|_{L_t^2H_{co}^{m-2}}\\
&\lesssim 
\il\nabla u\il_{0,\infty,t}\|(\p_{y^i}u, \omega\cdot\bN)\|_{L_t^2H_{co}^{m-2}}+\|\nabla u\|_{L_t^{\infty}H_{co}^{m-2}}\big(\int_0^t \|(\p_{y^i}u, \omega\cdot\bN)(s)\|_{{m-3},\infty}^2\d s \big)^{\f{1}{2}}\\
&\lesssim \il\nabla u\il_{0,\infty,t}\|u\|_{L_t^2H_{co}^{m-1}}+\|\nabla u\|_{L_t^{\infty}H_{co}^{m-2}}\|u\|_{E_t^m}\lesssim \Lambda\big(\f{1}{c_0}, \cN_{m,t}\big).
\end{aligned}
\eeqs
The other two terms in the definition of $G^{\omega}$ are similar or easier to treat, we omit the details.
%By product estimate \eqref{roughproduct1} and also identity \eqref{normalofnormalder},one can show that:
%\beqs\|\tilde{\chi_i}G^{\omega}\|_{L_t^2H_{co}^{m-2}}\lesssim \La \|(\sigma,u)\|_{E_t^{m-1}}\lesssim T^{\f{1}{2}}\Lambda\big(\f{1}{c_0},\cN_{m,t}\big).\eeqs
\end{proof}
\begin{rmk}
Collecting the results stated in Lemmas \ref{highestgrad}, \ref{lemofv}, \ref{lemofnablav}, \ref{normalofnablasigma}, we find that:
\beq\label{sumarize1}
\begin{aligned}
&\|\ep\nabla(\sigma,u)\|_{L_t^{\infty}H_{co}^{m-1}}^2+\|\nabla(\sigma,u)\|_{L_t^{\infty}H_{co}^{m-2}\cap L_t^2H_{co}^{m-1}}^2+\|(\sigma,u)\|_{L_t^{\infty}H_{co}^{m-1}}^2\\
&\lesssim Y^2_m(\sigma_0,u_0)+(T+\ep)^{\f{1}{2}}\Lambda(\f{1}{c_0},\cN_{m,T}).
\end{aligned}
\eeq
\end{rmk}

\subsection{$\ep-$dependent estimate of $\nabla^2 u$}

To finish the estimates for the energy norm, we are left to deal with $\|\ep \nabla^2 u\|_{L_t^{\infty}H_{co}^{m-2}},\ep \|\nabla^2\sigma\|_{L_t^{\infty}L^2}.$
\begin{lem}\label{propsecnormal1}
Under the assumption \eqref{preasption},  the following estimate holds:
\begin{equation}\label{secnormalinfty1}
    \begin{aligned}
 \|    \varepsilon\nabla^2 u(t)\|_{H_{co}^{m-2}}^2\lesssim 
    Y^2_m(\sigma_0,u_0)+(T+\ep)^{\f{1}{2}}\Lambda(\f{1}{c_0}, \cN_{m,T}).
    \end{aligned}
\end{equation}
\end{lem}
\begin{proof}

As $u$ satisfies the equation:
$$\varepsilon\mu\Delta u=-(\mu+\lambda)\varepsilon \nabla\div u+g_2(\varepsilon\partial_t u+\varepsilon u\cdot \nabla u)+\nabla\sigma.$$
 we have by elliptic regularity theory:
\begin{equation}\label{secondnormal0}
    \begin{aligned}
    \| \varepsilon\nabla^2 u(t)\|_{H_{co}^{m-2}}&\lesssim
    \ep \sum_{|I|\leq m-2}|Z^{I}\partial_{\bn} u(t)|_{{H}^{\frac{1}{2}}}+\varepsilon \|\nabla\div u(t)\|_{H_{co}^{m-2}}
     \\
     &\qquad+\|u(t)\|_{H_{co}^{m-1}} +\|\nabla\sigma(t)\|_{H_{co}^{m-2}}+\ep^{\f{1}{2}}\cE_{m,T}\Lambda\big(\f{1}{c_0},\cA_{m,T}\big).
   \end{aligned}
\end{equation}

It follows 
from the boundary condition \eqref{bdryconditionofu}, the identity \eqref{normalofnormalder}
 and the  trace inequality \eqref{traceLinfty} that:
\beq\label{bdry3}
\varepsilon\sum_{|I|\leq m-2}|Z^{I}\partial_{\bn} u(t)|_{{H}^{\frac{1}{2}}} \lesssim \ep\|\nabla\div u(t)\|_{H_{co}^{m-2}}+\ep\|(u,\nabla u)(t)\|_{H_{co}^{m-1}}.
\eeq
Inserting \eqref{nabladivu-5} and \eqref{bdry3} into \eqref{secondnormal0}, 
one arrives at:
\beqs
\ep\|\nabla^2 u(t)\|_{H_{co}^{m-2}}\lesssim \ep \|\nabla(\sigma,u)(t)\|_{H_{co}^{m-1}}+\|\nabla\sigma(t)\|_{H_{co}^{m-2}}+\|u(t)\|_{H_{co}^{m-1}}+\ep^{\f{1}{2}}\cE_{m,T}\Lambda\big(\f{1}{c_0},\cA_{m,T}\big),
\eeqs
which, combined with \eqref{sumarize1} leads to
\eqref{secnormalinfty1}. 
\end{proof}
\begin{lem}\label{lemnabla2sigma}
Under the assumption \eqref{preasption}, we have the following estimate for $\nabla^2 \sigma$:
\begin{equation}\label{es-nabla2sigma}
\|\ep\nabla^2\sigma\|_{L_t^{\infty}L^2}
^2+\|\nabla^2 \sigma\|_{L^2(Q_t)}^2\lesssim
Y^2_m(0)+(T+\ep)\Lambda(\f{1}{c_0}, \cN_{m,T}).
\end{equation}
\end{lem}
\begin{proof}
By \eqref{rewriteofsigma} and 
\eqref{rewriteofu}, one finds that $\nabla\sigma$ solves:
 \beq\label{sec:eq6.5}
 \ep^2 g_1(\p_t +u\cdot\na )\nabla\sigma+\f{1}{(2\mu+\lambda)}\nabla\sigma=G
 \eeq
 where $$G=-\ep^2(g_1' S\ep\pt\si+\nabla(g_1 u_{k})\cdot\partial_{k}\sigma)-\ep \f{\mu}{(2\mu+\lambda)}\curl \omega-\f{1}{(2\mu+\lambda)} g_2(\ep\pt u+\ep u\cdot\na u).$$
 By taking the  divergence of the equation  \eqref{sec:eq6.5}, one arrives at:
 \beq\label{laplacesigma}
 \ep^2g_1(\p_t+u\cdot\nabla)\Delta\sigma+\f{1}{2\mu+\lambda}\Delta\sigma=\div G-\ep^2\big[g_1'\nabla\sigma\cdot\ep\pt\nabla\sigma+\sum_{i=1}^3\p_i(g_1u)\cdot \nabla\p_i\sigma\big]=\colon \tilde{G}
 \eeq
From an energy estimate, we find
 \beq\label{sec3:eq12}
\ep^2 \|\Delta\sigma\|_{L_t^{\infty}L^2}^2+\|\Delta\sigma\|_{L^2(Q_t)}^2\lesssim 
 T^{\f{1}{2}}\|\Delta\sigma\|_{L^2(Q_t)}\|\tilde{G}\|_{L_t^{\infty}L^2}
 +T\La\|\ep\Delta\sigma\|_{L_t^{\infty}L^2}^2.
 \eeq
We first observe  that:
 $$\|\tilde{G}\|_{L_t^{\infty}L^2}\lesssim \La\cE_{m,t}^2. $$
 Moreover, since in the local coordinate, we can find some coefficients $a_{ij}$ that depends smoothly  on $\bn,$ such that (we use the convention $\p_{y^3}=\p_{\bn}$ ):
 \beq\label{Laplace-local}
 \Delta=\p_{\bn}^2+\sum_{0\leq i,j\leq 3, 
(i,j)\neq (3,3).
} \p_{y^i}(a_{ij}\p_{y^j})
\eeq
which yields:
 \beqs
 \|\p_{\bn}\nabla \sigma\|_{L_t^{\infty}L^2}\lesssim \|\Delta\sigma\|_{L_t^{\infty}L^2}+\|\nabla\sigma\|_{L_t^{\infty}H_{co}^1}.
 \eeqs
 We thus obtain \eqref{es-nabla2sigma} from \eqref{sec3:eq12}.

\end{proof}
\subsection{Proofs of Proposition \ref{energybdd}}
By collecting \eqref{sumstep1},
\eqref{sumarize1} \eqref{secnormalinfty1} and \eqref{es-nabla2sigma}, we get \eqref{energybound}. 
\begin{rmk}\label{highordersigma}
In view of the  formal expansion \eqref{formalapp}, one expects
 the first three normal derivatives of $\sigma$ to be bounded in $L^2(Q_t)$. This can be achieved in the following way. By imposing additional assumption on $\sigma_0,$ namely $\ep\nabla^2\sigma_0\in H_{co}^1(\Omega),\nabla^3\sigma_0\in L^2(\Omega),$
 one can show by following  similar computations as in the proof of  Lemma \ref{lemnabla2sigma}  that:
 $\ep\nabla^2 \sigma\in 
 L_t^{\infty}H_{co}^1,  \nabla\sigma\in L_t^2H_{co}^1.$
 These estimates at hand, one can carry out another energy estimate to control  $\|\p_{\bn}\Delta\sigma\|_{L^2(Q_t)
 },$ which further leads to the boundedness of $\|\nabla^3 \sigma\|_{L^2(Q_t)}.$ We remark that in the latter energy estimate, the knowledge of $\|\ep\nabla^3 u\|_{L^2(Q_t)}$ is needed. Nevertheless, this term can be bounded by all the controlled norms appearing in $\cN_{m,T}$. More precisely,
 one has by equation for the velocity $$\varepsilon \div \mathcal{L}u=g_2(\varepsilon\partial_t u+\varepsilon u\cdot \nabla u)+\nabla\sigma,$$
 and thus by \eqref{Laplace-local}:
 \beq\label{threeorderofu}
 \begin{aligned}
 \|\ep\nabla^3 u\|_{L^2(Q_t)}&\lesssim 
 \|\ep\nabla\div\cL u \|_{L^2(Q_t)}+\|\ep\nabla^2 u\|_{L_t^2H_{co}^1}\\
 &\leq 
 \Lambda(1/c_0,\mathcal{A}_{m,T})(\|({\sigma,u})\|_{L_t^2H_{co}^{2}}+\|\nabla(\sigma,u)\|_{L_t^2H_{co}^1}+\|\nabla^2\sigma\|_{L^2(Q_t)}).
 \end{aligned}
 \eeq
% Nevertheless, since these high order estimates for $\sigma$ are not necessary to establish the uniform energy estimates, we did not involve them in our arguments. 
\end{rmk}
 \section {Uniform estimates- $L_{t,x}^{\infty}$ norms
 }\label{Linfty}
 
In this section, we aim to control the $L_{t,x}^{\infty}$ norms appearing in $\mathcal{A}_{m,T}.$ Part of them can be deduced directly from the Sobolev embedding in the conormal setting (see Proposition \ref{propembeddding}) and the norms controlled in the previous section. Moreover, we use the maximum principle for transport-diffusion equation \eqref{eqofomega} satisfied by $\omega$ and of the damped transport equation \eqref{sec:eq6.5} for $\nabla\sigma$ to get the $L_{t,x}^{\infty}$ estimates of $\nabla u$ and $\nabla\sigma$ respectively.

We will prove the following proposition. 
 \begin{prop}\label{Linftybd}
 Assuming that \eqref{preasption} \eqref{preasption1} hold, then there is a constant $C_2(1/c_0)$ depending only on $1/c_0$ and a polynomial $\bar{\Lambda}$ whose coefficients are independent of $\ep,$ such that:
 \beq
 \begin{aligned}
 \mathcal{A}_{m,T}&\leq C_2(1/c_0)
 \big(Y_{m}(\sigma_0,u_0)+\cE_{m,T}\big)
 +(\varepsilon^{\frac{1}{2}}+T)\mathcal{A}_{m,T}\bar{\Lambda}(1/c_0,\cA_{m,T}).\\
 %&+\sup_{0\leq s\leq T} \|(\sigma,u)(s)\|_{H_{co}^{m-1}}+\|\nabla(\sigma,u)(s)\|_{H_{co}^{m-2}}\\%+\|\omega_{\tau}(s)\|_{m-1}\\
 %&+\sup_{0\leq s\leq T}\varepsilon(\|(\sigma,u)(s)\|_{H_{co}^{m}}+\|\nabla(\sigma,u)(s)\|_{H_{co}^{m-1}}+\|\nabla^2 u(s)\|_{H_{co}^{m-2}}).
 \end{aligned}
 \eeq
 \end{prop}
 \begin{proof}
  Let us recall that  $\mathcal{A}_{m,T}$ is defined as:
  \begin{equation}\label{defofLambdam-1}
\begin{aligned}
\cA_{m,T}&=
\il \nabla u\il_{0,\infty,T}+\il(\nabla\sigma,\div u)\il_{[\frac{m-1}{2}],\infty,T}+
\il (\sigma,u)\il_{[\frac{m+1}{2}],\infty,T}\\%+\il\ep \nabla\sigma \il_{[\frac{m+1}{2}],\infty,T}\\
&\quad+
\il\varepsilon^{\frac{1}{2}}\nabla u\il_{[\frac{m-1}{2}],\infty,T}+
\il\ep \nabla u\il_{[\frac{m+1}{2}],\infty,T}+\varepsilon\il(\sigma, u)\il_{[\frac{m+3}{2}],\infty,T}.
\end{aligned}
\end{equation}
The last four terms of $\cA_{m,T}$ can be controlled directly by the Sobolev embedding \eqref{sobebd}. For instance,
\begin{equation}\label{sobebb1}
 \il(\sigma,u)\il_{[\frac{m+1}{2}],\infty,T}\lesssim  \sup_{0\leq s\leq T}\big(\|(\sigma,u)(s)\|_{H_{co}^{[\frac{m+5}{2}]}}+\|\nabla(\sigma,u)(s)\|_{H_{co}^{[\frac{m+3}{2}]}}\big) \lesssim \cE_{m,T},
\end{equation}
\begin{equation*}
\begin{aligned}
 &\varepsilon^{\frac{1}{2}}\il\nabla u\il_{[\frac{m-1}{2}],\infty,T}\lesssim  \sup_{0\leq s\leq T} \big( \|\nabla u (s)\|_{H_{co}^{[\frac{m+3}{2}]}}+\varepsilon\|\nabla^2 u(s)\|_{H_{co}^{[\frac{m+1}{2}]}}\big)\lesssim \cE_{m,T} ,\\
 \end{aligned}
\end{equation*}
\begin{equation*}\label{epdeltau}
\begin{aligned}
 &\varepsilon\il \nabla u\il_{[\frac{m+1}{2}],\infty,T}\lesssim \varepsilon \sup_{0\leq s\leq T} \big( \|\nabla u(s)\|_{H_{co}^{[\frac{m+5}{2}]}}+\|\nabla^2 u(s)\|_{H_{co}^{[\frac{m+3}{2}]}}\big)\lesssim \cE_{m,T} .\\
 \end{aligned}
\end{equation*}
Note that we have $[\frac{m+3}{2}]+1\leq m-2,\, [\frac{m+5}{2}]\leq m-1$
if $m\geq 6.$ 

We remark also that
$\il\div u \il_{[\frac{m-1}{2}],\infty,T}$
can be estimated by the other quantities in  the definition of $\mathcal{A}_{m,T}.$ 
Indeed, by using the equation satisfied by $\sigma,$ 
we have that:
\begin{equation}\label{divbd}
\begin{aligned}
&\il\div u\il_{[\frac{m-1}{2}],\infty,T}\lesssim \il\sigma\il_{[\frac{m+1}{2}],\infty,T}+\varepsilon \mathcal{A}_{m,T}^2,\\
\end{aligned}
\end{equation}

 It thus remains  to control $\il\nabla u\il_{0,\infty,T},\il\nabla \sigma\il_{[\frac{m-1}{2}],\infty,T} %\il \ep\nabla \sigma\il_{[\frac{m+1}{2}],\infty,T}
 .$ 
 We note that away from the boundaries  where the conormal Sobolev norm is equivalent to the usual Sobolev norm,  these two terms can be bounded by the standard Sobolev embedding. Therefore, it suffices to control 
 $\il\chi_{i}\partial_{\bn} u\il_{0,\infty,T},\il\chi_{i}\partial_{\bn} \sigma\il_{[\frac{m-1}{2}],\infty,T}, %\il \ep\chi_i \p_{\bn} \sigma\il_{[\frac{m+1}{2}],\infty,T}
 $ where $\chi_i, (1\leq i\leq N)$ are  smooth functions compactly supported in $\Omega_i$. Moreover, by identity \eqref{normalofnormalder} and 
 $$\Pi(\partial_{\bn} u)=\omega\times \bn+2\Pi(-(D \bn)u),$$
 we reduce our problem to the control of $\|\omega\|_{0,\infty,T},\|\chi_{i}\partial_{\bn} \sigma\|_{[\frac{m-1}{2}],\infty,T},$ which is the aim of  
 the following two lemmas.

 \end{proof}
 
 We begin with the estimate for $\il\omega\il_{0,\infty,T}$ which follows from the maximum principle of the transport-diffusion equation
 for the vorticity. 
 \begin{lem}
Under the assumption \eqref{preasption}, the following estimate holds:
 \begin{equation}\label{eqofomega}
     \il \omega \il_{0,\infty,T}\lesssim \| \omega(0)\|_{L^{\infty}(\Omega)}+\cE_{m,T}+(T+\ep)\cA_{m,T}^2.
 \end{equation}
 \end{lem}
 \begin{proof}
Recall that $\omega$ solves  \eqref{defGw} which is rewritten below for convenience:
 \begin{equation*}
     g_2(\partial_t+u\cdot\nabla)\omega-\mu\Delta\omega=g_2(\omega\cdot\nabla u-\omega\div u)+\nabla g_2\times [(\partial_t+u\cdot\nabla)u]=G^{\omega} \quad x\in\Omega. 
 \end{equation*}
 Since $g_2(\varepsilon\sigma)$ satisfies the transport equation:
 $\partial_t g_2+\div(g_2 u)=0,$
  by the maximum principle, (one can refer to Proposition 13 of \cite{MR3485413})
 \begin{equation}\label{sec4:eq1}
     \begin{aligned}
     \| \omega(t)\|_{L^{\infty}(\Omega)} &\leq \| \omega(0)\|_{L^{\infty}(\Omega)}+|\omega(t)|_{L^{\infty}(\partial\Omega)}+\frac{1}{\inf g_2} \int_0^t \|G^{\omega}(s)\|_{L^{\infty}(\Omega)}\,\d s.
     \end{aligned}
 \end{equation}
 For the second term in the  right hand side of \eqref{sec4:eq1}, we use the boundary condition \eqref{bdryconditionofu},
the  identity \eqref{normalofnormalder}  and \eqref{sobebb1}, \eqref{divbd} to get that:
\begin{equation*}
    |\omega(t)|_{L^{\infty}(\partial\Omega)}\lesssim |(u,\partial_y u,\div u)(t)|_{L^{\infty}(\partial\Omega)}\lesssim%\il\div u\il_{0,\infty,T}+(\|u(t)\|_{H_{co}^3}+\|\nabla u(t)\|_{H_{co}^2}).
    \cE_{m,T}+\ep\cA_{m,T}^2.
\end{equation*}
%{\color{red} you already estimated these $L^\infty$ norms before, quote the estimates instead of doing them again}
For the last term, we have by the  assumption \eqref{preasption} and  the property \eqref{preasption1} that there is some $C(1/c_0),$ such that:
 \beqs
\frac{1}{\inf g_2} \int_0^t \|G^{\omega}(s)\|_{L^{\infty}(\Omega)}\, \d s\leq  C({1}/{c_0})T\cA_{m,T}^2,
\eeqs
 which ends  the proof.
 \end{proof}

 In the following, we estimate
$\il\chi_i\partial_{\bn} \sigma\il_{[\frac{m-1}{2}],\infty,T}:$
%,\il \ep\chi_i \p_{\bn} \sigma\il_{[\frac{m+1}{2}],\infty,T}:$ 
\begin{lem}
Under the assumption \eqref{preasption}, we have:
\begin{equation}\label{nablasigmaLinfty}
\begin{aligned}
 \il \chi_i\partial_{\bn}\sigma\il _{[\frac{m-1}{2}],\infty,T}%+\il \ep\chi_i \p_{\bn} \sigma\il_{[\frac{m+1}{2}],\infty,T}
 &\lesssim
 Y_{m}(\sigma_0,u_0)+\cE_{m,T}
 +\varepsilon^{\frac{1}{2}}\mathcal{A}_{m,T}\Lambda\big(\f{1}{c_0},\cA_{m,T}\big)
\end{aligned}
\end{equation}
where $\chi_{i}$ is a smooth function that is  compactly supported in $\Omega_i.$
\end{lem}
\begin{proof}
We define $R=\chi_i\partial_{\bn}\sigma=\chi_i \bn\cdot\nabla\sigma.$ By \eqref{sec:eq6.5},
$R$ solves the following equation:
\begin{equation}\label{eqofR}
    \varepsilon^2 g_1(\partial_t R +u\cdot\nabla R)+\frac{1}{2\mu+\lambda}R=-\varepsilon^2 g_1 u\cdot \nabla (\chi_i \bn_k)\partial_k\sigma +G\cdot\chi_i \bn=\colon G_R
\end{equation}
where
 $$G=-\ep^2(g_1' R\ep\pt\si+\nabla(g_1 u_{k})\cdot\partial_{k}\sigma)-\ep \f{\mu}{(2\mu+\lambda)}\curl w-\f{1}{(2\mu+\lambda)} g_2(\ep\pt u+\ep u\cdot\na u).$$
  By applying  $Z^{I}$  $(|I|\leq [\frac{m-1}{2}])$ to the  equation \eqref{eqofR}, we get by setting $R^{I}=Z^{I}R$ that 
 \begin{equation*}
    \varepsilon^2 g_1(\partial_t R^{I} +u\cdot\nabla R^{I})+\frac{1}{2\mu+\lambda}R^{I}=Z^{I}G_R+\mathcal{C}_{R,1}^I+\mathcal{C}_{R,2}^I=\colon \mathcal{H}^I
 \end{equation*}
 where  $\mathcal{C}_{R,1}^I=-\varepsilon^2[Z^{I},g_1/\varepsilon]\varepsilon\partial_t R,\quad \mathcal{C}_{R,1}^I=-\varepsilon^2[Z^{I},g_1u\cdot\nabla]R.$
 
 It is  convenient to use the Lagranian coordinates. Define the unique flow $X_t(x)=X(t,x)$
associated to $u$: 
 \begin{equation}
 \left\{
  \begin{array}{l}
   \partial_t X(t,x)=u(t,X(t,x))\\[4pt]
   X(0,x)=x\in\Omega.
  \end{array}  
  \right.
 \end{equation}
Note that since $u\cdot \bn|_{\partial{\Omega}}=0,$ and $u\in Lip([0,T]\times\Omega),$ we have for each $t\in[0,T],$ $X_t:\Omega\rightarrow\Omega$
 is a diffeomorphism.
 %Denote $f^X(t,x)=f(t,X(t,x)),${\color{red} this is useless you introduce this notation for a single formula}
 By using the characteristics method, $R^I(t, X_t(x))$
 can then be expressed in the following way:  \begin{equation}\label{explicitformulaeforR}
 R^I(t,X_t(x))=e^{-\Gamma(t,x)}R^{I}(0)+\int_0^t e^{-\Gamma(t-s,x)}\big(\frac{1}{\varepsilon^2 g_1}\mathcal{H}^{I}\big)(s, X_s(x))\,
 \d s
 \end{equation}
 where $\Gamma(t,x)=\frac{1}{2\mu+\lambda}\int_0^t \frac{1}{\varepsilon^2g_1(s, X_s(x))}\d s\geq \frac{c_0 t}{(2\mu+\lambda)\varepsilon^2}.$ Note that we have used assumption \eqref{preasption} and property \eqref{preasption1}.  
 Taking the supremum in $(t,x)\in[0,T]\times\Omega$ on  both 
 sides of \eqref{explicitformulaeforR}, and using that $X(t,\cdot)(0\leq t\leq T)$ is a diffeomorphism of  $\Omega$,  we arrive at:
 \begin{equation}\label{RILinfty}
 \|R^{I}(t)\|_{L^\infty(\Omega)}\lesssim \|R^I(0)\|_{L^{\infty}(\Omega)}+\int_0^t
 e^{-\frac{t-s}{(2\mu+\lambda)c_1\varepsilon^2}}\frac{1}{c_0\varepsilon^2}\d s\il\mathcal{H}^{I}\il_{\infty,T}\lesssim\|R^I(0)\|_{L^{\infty}(\Omega)}+\il\mathcal{H}^{I}\il_{\infty,T}.
 \end{equation}
  We have  thus reduced the problem to the estimate  of $\il(\mathcal{C}_{R,1}^I,\mathcal{C}_{R,2}^I)\il_{\infty,T}$ and $\il G_R\il_{[\frac{m-1}{2}],\infty,T}.$
By the identities \eqref{convection identity}
\eqref{sec5:eq-16},
and the definition of $\cA_{m,T}, $ we have:
 \begin{equation}\label{comutatorLinfty}
    \il (\mathcal{C}_{R,1}^I,\mathcal{C}_{R,2}^I)\il_{\infty,T}\leq \varepsilon 
    \Lambda(\f{1}{c_0},\cA_{m,T})\mathcal{A}_{m,T}. 
 \end{equation}
 Moreover, 
 $G_R$ (defined in \eqref{eqofR}) can be controlled as:
 \begin{equation*}
  \il G_R\il_{[\frac{m}{2}]-1,\infty,T}\lesssim \varepsilon^{\frac{1}{2}} \mathcal{A}_{m,T}\Lambda(1/c_0,\cA_{m,T})
  +\il(\sigma,u)\il_{[\frac{m+1}{2}],\infty,T}+\varepsilon \il\chi\curl\omega\cdot \bn\il_{[\frac{m-1}{2}],\infty,T}.
 \end{equation*}
 Since $\curl\omega\cdot \bn=\div(\omega\times \bn)+\omega\cdot\curl \bn,$ the  identity \eqref{omegatimesn}  yields
  \begin{equation*}
  \varepsilon\il\chi\curl\omega\cdot \bn\il_{[\frac{m-1}{2}],\infty,T} \lesssim
  \varepsilon \il\nabla u\il_{[\frac{m+1}{2}],\infty,T},
 \end{equation*}
which further leads to:
  \begin{equation}\label{nonlineartermLinfty}
  \il G_R\il_{[\frac{m}{2}]-1,\infty,T}\lesssim 
  \varepsilon^{\frac{1}{2}}\mathcal{A}_{m,T}\Lambda(\f{1}{c_0},\cA_{m,T})+\cE_{m,T}.
 \end{equation}
 Inserting \eqref{comutatorLinfty}-\eqref{nonlineartermLinfty} into \eqref{RILinfty}, we get \eqref{nablasigmaLinfty}.
\end{proof}
\section{Proof of Theorem \ref{thm1}}

Based on the uniform estimates established in previous sections, Theorem \ref{thm1} can be showed
by combining a classical local existence results with  a bootstrap argument.

By following  similar arguments as in \cite{MR2038120} \cite{MR2914244}, one can prove the following local existence result:
\begin{thm}\label{classical-local}
 Assume that $(\sigma_0^{\ep},u_0^{\ep})\in H^2(\Omega),$ and
 $$-\bar{c}\bar{P}\leq \ep \sigma_0^{\ep}(x)\leq {\bar{P}}/\bar{c}, \quad \forall x\in\Omega, \ep\in (0,1].$$
there is some $T_{\ep}>0$ such that 
\eqref{NS1} has a unique strong solution which satisfies:
$(\sigma^{\ep},u^{\ep})\in C([0,T^{\ep}],H^2), u^{\ep}\in L^2([0,T^{\ep}],H^3)$. 
Moreover, the following property holds:
\beq\label{preasumption1}
-3{\bar{c}}\bar{P}\leq \ep\sigma^{\ep}(t,x)\leq 3\bar{P}/\bar{c} \quad \forall (t,x)\in [0,T^{\ep}]\times\Omega. 
\eeq
\end{thm}
By using this result, we can give the proof of Theorem \ref{thm1}.
\subsection*{Proof of Theorem \ref{thm1}:}
%$\textbf{Claim 2.}$ Assume that$(\sigma_{0}^{\ep},u_0^{\ep})\in Y_{m}$ with $\sup_{\ep\leq 1}Y_m(\sigma_0^{\ep},u_0^{\ep})\leq M$If for some time  a uniform time $\tilde{T}$ and a constant $C(1/c_0),$ such that for any $0<\ep\leq \ep_0,$ 
 %\beq\label{enineq1}\cN_{m,T}(\sigma^{\ep},u^{\ep})\leq 2 C(1/c_0) M, \quad \forall T\leq T_0=\colon\min\{\tilde{T},T_{\ep}\} \eeq
% \beq\label{preasumption2} -2{\bar{c}}\bar{P}\leq \ep\sigma^{\ep}(t,x)\leq 2\bar{P}/\bar{c} \quad \forall (t,x)\in [0,T_0]\times\Omega, \forall \ep\in(0,\ep_0]. \eeq
%Once the above two claims are shown, we can get  the local existence in on a  uniform time interval $[0, T_0].$ Indeed, 
%By Claim 2, there exists $\ep_0 \in (0, 1]$ and a uniform $\tilde{T},$ such that\eqref{enineq1} and \eqref{preasumption2} hold for any $\ep \in (0, \ep_0].$ 

On the one hand, $(\sigma_0^{\ep},u_0^{\ep})\in H^2,$ by Theorem \ref{classical-local}, one can find some $T^{\ep}>0$ such that there is a unique solution of \eqref{NS1}  satisfying:
$(\sigma^{\ep}, u^{\ep})\in C([0,T^{\ep}],H^2), u^{\ep}\in L^2([0,T^{\ep}],H^3).$ Moreover, condition \eqref{preasumption1} holds.
On the other hand, as $(\sigma_0^{\ep},u_0^{\ep})\in Y_m,$ 
a higher regularity space, 
by standard propagation of regularity arguments (for example based on applying  finite difference instead of derivatives)
 in the estimates of Section 3 and Section 4, we find that % based on the property \eqref{preasption1}, 
 the estimates of  Proposition \ref{prop-uniform es} hold on $[0, T^\ep]$.
More specifically, we can find 
a constant $C(1/c_0)$ and an increasing polynomial $\Lambda_0$ that are independent of $\ep$ and $T^{\ep},$ such that for any $0\leq T\leq \min\{1, T^{\ep}\}, 0< \ep\leq 1,$
\beq\label{sec6:eq1}
\cN^2_{m,T}(\sigma^{\ep},u^{\ep})\leq C\big(\f{1}{c_0}\big) Y^2_m(\sigma_0^{\ep}, u_0^{\ep})+(T+\ep)^{\f{1}{2}}\Lambda_0\big(\f{1}{c_0}, \cN_{m,T}\big).
\eeq
Moreover, by using the characteristics method, we have  that $\ep\sigma$ can be expressed as,
\beq\label{lagranian}
\ep\sigma^{\ep}(t,x)=\ep\sigma_0^{\ep}(X^{-1}(t,x))%\exp\big(
-\int_0^t (\div u^{\ep}/g_1)(X(s,X^{-1}(t,x)))\,\d s
\eeq
where $X(t,\cdot)$ is the flow associated to $u.$ 

Let us define      
\beqs
T^{\ep}_{*}=\sup\{T\big| (\sigma^{\ep},u^{\ep})\in C([0,T],H^2), u^{\ep}\in L^2([0,T],H^3)\},
\eeqs
\beqs
\begin{aligned}
T_0^{\ep}=\sup\{T\leq\min\{T^{\ep}_{*},1\}&\big| \cN_{m,T}(\sigma^{\ep}, u^{\ep})\leq 2 \sqrt{C(1/c_0)}M, \\
&-2{\bar{c}}\bar{P}\leq \ep\sigma^{\ep}(t,x)\leq 2\bar{P}/\bar{c} \quad \forall (t,x)\in [0,T]\times\Omega \}
\end{aligned}
\eeqs
where $M >\sup_{\ep\in(0,1]}Y_m(\sigma_0^{\ep}, u_0^{\ep}).$

 We now choose successively two  constants $0<\ep_0\leq 1$ and $0<T_0\leq 1$ (uniform in $\ep\in(0,\ep_0]$) which are small enough, such that:
 \beqs
 ({T_0}+\ep_0)^{\f{1}{2}}\Lambda_0\big(1/c_0,2\sqrt{C(1/c_0)}M\big)<{1}/{2},\quad %\exp\big(
 2\sqrt{C(1/c_0)}M{T_0}/c_0\leq \bar{c}\bar{P}.
 \eeqs
In order to prove Theorem \ref{thm1}, it suffices  to show that $T_0^{\ep}\geq {T_0}$ for every  $0<\ep\leq \ep_0.$ Suppose otherwise $T_0^{\ep}<{T_0}$ for some $0<\ep\leq \ep_0,$ then in view of inequalities 
\eqref{sec6:eq1} and formula \eqref{lagranian}, we have by the definition of $\ep_0$ and $T_0$ that:
\beq\label{sec6:eq3}
\cN_{m,T}(\sigma^{\ep},u^{\ep})\leq \sqrt{2C(1/c_0)}M, \qquad \forall T\leq \tilde{T}=\min\{T_0, T^{\ep}_{*}\},
\eeq
\beq\label{sec6:eq4}
-2{\bar{c}}\bar{P}\leq \ep\sigma^{\ep}(t,x)\leq 2\bar{P}/\bar{c} \quad \forall (t,x)\in [0,\tilde{T}]\times\Omega.
\eeq
We will prove that $\tilde{T}=T_0\leq T_{*}^{\ep}.$ This fact, combined with the definition of $T_0^{\ep}$ and estimates \eqref{sec6:eq3}, \eqref{sec6:eq4},
yield $T_0^{\ep}\geq T_0,$ which is a contradiction with the  assumption $T_0^{\ep}<T_0.$ To continue, we shall need the claim stated and proved below.
 Indeed, once the following claim holds, we have by \eqref{sec6:eq3} that 
 $\|(\sigma^{\ep},u^{\ep})(T_0)\|_{H^2(\Omega)}<+\infty.$ Combined with the local existence result stated in 
 Theorem \ref{classical-local}, this  yields that $T_{*}^{\ep}>T_0.$
% \end{proof}

 $\textbf{Claim.}$ 
For all $\ep\in(0,1],$
if $\cN_{m,T}(\sigma^{\ep},u^{\ep})<+\infty,$ then $(\sigma^{\ep},u^{\ep})\in C([0,T], H^2),$
 $u^{\ep}\in L^2([0,T],H^3).$
 
\begin{proof}[Proof of claim]
We see from the definition of $\cN_{m,T}$ and the 
estimate \eqref{threeorderofu} that:
$$\ep u^{\ep}\in %L^{\infty}([0,T],H^2)\cap 
L^2([0,T],H^3),\quad \ep\pt u^{\ep}\in L^2([0,T],H^1), \quad \ep \sigma^{\ep} \in L^{\infty}([0,T],H^2).$$
 one deduces from  interpolation that $\ep u^{\ep}\in C([0,T],H^2).$
Moreover, carrying out direct energy estimates for $\sigma^{\ep}$ in $H^2(\Omega),$
one gets that:
\beq\label{ineq-gronwall}
|\pt R^{\ep}(t)|
\leq  K^{\ep}\big(R^{\ep}(t)+f^{\ep}(t)\big)
\eeq
where  $K^{\ep}=\Lambda(1/c_0, \il(\nabla\sigma^{\ep},\nabla u^{\ep},\ep\nabla^2 u^{\ep})\il_{\infty,t})$ is  uniformly bounded
and $$R^{\ep}(t)=\|\ep\sigma^{\ep}(t)\|_{H^2}^2,\quad f^{\ep}(t)=\|\ep u^{\ep}(t)\|_{H^3}\|\sigma^{\ep}(t)\|_{H^2}\in L^1([0,T]).$$
Inequality \eqref{ineq-gronwall} and the boundedness of $\|R^{\ep}(\cdot)\|_{L^{\infty}([0,T])}$
leads to the fact that $R^{\ep}(\cdot)\in C([0,T]),$ which further yields that $\ep\sigma^{\ep}\in C([0,T],H^2).$ This ends  the proof of the claim.
Note that at this stage we do not require the norm $\|(\sigma^{\ep},u^{\ep})\|_{C([0,T], H^2)}$ to be bounded uniformly in $\ep$.
\end{proof}
\begin{section}{proof of Theorem \ref{convergence}}
 The convergence result follows from compactness arguments.
 At first, since $\si^{\ep}=\frac{P(\rho^{\ep})-P(\bar{\rho})}{\varepsilon}$ is uniformly bounded in $L^{\infty}([0,T_0], W^{1,\infty}(\Omega) )\cap L^2([0,T_0],H^1(\Omega))$, we have that:
$P(\rho^{\ep})\rightarrow P(\bar{\rho})$
in $L^{\infty}([0,T_0], W^{1,\infty}(\Omega) )\cap L^2([0,T_0],H^1(\Omega)),$ 
which yields that $\rho^{\ep}\rightarrow \bar{\rho}$
in $L^2([0,T_0],H^1(\Omega)).$

%We first remark that, up to possible extraction of subsequences, the term $\rho^{\ep}u^{\ep}\cdot\na u^{\ep}\rightarrow \bar{\rho}u^0\cdot\na u^{0}$ in $L_{t,w}^{1}L_{x,loc}^1.$ Indeed, since $(\rho^{\ep}u^{\ep})$ is bounded in $L_t^2H^1$, $\rho^{\ep}u^{\ep}$  tends to (up to subsequences) some function $m^0$ in $L_{t,w}^{2}L_{x}^2$.Notice also $\rho^{\ep} u^{\ep}\rightarrow \bar{\rho}u^0,$ we  have $m^0=\bar{\rho}u^{0}.$Finally, let us write:$\rho^{\ep}u^{\ep}\cdot\na u^{\ep}-\bar{\rho} u^0\na u^0=(\rho^{\ep}u^{\ep}-\bar{\rho}u^0)\na u^{\ep}+\bar{\rho} u^0 (\na u^{\ep}-\na u^0)\rightarrow 0 $ in $L_t^{1,w}L_{x}^1$ since $\na u^{\ep}$ is uniformly bounded in $L_{t,x}^2$ and $\na u^{\ep}\rightarrow \na u^0$ weakly in $L_{t,x}^2.$ ({\color{red}the first term does not convergence since you do not have the compactness in time.  })
%For the convergence of $u^{\ep},$ by the uniform bound of $\|\nabla u^{\ep}\|_{L^2(Q_{T_0})}$, $u^{\ep}$ converges up to possible extraction of the subsequences (say to $u^0$) in $L_{t,w}^2L_x^2.$ 
For the convergence of $u^{\ep},$
let us split the velocity into compressible  part and incompressible part: $u^{\ep}=\nabla\Psi^{\ep}+v^{\ep}$ by using the Leray decomposition \eqref{def-projection}.
we shall prove that the compressible part $\nabla\Psi^{\ep}$ tends to $0$ in $L_{t,w}^2H^1(\Omega)$ whereras incompressible part of $u^{\ep}$ tends to $u^{0}$ in $L^2(Q_{T_0}).$ 
Since $\nabla\Psi^{\ep}$ is uniformly bounded in $L_t^2H^2(\Omega),$ we have that, up to the  extraction of  a subsequence (that we do not mention explicitely) $\nabla\Psi^{\ep}$ converges to $\mathbb{Q}u^0$
in $L_{w}^2([0,T_0],H^{1}(\Omega)).$ 
%which 
Nevertheless, by the equation \eqref{rewriteofsigma}, %$\sigma$: $\div u^{\ep}=g_1(\ep\p_t\si^{\ep}+\ep u^{\ep}\cdot \na \si^{\ep})$, 
$\div u^{\ep}$ tends to $0$ in the sense of distribution, which leads to $\mathbb{Q}u^{0}=0.$  Because of this, one can indeed see that, without any extraction of the subsequences,  
$\nabla\Psi^{\ep}\rightarrow 0$ in 
$L_{w}^2([0,T_0],H^{1}(\Omega)).$ 

We are now in position to prove the convergence of $v^{\ep}.$ %It is direct to see that $v^{\ep}$ tends to $u^0$ in  $L_{t,w}^2L_x^2(\Omega).$ Nevertheless, one can prove  the strong  convergence in $L^2(Q_{T_0})$. Indeed, one gets 
By the equation of $v^{\ep}:$ %$\eqref{eqofv}$
$\eqref{reformulation}_3,$ $\partial_t v^{\ep}$ is uniformly bounded in $L^2([0,T_0],H^{-1}(\Omega))$ whereras $v^{\ep}$ is uniformly bounded in $L^2([0,T_0],H^{1}(\Omega)).$ Therefore, by Aubin-Lions lemma, %$v^{\ep}\in C([0,T_0], L^2(\Omega))${\color{red} what is the interest of this continuity here ?}
  $\{v^{\varepsilon}\}$ is compact in $L^2(Q_{T_0}),$ which yields, up to extraction of subsequences,  the convergence of $v^{\ep}$ (say to $u^0$) in $L^2(Q_{T_0}).$ 
%{\color{red}Since we will prove that $u^0$ is a \textit{unique} solution to \eqref{INS} that admits additional regularity property \eqref{additional-regularity}, this convergence holds indeed for the full sequence----> put it later. }

In the following, we aim to justify that $u^{0}$ is the unique weak solution of the incompressible Navier-Stokes equation \eqref{INS} satisfying \eqref{additional-regularity}.
Let us rewrite the equations of $v^{\ep}$ as follows:
\beq\label{momentum1}
\bar{\rho}\p_t v^{\ep}-
\mu\Delta v^{\ep}+\nabla {\pi}^{\ep}%-\ep\f{\rho^{\ep}-\bar{\rho}}{\ep}\partial_t u^{\ep}-\rho^{\ep}u^{\ep}\cdot\nabla u^{\ep}.
= F^\ep= F_1^{\ep}+F_2^{\ep}.
\eeq
where 
$$ F_1^{\ep}=-(\rho^{\ep}-\bar{\rho})(\partial_t u^{\ep}+u^{\ep}\cdot\nabla u^{\ep}), \quad F_2^{\ep}=-\bar{\rho}(v^{\ep}\cdot \nabla u^{\ep}+\nabla\Psi^{\ep}\cdot\nabla v^{\ep}), $$
Note that we put  the gradient terms
$\bar{\rho}\nabla(\pt\Psi^{\ep}+\f{1}{2}|\nabla\Psi^{\ep}|^2)$
into the pressure $\nabla \pi^{\ep}.$ 
%Although not necessary for the justification of the limit process,  it is still interesting to notice that 
%%The advantage of doing so is that 
%$\nabla \pi^{\ep}$ is uniformly bounded in $L^2(Q_{T_0}).$ 
%Indeed, by the equation \eqref{momentum1}, $\pi^{\ep}$ can be defined as the solution of the following Neumann problem:
%\begin{equation}\label{eqofpi}
%\left\{
% \begin{array}{l}
%    \displaystyle \Delta \pi^{\ep}=- \div (F_1^{\ep}+F_2^{\ep})
%   \quad \text{ in} \quad \Omega.\\
% \partial_{\bn} \pi^{\ep}=-(F_1^{\ep}+F_2^{\ep})\cdot \bn+\mu \Delta v^{\ep}\cdot \bn \quad \text{on} \quad \partial\Omega.
% \end{array}   
% \right.
%\end{equation}
%The uniform boundedness of $\nabla\pi^{\ep}$
%in $L^2(Q_{T_0})$
% thus follows from elliptic regularity theory (see the arguments in Lemma \ref{lemofq}).
Let us write down the weak formulation 
for  \eqref{momentum1}.
Multiplying equation \eqref{momentum1} by a test function $\psi\in (C^{\infty}([0, T_{0}] \times \overline{\Omega}))^3$ which satisfies  $\div \psi =0, \psi\cdot\bn|_{\p\Omega}=0,$ we obtain that for each $0< t\leq T_0,$
 \beq\label{com-EI}
 \begin{aligned}
 &\bar{\rho}\int_{\Omega}(v^{\ep}\cdot\psi)(t,\cdot)\ \d x+\mu\iint_{Q_{t}}\nabla v^{\ep}\cdot\nabla\psi\ \d x\d s+\iint_{Q_t}F^{\ep}\cdot\psi \ \d x\d s\\
    &=\bar{\rho}\int_{\Omega}(v^{\ep}\cdot\psi)(0,\cdot)\, \d x+\bar{\rho}\iint_{Q_t}v^{\ep}\cdot\pt\psi\ \d x\d s+\mu\int_0^t\int_{\p\Omega}\Pi
    \p_{\bn}v^{\ep}\cdot\psi\ \d S_y \d s.
 \end{aligned}
 \eeq
It remains  to pass to the limit to show that $u^0$ satisfies \eqref{incom-EI}. We shall only focus on the last terms in both sides of
\eqref{com-EI}, as the other terms are direct.
Since $\rho^\ep = g_{2}(\ep \sigma^\ep)$, we have that  $ (\rho^{\ep}-\bar{\rho})/\ep$ is  uniformly bounded in $L^{\infty}(Q_{T_0}),$ 
it then follows from the velocity  equation in 
 $\eqref{NS1}_2$ that 
% integration by parts in space
%and the definition 
%\eqref{defofg12}: $\rho^{\ep}=g_2(\ep\sigma^{\ep})$ %\f{g_2'(\ep\sigma^{\ep})}{\rho^{\ep}}=g_1(\ep\sigma^{\ep})$
%that: %$\iint_{Q_t}F_1^{\ep}\cdot\psi \d x\d s\rightarrow 0.$
%{\color{red} you also have to clearly write  somewhere that $\rho^\ep  - \overline{\rho}$  is $O(\ep)$ I added it OK with an explanation it would be better, Ok, I gave more explanation below}
% %by integrating by parts in time, we find that:
$$
%\iint_{Q_t}F^{\ep}\cdot\psi \d x\d s=\ep\iint_{Q_t}\f{\rho^{\ep}-\bar{\rho}}{\ep}\pt\psi\cdot u^{\ep}\d x\d s-\ep\int_{\Omega}\f{\rho^{\ep}-\bar{\rho}}{\ep}u^{\ep}\cdot\psi\d x\big|_{0}^t\\-\iint_{Q_t}(-\pt\rho^{\ep}u^{\ep}+\rho^{\ep}u^{\ep}\cdot\nabla u^{\ep})\cdot \psi\d x\d s
\iint_{Q_t}F_1^{\ep}\cdot\psi\, \d x\d s
=\iint_{Q_t}\f{\rho^{\ep}-\bar{\rho}}{\rho^{\ep}}(\div \mathcal{L} u^{\ep}-\f{\nabla\sigma^{\ep}}{\ep})\psi
\, \d x \d s
.$$
We then  observe that  
$${ 1 \over \ep} \iint_{Q_{t}} \frac{ \rho^\ep- \overline{\rho}}{\rho^\ep} \nabla \sigma^\ep \cdot \psi \, \d x\d s
= \frac{1}{\ep} \iint_{Q_{t}} \frac{g_2(\ep \sigma^\ep)- g_{2}(0)}{g_{2}(\ep \sigma^\ep)} \nabla \sigma^\ep \cdot \psi \,\d x\d s= 0$$
by integrating by parts since 
$$   {g_{2}(\ep \sigma^\ep)- g_{2}(0) \over g_{2}(\ep\sigma^\ep)} \nabla \sigma^\ep= {1 \over \ep} \nabla \left(G(\ep \sigma^\ep)\right)$$
 where $G(s)$
 is such that
$$ G'(s)= {g_{2}(s) - g_{2}(0) \over g_{2}(s)}.$$
 In a similar way, we have that
 $$  \iint_{Q_t}\f{\rho^{\ep}-\bar{\rho}}{\rho^{\ep}}\nabla \div  u^{\ep} \cdot \psi  \, \d x \d s = 
  -  \ep \iint_{Q_t} \div u^\ep \, G''(\ep \sigma^\ep) \nabla \sigma^\ep\cdot \psi \, \d x\d s$$
   $$  \iint_{Q_t}\f{\rho^{\ep}-\bar{\rho}}{\rho^{\ep}}\Delta u^\ep \cdot \psi \, \d x \d s = 
    -  \ep \iint_{Q_t} G''(\ep \sigma^\ep) \left( (\nabla \sigma^\ep \cdot \nabla) u^\ep \right)\cdot \psi \, \d x\d s
   + \int_{0}^t \int_{\partial \Omega}  { \rho^\ep - \bar{\rho} \over \rho^\ep} \Pi \partial_{n}u^\ep \cdot \psi \, \d x \d s.$$
  These three above   terms  tend to zero, for the last  one, we use that $ \|\rho^\ep - \bar{\rho}\|_{L^\infty(Q_{t})}= O(\ep)$
   while $\Pi\partial_{n}u^\ep$ is uniformly bounded in $L^2(\partial \Omega)$ by using the Navier-boundary condition
   and the trace inequality. This yields
   $$ \iint_{Q_t}F_1^{\ep}\cdot\psi \d x\d s \rightarrow 0.$$
%
%=-\ep\iint_{Q_t} \nabla u^{\ep}\cdot\nabla \big[\f{\rho^{\ep}-\bar{\rho}}{\ep\rho^{\ep}}\psi\big]\d x\d s-\iint_{Q_t} %\f{\rho^{\ep}}{g_2'(\ep\sigma^{\ep})}
%\f{1}{(g_2^2)'(\ep\sigma^{\ep})}\psi\cdot \nabla\big|\f{\rho^{\ep}-\bar{\rho}}{\ep}\big|^2 \d x\d s\\
%=-\ep\iint_{Q_t} \nabla u^{\ep}\cdot\nabla \big[\f{\rho^{\ep}-\bar{\rho}}{\ep\rho^{\ep}}\psi\big]\d x\d s+\ep\iint_{Q_t}
%\big|\f{\rho^{\ep}-\bar{\rho}}{\ep}\big|^2 ({1}/(g_2^2)')'(\ep\sigma^{\ep})\nabla\sigma^{\ep}\cdot\psi\d x\d s\rightarrow 0.
%\end{multline*}
%The convergence to 0 of first two terms in the right hand side of the above identity is easy to see. As for the last term, we have by equation for density and integrating by parts in space that, 
%\begin{multline*}-\iint_{Q_t}(-\pt\rho^{\ep}u^{\ep}+\rho^{\ep}u^{\ep}\cdot\nabla u^{\ep})\cdot \psi\d x\d s=\iint_{Q_t}\rho^{\ep}(u^{\ep}\cdot\nabla\Psi)\cdot u^{\ep}\d x\d s\\=\ep\iint_{Q_t}\f{\rho^{\ep}-\bar{\rho}}{\ep}(u^{\ep}\cdot\nabla\Psi)\cdot u^{\ep}\d x\d s-\bar\rho\iint_{Q_t} \div u^{\ep} (u^{\ep}\cdot\psi)+ (u^{\ep}\cdot\nabla u^{\ep})\cdot\psi\d x\d s\end{multline*}
Next, since $\nabla\Psi^{\ep}\rightharpoonup 0, \nabla u^{\ep} \rightharpoonup  \nabla u^0, v^{\ep}\rightarrow u^0$ in
$L^2(Q_t)$ and $v^{\ep}$ is uniformly bounded in $L^2([0,T_0],H^1(\Omega)),$ we have that:
\beqs
\iint_{Q_t} 
F_2^{\ep}\cdot \psi\, \d x\d s\rightarrow \bar{\rho}\iint_{Q_t} (u^{0}\cdot \nabla u^{0})\cdot\psi\, \d x\d s.
\eeqs
Finally, for the boundary term in \eqref{com-EI}, we use the  boundary condition for $v^{\ep}$ (see \eqref{bdrycon-v}):
\beqs
\Pi(\partial_{\bn} v^{\ep})=\Pi (-2a v^{\ep}+(D \bn) v^{\ep})+2\Pi(-a \nabla\Psi^{\ep} +(D \bn) \nabla\Psi^{\ep}).
\eeqs
%By the boundedness of $v^{\ep}$ in $L^2([0,T_0],H^1(\Omega))$ and of $\nabla\Psi^{\ep}$ in $L^2([0,T_0],H^2(\Omega))$
As $v^{\ep}\rightarrow u^0$ in $L^2(Q_{t})$ and  $v^{\ep}$ is uniformly bounded in $L^2([0,t],H^1(\Omega)),$ 
$\nabla\Psi^{\ep}\rightarrow 0$ in $L_{w}^2([0,t], H^1(\Omega)),$
it follows from the  trace inequality and the  H\"older inequality that:
$v^{\ep}|_{\p\Omega}\rightarrow u^{0}|_{\p\Omega}$ in $L^2([0,t], L^2(\p\Omega)),$ $\nabla\Psi^{\ep}\rightarrow 0$ in
$L_{w}^2([0,t], L^2(\p\Omega)).$
This yields:
%as well as  the , one sees that $v^{\ep}, \nabla\Psi^{\ep}|_{\p\Omega}$ are uniformly bounded in $L^2([0,T_0],H^{\f{1}{2}}(\p\Omega))$ and $L^2([0,T_0],H^{\f{3}{2}}(\p\Omega)).$
%{\color{red} why do you need this regularity ? why do you need to separate between $v$ and $\nabla \Psi$
% here ?  you seem to use  that the acoustic part also  converges to zero on the boundary. I modified them}
\beqs
\mu\int_0^t\int_{\p\Omega}\Pi\p_{\bn}v^{\ep}\cdot\psi\,\d S_y\d s\rightarrow \mu \int_0^t\int_{\p\Omega}\Pi(-2a u^{0}+(D \bn) u^{0})\cdot\psi\, \d S_y\d s.
\eeqs
Therefore, $u^0$ satisfies the  formulation \eqref{incom-EI} and hence is a weak solution to \eqref{INS}.
Next, due to the uniform boundedness of $v^{\ep}$ in $L_{T_0}^{\infty}H_{co}^{m-1}$ and $\nabla v^{\ep}$ in $L_{T_0}^{2}H_{co}^{m-1}\cap L^{\infty}(Q_{T_0}),$  we get that $u^{0}$
has the  additional regularity property \eqref{additional-regularity}. The uniqueness result is easy owing to the boundedness of the Lipschitz norm.
 Since any subsequence of $u^\ep$ will have an extracted subsequence that solves   \eqref{incom-EI}
  and satisfies the additional regularity property  \eqref{additional-regularity}, we finally get from the uniqueness
   that the whole family $u^\ep$ converges to $u^0$.
This ends  the proof of Theorem \ref{convergence}.

%is now completed.
%Since $F_1^{\ep}$ converges to 0 in $\mathcal{D}'(Q_{T_0}),$ we are now left to show that $F_2^{\ep}$ converges weakly to $-u^0\cdot\nabla u^0.$ On the one hand, $L^2(Q_{T_0}),$ it is easy to show that
%$\rho^{\ep}\nabla\Psi^{\ep}\cdot\nabla v^{\ep}$ tends to 0 in $\mathcal{D}'(Q_{T_0}).$ On the other hand, since $\nabla u^{\ep}$ converges weakly in $L^2(Q_{T_0})$ to $\nabla u^{0},$ $\rho^{\ep}v^{\ep}$ converges strongly to $u^0$ in $L^2(Q_{T_0}),$ one achieves that: $\rho^{\ep}v^{\ep}\cdot\nabla u^{\ep}$tends to $u^0\cdot\nabla u^0$ in $\mathcal{D}'(Q_{T_0}).$It is now just the matter to prove that 
%We thus  that $u^0\in C(0,T;H_{co}^{m-1})$ is a weak solution of \eqref{INS}.
%By the above arguments, we find that $u^{0}$ satisfies the incompressible Navier-Stokes equation \eqref{INS}.
 \end{section}
 
 \begin{section}{Appendix}
  We state here the product  and commutator estimates 
  which are used throughout the paper:
%Let $Q_t=[0,t]\times \Omega$ where $\Omega$ is a smooth bounded domain, 
\begin{lem}
 For each $0\leq t\leq T,$ and for any integer $k\geq 2,$ one has the (rough) product estimates
  \beq\label{roughproduct1}
 \|(fg)(t)\|_{H_{co}^k}\lesssim \|f(t)\|_{H_{co}^k}\il g \il_{[\f{k-1}{2}],\infty,t}+\|g(t)\|_{H_{co}^k}\il f \il_{[\f{k}{2}],\infty,t},
 \eeq
 %\beq%\label{roughproduct} \|(fg)(t)\|_{H_{co}^k}\lesssim \|f(t)\|_{H_{co}^k}\il g \il_{[\f{k}{2}]-1,\infty,t}+\|g(t)\|_{H_{co}^{k-1}}\il f \il_{[\f{k+1}{2}],\infty,t}+\|g(t)\|_{H_{co}^k}\il f\il_{\infty,t}. \eeq
 and commutator estimates:
 \beq\label{roughcom}
 \|[Z^{I},f]g(t)\|_{L^2}\lesssim 
 \|Z f(t)\|_{H_{co}^{k-1}}
 \il g\il_{[\f{k}{2}]-1,\infty,t}+\|g(t)\|_{H_{co}^{k-1}}\il Z f\il_{[\f{k-1}{2}],\infty,t}, \quad\, |I|=k,
 \eeq
  \beq\label{roughcom-T}
 \|[(\ep\pt)^k,f]g(t)\|_{L^2}\lesssim
 \| (\ep\pt f)(t)\|_{\cH^{k-1}}
 \il g\il_{[\f{k}{2}]-1,\infty,t}+\|g(t)\|_{\cH^{k-1}}\il \ep\pt f\il_{[\f{k-1}{2}],\infty,t}.
 \eeq
 %Moreover, the above inequalities hold true %by substituting when $H_{co}^k$ is changed into $\cH^k$ and $Z^I$  is changed into $(\ep\pt)^k.$
 \end{lem}
 \begin{proof}
 This lemma follows from  simply  counting the derivatives hitting on $f$ or $g.$ For instance, to prove the product estimate \eqref{roughproduct1} and the commutator estimate \eqref{roughcom},
 one can use the following expansion:
 %For $|I|=k,$ one has:
 \beqs
 \begin{aligned}
 Z^I (fg)&=\big(\sum_{|J|\leq [(k-1)/2]}+\sum_{|I-J|\leq [k/2]}\big)(C_{I,J}Z^J g Z^{I-J}f)\\
  &=\big(\sum_{|J|\leq [k/2]-1}+\sum_{1\leq |I-J|\leq [(k+1)/2]}\big)(C_{I,J}Z^J g Z^{I-J}f)+f Z^I g, \qquad |I|=k.
 \end{aligned}
 \eeqs
 \end{proof}

 As a corollary of  Lemma 7.1 
 the following composition estimates hold:
 \begin{cor}
 Suppose that $h\in C^0(Q_t)\cap L_t^2H_{co}^m$ with 
 $$A_1\leq  h(t,x)\leq A_2, \quad \forall(t,x)\in Q_t.$$
 Let
 $F(\cdot):[A_1,A_2]\rightarrow \mathbb{R}$ be a smooth function satisfying 
 $$\sup_{s\in [A_1,A_2]}|F^{(m)}|(s)\leq B.$$
 Then we have the composition estimate, for $p=2,+\infty$
 \beqs
 \|F(h(\cdot,\cdot))-F(0)\|_{L_t^pH_{co}^m}\leq  \Lambda(B,\il h\il_{[\f{m}{2}],\infty,t})\|h\|_{L_t^pH_{co}^m},
 \eeqs
 where $\Lambda(B,\il h\il_{[\f{m}{2}],\infty,t})$ is a polynomial with respect to
 $B$ and $\il h\il_{[\f{m}{2}],\infty,t}.$ 
 \end{cor}

 This Corollary, combined with
 Lemma 6.1 and Lemma 6.3,
 leads to the following estimates:
 \begin{cor}
 Let $g_1(\ep\sigma),g_2(\ep\sigma)$ defined in \eqref{defofg12} and assume that   \eqref{preasption}, \eqref{preasption1} hold.
 Then one has the following 
 estimates: for $j=1,2,\, p=2,+\infty,$
 \beq\label{esofg12}
 \|Z g_j\|_{L_t^p\cH^{m-1}}\leq \ep \Lambda\big(\f{1}{c_0},\il \sigma\il_{[\f{m}{2}],\infty,t}\big)\|(\sigma,Z\sigma)\|_{L_t^p\cH^{m-1}}, 
 \eeq
 \beq\label{esofg12-2}
 \|Z g_j\|_{L_t^pH_{co}^{m-1}}\leq \ep \Lambda\big(\f{1}{c_0},\il \sigma\il_{[\f{m}{2}],\infty,t}\big)\|\sigma\|_{L_t^pH_{co}^m},
 \eeq
 \beq\label{esofg12-1}
 \|g_j(\ep\sigma)-g_j(0)\|_{L_t^pH_{co}^m}\lesssim \ep \Lambda\big (\f{1}{c_0}, \il \sigma\il_{[\f{m}{2}],\infty,t}\big)\|\sigma\|
 _{L_t^pH_{co}^m}.
 \eeq
 \end{cor}

 We will use often the following Sobolev embedding inequality whose proof is similar to that of Proposition 12  and Proposition 20 of \cite{MR2885569}. %and trace theorem.
 \begin{prop}\label{propembeddding}
Let  $\Omega=\mathbb{R}_{+}^3$ or a  smooth  bounded domain, we have the following Sobolev embedding inequality 
\begin{equation}\label{sobebd}
 \| f (t)\|_{L^{\infty}(\Omega)}\lesssim 
 \|\nabla f(t)\|_{H_{co}^{k+1}}^{\frac{1}{2}}\|f(t)\|_{H_{co}^{k+2}}^{\frac{1}{2}}+\|f(t)\|_{H_{co}^{k+2}}.
\end{equation}
 \end{prop}
 \begin{proof}
 For  the case of the  half-space, this is a consequence of  the inequality:
 for a function $g$ defined on $\mathbb{R}_{+}^3,$ 
 \begin{equation}
    \|f(t)\|_{L^{\infty}(\mathbb{R}_{+}^3)}\lesssim \|\partial_z f(t)\|_{H_{co}^{s_1}(\mathbb{R}_{+}^3)}^{\frac{1}{2}}\|f(t)\|_{H_{co}^{s_2}(\mathbb{R}_{+}^3)}^{\frac{1}{2}}
 \end{equation}
 where  $s_1,s_2$  are positive and satisfy  $s_1+s_2>2.$
 One can refer to (Prop 2.2) of \cite{MR3590375} for the proof. The case of  general smooth  bounded domains follows  by  working in  local coordinates. 
  \end{proof}
  The following trace inequalities are also used:
  \begin{lem}
  For multi-index $I=(I_0,\cdots, I_{M})$ with
$|I|=k,$ we have the following trace inequalities:
  \beq\label{traceLinfty}
  |Z^{I}f(t)|_{L^2(\p\Omega)}^2 \lesssim \|\nabla f(t)\|_{H_{co}^{k}}\|f(t)\|_{H_{co}^{k}}+\|f(t)\|_{H_{co}^{k}}^2.
  \eeq
  \beq\label{traceL2}
  \int_0^t |Z^{I}f(s)|_{L^2(\p\Omega)}^2\,\d s\lesssim \|\nabla f\|_{L_t^2H_{co}^{k}}\|f\|_{L_t^2H_{co}^{k}}+\|f\|_{L_t^2H_{co}^{k}}^2.
  \eeq
  \beq\label{normaltraceineq}
  \int_0^t |Z^{I}f(s)|_{H^{\f{1}{2}}(\p\Omega)}^2\d s\lesssim \|\nabla f\|_{L_t^2H_{co}^{k}}^2+\|f\|_{L_t^2H_{co}^{k}}^2.
  \eeq
  \end{lem}
  In the next proposition, we state some elliptic estimates which are used frequently.
  \begin{prop}\label{propneumann}
 Given a bounded domain $\Omega$ with $C^{k+1}$ boundary.
Consider the following elliptic equation with Neumann boundary condition:
\begin{equation}\label{Neumann problem}
  \left\{
  \begin{array}{l}
  \Delta q=\div f \quad \text{in}\quad  \Omega\\
  \partial_{\bn} q=f\cdot \bn+g \quad \text{on} \quad  \partial\Omega\\
  \int_{\Omega} q \d x=0
  \end{array}
  \right.
\end{equation}
The system  \eqref{Neumann problem} has  a unique solution in $H^1(\Omega)$ which satisfies the following gradient estimate:
 \begin{equation}\label{gradient estimate}
     \|\nabla q(t)\|_{L^2(\Omega)}\lesssim \|f(t)\|_{L^2(\Omega)}
     +|g(t)|_{H^{-\frac{1}{2}}(\partial\Omega)}.
 \end{equation}
Moreover, for $j+l=k,$
\begin{equation}\label{highconormal}
     \|\nabla q(t)\|_{\cH^{j,l}(\Omega)}\lesssim \|f(t)\|_{\cH^{j,l}(\Omega)}
     +|g(t)|_{\tilde{H}^{k-\frac{1}{2}}(\partial\Omega)}.
 \end{equation}
 \begin{equation}\label{secderelliptic}
 \|\nabla^2 q(t)\|_{\cH^{j,l}(\Omega)}\lesssim \|(f(t),\div f(t))\|_{\cH^{j,l}(\Omega)}%+\|f(t)\|_{H_{co}^{k-1}(\Omega)}
     +|g(t)|_{\tilde{H}^{k-\frac{1}{2}}(\partial\Omega)}.
 \end{equation}
\end{prop}

\begin{proof}
The existence of the  weak solution in 
$H=\colon\{q|\ q\in H^1(\Omega), \int_\Omega q\d x=0 \}$ 
as well as the gradient estimate \eqref{gradient estimate}  come from  Lax-Milgram Lemma. 
The estimates  \eqref{highconormal}-\eqref{secderelliptic} are then standard regularity estimates for elliptic equations,
that take into account the number of time derivatives (the time variable being only a parameter in this Lemma). 
%That is, on one hand,for the interior estimate, one multiplies the equation by a cut-off function and reduce the problem to the whole space case without boundary; on the other hand, for the boundary estimate, one uses the local coordinates and reduces the problem to the half domain case, for which one first controls the tangential derivatives by taking tangential derivation on the equation and then uses the equation itself to recover the conormal derivatives. 
\end{proof}
Finally, we state an elementary estimate of the heat kernel  which is useful in the estimates of the vorticity.
 \begin{lem}
 Let 
 $$K(s,y,z)=\tilde{\mu}|\bN|^2(4\pi\tilde{\mu}|\bN|^2 s)^{-\f{1}{2}}\p_z\big(e^{-\f{z^2}{4 \tilde{\mu}|\bN|^2 s}}\big),
 \quad \bN(y)=(-\p_1\varphi(y),-\p_2\varphi(y),1)^{t}  $$ where %$l\in \mathbb{N},
 $(y,z)\in \mathbb{R}_{+}^3$ and  
 set $\cZ^{\beta}=\p_{y^1}^{\beta_1}\p_{y^2}^{\beta_2}\cZ_3^{\beta_3}, \cZ_3=\f{z}{1+z}\p_z.$
 We have the following estimate:
 \beq\label{L2 norm of K}
 \|\cZ^{\beta}K(s,y,\cdot)\|_{L_z^2(\mathbb{R}_{+})}\leq C(\beta,\tilde{\mu}, |\varphi|_{C^{|\beta|+1}}) s^{-\f{3}{4}}.
 \eeq
 \end{lem}
 \begin{proof}
It suffices to prove that,
 for any $l\in \mathbb{N}$, there is a polynomial ${P}_{2|\beta|+1}$ with $2|\beta|+1$ degree, such that:
 \beq\label{induction on K}
 |\cZ^{\beta}K(s,y,z)|\leq C(\beta,\tilde{\mu}, |\varphi|_{C^{|\beta|+1}})  {P}_{2|\beta|+1}\big(\f{z}{\sqrt{s}}\big)e^{-\f{z^2}{4\tilde{\mu} |\bN|^2 s}}s^{-1}\quad \forall s>0, y\in\mathbb{R}^2.
 \eeq
 By direct computation, one can see that, there exists a polynomial
 with degree $2(\beta_1+\beta_2)+1:$
 $P_{2(\beta_1+\beta_2)+1},$ a smooth function depends on $\nabla_{y}\varphi$
 and its derivatives up to order $\beta_1+\beta_2: F_{\beta_1+\beta_2}(\nabla_y\varphi)$  such that 
 \beqs
 \p_{y_1}^{\beta_1}\p_{y_2}^{\beta_2}K(s,y,z)=P_{2(\beta_1+\beta_2)+1}\big(\f{z}{\sqrt{s}}\big)F_{\beta_1+\beta_2}(\nabla_y
 \varphi) e^{-\f{z^2}{4\tilde{\mu} |\bN|^2 s}}s^{-1}.
 \eeqs
To prove \eqref{induction on K}, 
it suffices to show by induction arguments that, there exists a smooth    function $F(|\bN|^2),$ such that 
\beqs
\p_z^{\beta_3}\bigg(P_{2(\beta_1+\beta_2)+1}\big(\f{z}{\sqrt{s}}\big)e^{-\f{z^2}{4\tilde{\mu}|\bN|^2 s}}\bigg)=F
(|\bN|^2) e^{-\f{z^2}{4\tilde{\mu}|\bN| s}}P_{2|\beta|+1}\big(\f{z}{\sqrt{s}}\big)z^{-{\beta_3}}.
\eeqs

\end{proof}
\end{section}
\bibliographystyle{plain}
\nocite{*}
\bibliography{referencein}
 \end{document}